\documentclass[11pt,reqno]{amsart}

\usepackage{amsmath,amsthm,amsfonts,amssymb,young}
\usepackage{array, boldline, makecell, booktabs}
\usepackage{bm}
\usepackage{ytableau}
\usepackage{tikz}
\usepackage{float}
\usetikzlibrary{backgrounds}
\usepackage{mathtools}
\usepackage[colorlinks]{hyperref}
\usepackage{cleveref}
\numberwithin{equation}{section}    
\usepackage[all]{xy}
\usepackage{graphicx,import}
\usepackage{amscd}
\usepackage{color}
\usepackage{enumerate}
\usepackage{cases}
\usepackage{cite}
\usepackage[T1]{fontenc}

\AtBeginDocument{%
  \setlength{\abovedisplayskip}{3pt plus 1pt minus 1pt}%
  \setlength{\belowdisplayskip}{3pt plus 1pt minus 1pt}%
  \setlength{\abovedisplayshortskip}{0pt plus 1pt}%
  \setlength{\belowdisplayshortskip}{3pt plus 1pt minus 1pt}%
}
\newtheorem{thm}{Theorem}[section]
\newtheorem{lem}[thm]{Lemma}
\newtheorem{prop}[thm]{Proposition}
\newtheorem{cor}[thm]{Corollary}
\theoremstyle{definition}
\newtheorem{exam}[thm]{Example}
\crefname{exam}{example}{examples}
\Crefname{exam}{Example}{Examples}
\newtheorem{defn}[thm]{Definition}

\newtheorem{remark}[thm]{Remark}

\newcommand\EH{\operatorname{EH}}
\newcommand\des{\operatorname{des}}
\newcommand\asc{\operatorname{asc}}
\newcommand\rev{\operatorname{rev}}
\newcommand\EWord{\mathsf{EWord}}
\newcommand\Word{\mathsf{Word}}
\newcommand\MCSel{\operatorname{MCSel}}
\newcommand\val{\operatorname{val}}
\newcommand\ESort{\operatorname{ESort}}
\newcommand\xt{(\vec{x};\vec{t})}
\newcommand\FL{\mathfrak{L}}
\newcommand\kappacan{\kappa^{\mathrm{can}}}
\newcommand\CSel{\operatorname{CSel}}
\newcommand\ICSel{\operatorname{ICSel}}
\newcommand\col{\operatorname{col}}
\newcommand\rad{\xrightarrow{\mathrm{d}}}
\newcommand\rav{\xrightarrow{\mathrm{v}}}
\newcommand\trl{\triangleleft}
\newcommand\wt{\operatorname{wt}}

\newcommand\supp{\mathsf{supp}}
\newcommand\Codes{\mathsf{Codes}}

\newcommand\IN{\mathsf{IN}}
\newcommand\For{\mathsf{For}}
\newcommand\Red{\operatorname{Red}}
\newcommand\Syl{\operatorname{Syl}}
\newcommand\ERed{\operatorname{ERed}}

\newcommand{\CC}{\mathbb{C}}
\newcommand{\ZZ}{\mathbb{Z}}
\newcommand{\RR}{\mathbb{R}}
\newcommand\Path{\mathsf{Path}}
\newcommand\PP{\mathfrak{P}}

\let\oldemph\emph
\renewcommand{\emph}[1]{\textcolor{blue}{\oldemph{#1}}}

\renewcommand\vec[1]{\mathbf{#1}}

\title[A lattice path model for Schubert and forest polynomials]{A lattice path model for Schubert and forest polynomials: positivity and combinatorial interpretations}
\author{Hojoon Lee}
\address{Department of Mathematics, Sungkyunkwan University, Suwon, South Korea}
\email{hojoon1101@g.skku.edu}
\subjclass[2020]{Primary - 05E05, 05E14}

\begin{document}

\maketitle
\begin{abstract}
In this paper, we give combinatorial interpretations of (i) the Graham-positive
coefficients in the expansion of double Schur polynomials in the double
fundamental quasisymmetric basis and (ii) the coefficients in the expansion
of Grothendieck polynomials in the grove basis.
To obtain these interpretations, we construct a common lattice path model
for Schubert, Grothendieck, and double Schubert polynomials and their
respective quasisymmetric counterparts: forest, grove, and double forest
polynomials.
The first interpretation answers a question of Bergeron, Gagnon, Nadeau,
Spink, and Tewari asking for an explicit combinatorial interpretation of
the coefficients. It also yields formulas
for the equivariant degrees of the quasisymmetric Schubert cycles
introduced by Bergeron, Gagnon, Spink, and Tewari.
The specialization of the second interpretation recovers
Nadeau and Tewari's Schubert-to-forest formula, expressed using the
rightmost representatives of their Omega-equivalence classes.
\end{abstract}

\section{Introduction}

Schubert calculus connects the geometry of flag varieties with the
combinatorics of symmetric function theory. A polynomial in
\(x_1,\ldots,x_n\) is \emph{symmetric} if it is invariant under all
permutations of the variables. The \emph{complete flag variety}
\(\mathrm{Fl}_n\) is the space of chains of vector subspaces of \( \CC^n \) of the form
\[
  0=F_0\subset F_1\subset\cdots\subset F_n=\CC^n,
  \qquad \dim F_i=i.
\]
The Bruhat decomposition of \(\mathrm{Fl}_n\) gives Schubert classes
indexed by the symmetric group \(S_n\), and Borel's presentation \cite{Borel1953} identifies its
cohomology ring with the coinvariant algebra
\[
  H^\bullet(\mathrm{Fl}_n)
  \cong\mathbb Z[x_1,\ldots,x_n]/\operatorname{Sym}_n^+,
\]
where the ideal \( \operatorname{Sym}_n^+ \) is generated by symmetric polynomials with zero constant
term. This ring is a free \(\mathbb Z\)-module of rank \(n!\).
Schubert polynomials give distinguished representatives of its Schubert
classes \cite{LS1982}. Grothendieck polynomials \cite{LS1982Hopf} and
double Schubert polynomials \cite{Lascoux1982Chern} extend this picture
to \(K\)-theory and equivariant cohomology, respectively.

For \(1\le r<n\), retaining only the subspace \(F_r\) from each complete
flag gives the \emph{Grassmannian} \(\operatorname{Gr}(r;n)\), the
variety of all \(r\)-dimensional vector subspaces of \(\mathbb C^n\).
The natural projection
\(\pi_r:\mathrm{Fl}_n\to\operatorname{Gr}(r;n)\) sends a flag
\((F_i)_{i=0}^n\) to \(F_r\). Schur polynomials represent the
Schubert classes of \(\operatorname{Gr}(r;n)\), and double Schur
polynomials represent the corresponding \(T\)-equivariant classes,
where \(T=(\mathbb C^\times)^n\) is the diagonal torus
\cite{LS1982,KnutsonTao2003}. In both cases, the indices are partitions
contained in an \(r\times(n-r)\) rectangle.

Quasisymmetric polynomials, introduced by Gessel \cite{Gessel1984},
form a ring containing the symmetric polynomials. A polynomial in
\(x_1,\ldots,x_n\) is \emph{quasisymmetric} if, for every sequence of
positive integers \(a_1,\ldots,a_k\), the coefficient of
\(x_{i_1}^{a_1}\cdots x_{i_k}^{a_k}\) is independent of the choice of
\(1\le i_1<\cdots<i_k\le n\).
A quasisymmetric counterpart to the complete flag variety is the
quasisymmetric flag variety of Bergeron, Gagnon, Nadeau, Spink,
and Tewari \cite{BGNST2026QFV}. This is a toric complex
\(\mathrm{QFl}_n\subseteq\mathrm{Fl}_n\) whose torus fixed points
are indexed by noncrossing partitions and whose cohomology ring is
\[
  H^\bullet(\mathrm{QFl}_n)
  \cong \mathbb Z[x_1,\ldots,x_n]/\operatorname{QSym}_n^+.
\]
Here the ideal \( \operatorname{QSym}_n^+ \) is generated by quasisymmetric polynomials with zero
constant term, and the rank is the \(n\)-th Catalan number
\cite{ABB2004}. Forest polynomials, introduced by Nadeau and Tewari
\cite{NT2024}, represent a basis of
\(H^\bullet(\mathrm{QFl}_n)\) that is Kronecker dual to the
fundamental homology classes of quasisymmetric Schubert cell closures
\cite[Theorem~12.12]{BGNST2026QFV}. Their \(K\)-theoretic and equivariant
analogues are grove polynomials \cite{NST2026Grove} and double forest
polynomials \cite{BGNST2025a}, respectively.

The restriction of \(\pi_r\) to \(\mathrm{QFl}_n\) has image
\(\operatorname{QGr}(r;n)\), called the \emph{quasisymmetric
Grassmannian} \cite{BGST2026QG}.
The relevant subfamilies of forest and double forest polynomials are
fundamental and double fundamental quasisymmetric polynomials,
respectively \cite{BGNST2025a}. Indexed by compositions whose ribbons
fit in the \(r\times(n-r)\) rectangle, these polynomials represent the
ordinary and \(T\)-equivariant cohomology bases Kronecker dual to the
corresponding homology classes of quasisymmetric Schubert cell closures
\cite{BGST2026QG}.

These families give the following correspondences for flag varieties
and Grassmannians:
\begin{center}
\small
\setlength{\arraycolsep}{3pt}
\renewcommand{\arraystretch}{1.15}
\begin{minipage}[t]{.51\textwidth}
\centering
\(
\begin{array}{@{}c|ccc@{}}
 & H^\bullet & K^\bullet & H_T^\bullet\\ \hline
\mathrm{Fl}_n
 & \text{Schubert} & \text{Grothendieck}
 & \text{\makecell{double Schubert}}\\[1mm]
\mathrm{QFl}_n
 & \text{forest} & \text{grove}
 & \text{\makecell{double forest}}
\end{array}
\)
\end{minipage}\hfill
\begin{minipage}[t]{.46\textwidth}
\centering
\(
\begin{array}{@{}c|cc@{}}
 & H^\bullet & H_T^\bullet\\ \hline
\operatorname{Gr}(r;n)
 & \text{Schur} & \text{\makecell{double Schur}}\\[1mm]
\operatorname{QGr}(r;n)
 & \text{fundamental} & \text{\makecell{double fundamental}}
\end{array}
\)
\end{minipage}
\end{center}
Each column also carries a positive expansion from the classical side
to the quasisymmetric side. Schubert polynomials expand positively in
the forest basis \cite[Theorem~1.4]{NT2024}. Grothendieck polynomials
expand positively in the grove basis
\cite[Corollary~4.8]{NST2026Grove}. Double Schubert polynomials
expand Graham-positively in the double forest basis
\cite[Theorem~11.7]{BGNST2025a}.
Likewise, Schur polynomials expand positively in the fundamental
quasisymmetric basis \cite{Gessel1984}, and double Schur polynomials
expand Graham-positively in the double fundamental quasisymmetric
basis \cite{BGNST2025a,BGST2026QG}.
In the \(K\)-theoretic statement, \(\beta\) is the deformation
parameter, \(\ell(w)\) is the number of inversions of \(w\), and
\(|F|\) is the number of internal nodes of the indexing forest \(F\).
Positivity refers to the coefficients after factoring out
\(\beta^{|F|-\ell(w)}\); at \(\beta=-1\), this factor gives the
usual alternating sign. Graham-positivity means that the coefficients
are nonnegative integral polynomials in the simple roots
\(t_{i+1}-t_i\).

Despite these positivity results, relatively few combinatorial
interpretations of the expansion coefficients are known.
For the Schur-to-fundamental expansion, Gessel's formula counts standard
Young tableaux of a given shape with a prescribed descent set
\cite{Gessel1984}. For the Schubert-to-forest expansion, Nadeau and
Tewari's \(\Omega^\bullet\)-correspondence gives a combinatorial
interpretation \cite[Theorem~1.4]{NT2024}.
For the \(K\)-theoretic and equivariant expansions, direct combinatorial
interpretations remain much less developed.

Moreover, on the Schubert side, ordinary, \(K\)-theoretic, and equivariant
Schubert polynomials can all be computed using closely related
combinatorial models, including pipe dreams \cite{FK1996,KM2005},
bumpless pipe dreams \cite{LLS2021,Weigandt2021}, and subword models
\cite{KM2005,BGNST2025a}. On the quasisymmetric side, the cited
models use compatible labelings of indexed forests for forest
polynomials \cite{NT2024} and set-valued labelings for grove
polynomials \cite{NST2026Grove}, whereas double forest polynomials are described
by the vine model \cite{BGNST2025a}. To interpret the transition
coefficients, we compare these descriptions in a common combinatorial
framework. Motivated by the Catalan rank
of the quasisymmetric coinvariant ring, we construct a lattice path model
that provides such a framework.

\subsection{Main results}

Our first main result is a common lattice path model for all six families associated with the flag varieties. The model also covers all of the four families associated with Grassmannian cases. It describes forest and double forest polynomials (\Cref{thm:1}), grove polynomials (\Cref{thm:grove-multicell-model}), Schubert and double Schubert
polynomials (\Cref{thm:double-schubert-lattice-path}), and Grothendieck
polynomials (\Cref{thm:grothendieck-lattice-path}).
For the ordinary and double polynomials, the common mechanism is a
weight-preserving bijection between subwords and cell selections on
lattice paths (\Cref{lem:subword-path-bijection}), applied to canonical
forest labelings and reduced word labelings. The \(K\)-theoretic
models use injective cell selections for Hecke words
(\Cref{lem:hecke-path-bijection}) and multi-cell selections for grove
polynomials (\Cref{thm:grove-multicell-model}).

Our second main result gives combinatorial formulas for
(i) the coefficients in the double fundamental expansion of double Schur
polynomials, expressed as weighted sums over admissible deletion sets in
reduced words (\Cref{cor:graham-pos}), and
(ii) the coefficients in the grove expansion of Grothendieck
polynomials, expressed as counts of endpoint-decreasing Hecke words
(\Cref{thm:grothendieck-grove-labeling-positivity}).
Specializing (ii) recovers Nadeau and Tewari's Schubert-to-forest
formula \cite[Theorem~1.4]{NT2024}, expressed in terms of
endpoint-decreasing reduced words
(\Cref{thm:ordinary-forest-expansion}), while specializing (i) yields
the fundamental expansion of Schur polynomials in terms of
increasing-run decompositions of reduced words (\Cref{rem:Gessel-recover}).
\Cref{fig:introduction-transitions} compares these consequences of
the paper, written above the arrows, with earlier results, written
below them.

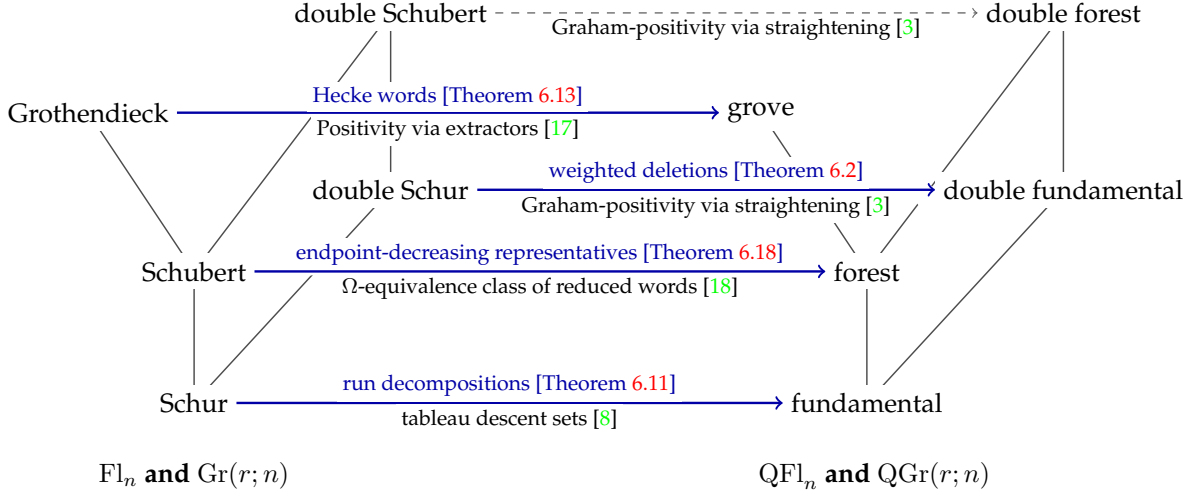
\begin{figure}[htbp]
  \centering
  \begin{tikzpicture}[
      x=1cm,y=1.2cm,
      family/.style={font=\small,inner sep=3pt,fill=white},
      arrowlabel/.style={font=\scriptsize,inner sep=2pt,fill=white},
      known/.style={->,dashed,draw=black!60,line width=0.5pt,
        preaction={draw=white,line width=3pt}},
      result/.style={->,draw=blue!65!black,line width=0.9pt,
        preaction={draw=white,line width=3pt}}
    ]
    \node[family] (groth) at (0.6,3.7) {Grothendieck};
    \node[family] (schub) at (2,1.95) {Schubert};
    \node[family] (schur) at (2,0.5) {Schur};
    \node[family] (ds) at (4.6,4.8) {double Schubert};
    \node[family] (dschur) at (4.6,2.85) {double Schur};

    \node[family] (grove) at (9.5,3.7) {grove};
    \node[family] (forest) at (10.9,1.95) {forest};
    \node[family] (fund) at (10.9,0.5) {fundamental};
    \node[family] (df) at (13.5,4.8) {double forest};
    \node[family] (dfund) at (13.5,2.85) {double fundamental};

    \begin{scope}[on background layer]
      \draw[black!70,line width=0.55pt]
        (groth.center)--(schub.center)
        (schub.center)--(ds.center)--(dschur.center)
          --(schur.center)--(schub.center)
        (grove.center)--(forest.center)
        (forest.center)--(df.center)--(dfund.center)
          --(fund.center)--(forest.center);
    \end{scope}

    \draw[known] (ds.east)--(df.west)
      node[midway,below,arrowlabel]
        {Graham-positivity via straightening \cite{BGNST2025a}};
    \draw[result] (groth.east)--(grove.west)
      node[midway,above,arrowlabel,text=blue!65!black] {Hecke words [\Cref{thm:grothendieck-grove-labeling-positivity}]}
      node[midway,below,arrowlabel]
        {Positivity via extractors \cite{NST2026Grove}};
    \draw[result] (dschur.east)--(dfund.west)
      node[midway,above,arrowlabel,text=blue!65!black] {weighted deletions [\Cref{cor:graham-pos}]}
      node[midway,below,arrowlabel] {Graham-positivity via straightening \cite{BGNST2025a}};
    \draw[result] (schub.east)--(forest.west)
      node[midway,above,arrowlabel,text=blue!65!black] {endpoint-decreasing representatives [\Cref{thm:ordinary-forest-expansion}]}
      node[midway,below,arrowlabel] {\( \Omega \)-equivalence class of reduced words \cite{NT2024}};
    \draw[result] (schur.east)--(fund.west)
    node[midway,above,arrowlabel,text=blue!65!black] {run decompositions [\Cref{rem:Gessel-recover}]}
    node[midway,below,arrowlabel] {tableau descent sets \cite{Gessel1984}};

    \node[font=\small\bfseries] at (2,-0.3) {\( \mathrm{Fl}_n \) and \( \mathrm{Gr}(r;n) \)};
    \node[font=\small\bfseries] at (11,-0.3) {\( \mathrm{QFl}_n \) and \( \mathrm{QGr}(r;n) \)};
  \end{tikzpicture}
  \caption{Positive expansions. Labels above arrows describe consequences
  of this paper, and labels below describe previous results.}
  \label{fig:introduction-transitions}
\end{figure}

The extractor and straightening methods in
\Cref{fig:introduction-transitions} establish positivity and provide
procedures for computing the coefficients \cite{NST2026Grove,BGNST2025a}.
Our formulas give explicit combinatorial interpretations of the two coefficients.

\subsection{Description of the main results}
\label{subsec:main-result-descriptions}

We first state the equivariant and \(K\)-theoretic coefficient
formulas, followed by their specializations. Fix \(n\geq1\)
and write \(\vec x_n=(x_1,\ldots,x_n)\) and
\(\vec t=(t_1,t_2,\ldots)\), where the \(t_i\) are the equivariant
parameters. For a basis element \(B\),
write \([B]f\) for its coefficient in the corresponding expansion of
\(f\).

\medskip\noindent\textbf{Double Schur to double fundamental.}
Let \(\lambda\) be a partition with at most \(n\) parts, and write
\(s_\lambda(\vec x_n;\vec t)\) for the double Schur polynomial of
shape \(\lambda\). Let \(w_\lambda\) be the associated Grassmannian
permutation, as in \Cref{subsec:schur-fundamental}. For each reduced word
\(a=(a_1,\ldots,a_m)\in\Red(w_\lambda^{-1})\), we consider
\(a_1,\ldots,a_m\) in order, keeping \(a_r\) or deleting it when
permitted by the rule in \Cref{sec:double-schur-successor-replay}.
Deletability and the positive-root weight are evaluated after all
earlier deletions. A deletion set \(D\) produced this way is admissible
if it leaves at most \(n\) increasing runs; let \(\mathcal A_n(a)\)
denote the set of such \(D\). If these runs have lengths
\((\ell_1,\ldots,\ell_q)\), their basis index is
\[
  c_D=(0^{n-q},\ell_q,\ldots,\ell_1),
\]
and \(\wt_a(D)\) is the product of the deletion weights.
For the empty word, the index is \((0^n)\) and the weight is \(1\).
The possible indices form the set \(\operatorname{Pad}_n\) of sequences of length \(n\) with leading zeros followed by positive parts.
Write \(\mathfrak F_c(\vec x_n;\vec t)\) for the corresponding
double fundamental polynomial.

\begin{thm}\label{thm:intro-double-schur-coefficients}
For every partition \(\lambda\) with at most \(n\) parts and every
\(c\in\operatorname{Pad}_n\),
\[
  [\mathfrak F_c]s_\lambda\xt
  =\sum_{a\in\Red(w_\lambda^{-1})}
     \ \sum_{\substack{D\in\mathcal A_n(a)\\c_D=c}}
        \wt_a(D).
\]
In particular, these coefficients are sums of products of positive
roots \(t_j-t_i\) with \(i<j\).
\end{thm}

The formula is proved in \Cref{cor:graham-pos}; the stronger
wordwise Graham-positivity statement is
\Cref{thm:double-schur-successor-replay}. It gives a direct weighted
counting rule in the double Schur case, complementing the manifestly
positive straightening algorithm for the double Schubert-to-double
forest expansion in \cite{BGNST2025a}. In particular, it answers
\cite[Question~1.2(2)]{BGNST2025a}, which asks for an explicit
combinatorial interpretation of the double Schur-to-double fundamental
coefficients. Through the quasisymmetric Grassmannian of
Bergeron, Gagnon, Spink, and Tewari \cite{BGST2026QG}, the formula
also computes equivariant degrees of quasisymmetric Schubert cycles.
Their Theorem~8.8 identifies these coefficients with equivariant
degree pairings, and their Theorem~8.9 gives a manifestly positive
computation of the degrees. Our formula expresses them as weighted
counts of admissible subsets of reduced words
(\Cref{cor:qgr-deletion-degree}).

\medskip\noindent\textbf{Grothendieck to grove.}
Let \(G_w^{(\beta)}(\vec x_n)\) be the Grothendieck polynomial
of \(w\in S_{n+1}\). We express its grove coefficients as counts of
Hecke words. For a word \(b=(b_1,\ldots,b_m)\) in \([n]\), start
with the identity permutation and read the letters in order.
At step \(r\), interchange the entries in positions \(b_r\) and
\(b_r+1\) if this increases the number of inversions; otherwise,
leave the permutation unchanged. The word \(b\) is a \emph{Hecke word}
of \(w\) if the resulting permutation is \(w\).

Apply Nadeau and Tewari's \(\Omega\)-parking
\cite[Sections~1.2 and 4.2]{NT2024} to \(b\), using the block
description in \Cref{subsec:endpoint-decreasing}.
Immediately after the \(r\)-th insertion, let \(e_r(b)\) be the
left endpoint of the maximal interval of occupied positions
containing the new block. We call \(b\) \emph{endpoint-decreasing}
if \(e_1(b)\geq\cdots\geq e_m(b)\). When all endpoints are positive,
the \emph{endpoint code} \(c^{\mathrm{end}}(b)
=(c_i^{\mathrm{end}}(b))_{i\geq1}\) counts how often each endpoint occurs,
so
\[
  c_i^{\mathrm{end}}(b)=\bigl|\{r\in[m]:e_r(b)=i\}\bigr|.
\]
Let \(\EH_n^+(w)\) consist of the endpoint-decreasing Hecke words
\(b\) of \(w\) in \([n]\) of length at most \(\binom{n+1}{2}\)
such that \(e_r(b)>0\) for \(1\leq r\leq m\) and
\(b_r\neq b_{r+1}\) for \(1\leq r<m\).
For a finitely supported code \(c=(c_i)_{i\geq1}\) of nonnegative
integers, write \(\widetilde\PP_c^{(\beta)}\) for the corresponding
grove polynomial and put \(|c|=\sum_i c_i\).

\begin{thm}\label{thm:intro-grothendieck-grove}
For every \(w\in S_{n+1}\) and every code \(c\) supported on
\([n]\) with \(|c|\geq\ell(w)\),
\[
  [\widetilde\PP_c^{(\beta)}]G_w^{(\beta)}
  =\beta^{|c|-\ell(w)}
    \bigl|\{b\in\EH_n^+(w^{-1}):c^{\mathrm{end}}(b)=c\}\bigr|.
\]
The coefficient is zero when \(|c|<\ell(w)\).
\end{thm}

The formula is proved in
\Cref{thm:grothendieck-grove-labeling-positivity} and gives a counting
interpretation of the positivity established in
\cite[Corollary~4.8]{NST2026Grove}.

\medskip\noindent\textbf{Schubert to forest.}
Setting \(\beta=0\) gives the Schubert polynomial
\(G_w^{(0)}(\vec x_n)=\mathfrak S_w(\vec x_n)\) and the forest
polynomial \(\widetilde\PP_c^{(0)}=\PP_c\) indexed by \(c\).
A Hecke word \(b\) of \(w^{-1}\) contributes the factor
\(\beta^{|b|-\ell(w)}\), so precisely the reduced words survive.
For \(w\ne\mathrm{id}\), let \(\ERed^+(w)\) be the set of reduced
words \(a\) of \(w\) satisfying
\(e_1(a)\geq\cdots\geq e_{\ell(w)}(a)>0\), and set
\(\ERed^+(\mathrm{id})=\{\varnothing\}\).

\begin{cor}\label{cor:intro-schubert-forest}
For every \(w\in S_{n+1}\) and every code \(c\) supported on \([n]\),
\[
  [\PP_c]\mathfrak S_w
  =\bigl|\{a\in\ERed^+(w^{-1}):c^{\mathrm{end}}(a)=c\}\bigr|.
\]
\end{cor}

Nadeau and Tewari's \(\Omega^\bullet\)-correspondence expresses the
Schubert-to-forest expansion as a sum over \(\Omega\)-equivalence
classes of reduced words \cite[Theorem~1.4]{NT2024}.
By \eqref{eq:endpoint-sort-omega-classes}, each class contributing
a nonzero forest polynomial has a unique representative
\(a\in\ERed^+(w^{-1})\), and the forest of that class has code
\(c^{\mathrm{end}}(a)\). Thus the coefficient \( [\PP_c]\mathfrak S_w \) counts the
classes with forest code \(c\) by their endpoint-decreasing
representatives.

\medskip\noindent\textbf{Schur to fundamental.}
Setting \(\vec t=\vec0\) specializes \(s_\lambda(\vec x_n;\vec t)\)
to the Schur polynomial \(s_\lambda(\vec x_n)\).
Every nonempty deletion set has weight zero, since each deletion
contributes a factor \(t_j-t_i\). For
\(a\in\Red(w_\lambda^{-1})\), let
\((\ell_1,\ldots,\ell_{q(a)})\) be the lengths of the maximal
increasing runs of \(a\), read from left to right.
The empty deletion set has weight \(1\) and is admissible exactly
when \(q(a)\leq n\); in that case, put
\[
  c_\varnothing(a)=(0^{n-q(a)},\ell_{q(a)},\ldots,\ell_1).
\]
For the empty word, set \(q(\varnothing)=0\) and
\(c_\varnothing(\varnothing)=(0^n)\).

\begin{cor}\label{cor:intro-schur-fundamental}
For every partition \(\lambda\) with at most \(n\) parts,
\[
  s_\lambda(\vec x_n)
  =\sum_{\substack{a\in\Red(w_\lambda^{-1})\\q(a)\leq n}}
     \mathfrak F_{c_\varnothing(a)}(\vec x_n;\vec0).
\]
\end{cor}

Each summand is the ordinary fundamental quasisymmetric polynomial
indexed by the reversed run composition
\((\ell_{q(a)},\ldots,\ell_1)\).
Under the heap correspondence, the increasing-run boundaries are
the descent positions of the associated standard Young tableau
of shape \(\lambda\). As explained in \Cref{rem:Gessel-recover},
reversing the variable order and using the symmetry of \(s_\lambda\)
converts the right hand side of the formula in \Cref{cor:intro-schur-fundamental} into Gessel's tableau expansion
\cite{Gessel1984}.

\subsection{Organization}

The paper is organized as follows. \Cref{sec:preliminaries} recalls
the polynomial families and their subword models.
\Cref{sec:lattice-paths} introduces the lattice path model and successor relations. \Cref{sec:forest-path-models} proves the forest
and grove formulas. \Cref{sec:schubert-path-models} first recalls
Nadeau and Tewari's \(\Omega^\bullet\)-correspondence, then develops
its marked-block realization, word paths, and endpoint sorting to
prove the Schubert and Grothendieck formulas.
\Cref{sec:change-of-basis} gives the four coefficient interpretations.

\subsection*{Acknowledgements}
We are grateful to Younggwang Cho for discussions that helped shape the construction of the lattice path model. We thank Jang Soo Kim and Philippe Nadeau for their careful reading of the draft and valuable feedback, and Jing Liu for introducing us to this problem. We are particularly grateful to Hunter Spink for many helpful and inspiring conversations.

\subsection*{Declaration on Generative AI}
During the preparation of this work, the author used OpenAI's ChatGPT (GPT-5.5, GPT-5.6 Sol, and GPT-6 Astra) for literature searches, generating SageMath code for computational experiments, improving the grammar and readability of the manuscript, and generating TikZ code for figures. The computational experiments were used to explore examples, identify and test conjectural patterns, and provide experimental evidence relevant to various lemmas, propositions, and theorems.
Conversations with ChatGPT were helpful and inspiring. The author solely developed the lattice path model and the mathematical ideas underlying the statements and proofs in this paper. All references suggested or located with AI assistance were independently verified, and all AI-generated code, computations, and manuscript text were reviewed and checked by the author. The author takes full responsibility for the correctness and content of the paper.

\section{Preliminaries}\label{sec:preliminaries}

We recall the forest labelings and subword formulas used in the lattice
path constructions. Write \([n]=\{1,\ldots,n\}\), with
\([0]=\varnothing\), and let \(\vec x=(x_1,x_2,\ldots)\) and
\(\vec t=(t_1,t_2,\ldots)\). The \(t_i\) are the equivariant
variables. All polynomials involve only finitely many variables.

\subsection{Indexed forests and codes}
We follow \cite{NT2024,NST2024,BGNST2025a}. A \emph{binary plane tree}
is a rooted tree in which every vertex is either a leaf or an internal
node with two ordered children. Denote the one-vertex tree by \(*\).
An \emph{indexed forest} is a sequence \(F=(T_1,T_2,\ldots)\) of binary
plane trees, all but finitely many of which equal \(*\). Let \(\For\)
be the set of indexed forests, let \(\IN(F)\) be the set of internal
nodes, and write \(|F|=|\IN(F)|\). The empty forest is
\(\varnothing=(*,*,\ldots)\).

Label the leaves from left to right by \(\mathbb N\). The support
\(\supp(F)\) consists of the labels of leaves belonging to nontrivial
trees, and
\[
  \For_n=\{F\in\For:\supp(F)\subseteq[n]\}.
\]
Let \(\Codes\) be the set of finitely supported sequences
\(c=(c_1,c_2,\ldots)\) of nonnegative integers. Put
\[
  \supp(c)=\{i:c_i>0\},\qquad
  |c|=\sum_i c_i,\qquad
  \Codes_n=\{c\in\Codes:\supp(c)\subseteq[n]\}.
\]
We omit trailing zeros when convenient. The \emph{code} of \(F\) is
\[
  c(F)=(c_i)_{i\geq1},\qquad
  c_i=|\{v\in\IN(F):\rho_F(v)=i\}|.
\]
The map \(c:\For\to\Codes\) is a bijection, with inverse
\(c\mapsto F(c)\).

Two labels on internal nodes will be used. The \emph{flag} \(\rho_F(v)\) is the label of the leftmost
leaf below \(v\), reached by repeatedly following left edges.
In particular, \(c(F)\in\Codes_n\) if and only if
\(\rho_F(v)\leq n\) for every internal node \(v\).
The \emph{canonical label}
\(\ell_F(v)\) is the label of the rightmost leaf in the left subtree
of \(v\).
A \emph{Sylvester word} of \(F\) lists its canonical labels in an
order in which every parent precedes its internal children. We write
\(\Syl(F)\) for the set of these words.

\begin{figure}
  \centering
  \begin{tikzpicture}[
      x=0.72cm,
      y=0.72cm,
      internal/.style={circle,fill=black,inner sep=1.7pt},
      rho/.style={text=red,font=\scriptsize}
    ]
    \draw[gray!55] (0.7,0)--(9.35,0);
    \foreach \x/\lab in {1/1,2/2,3/3,4/4,5/5,6/6,7/7,8/8,9/9} {
      \draw[semithick] (\x-0.10,-0.10)--(\x+0.10,0.10);
      \draw[semithick] (\x-0.10,0.10)--(\x+0.10,-0.10);
      \node[below=3pt,font=\scriptsize] at (\x,0) {\lab};
    }

    \node[internal] (root) at (3.5,1.8) {};
    \node[internal] (left) at (2.5,0.6) {};
    \node[internal] (right) at (4.5,0.6) {};
    \draw (root)--(left) (root)--(right);
    \draw (left)--(2,0) (left)--(3,0);
    \draw (right)--(4,0) (right)--(5,0);

    \node[internal] (small) at (7.5,0.6) {};
    \draw (small)--(7,0) (small)--(8,0);

    \node[rho] at (3.5,2.07) {\(2\)};
    \node[rho] at (2.25,0.84) {\(2\)};
    \node[rho] at (4.75,0.84) {\(4\)};
    \node[rho] at (7.5,0.87) {\(7\)};
    \node[right] at (9.35,0) {\(\cdots\)};
  \end{tikzpicture}
  \caption{An indexed forest \(F\). The crosses represent leaves, and the
  red labels are the values of \(\rho_F\) on the internal nodes.}
  \label{fig:indexed-forest-code}
\end{figure}
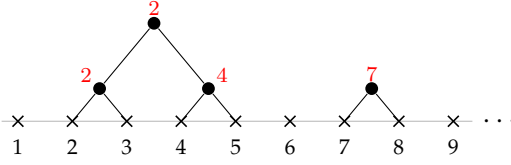

The forest in \Cref{fig:indexed-forest-code} has code
\(c(F)=(0,2,0,1,0,0,1)\). Its canonical labels are
\(2,3,4,7\); a Sylvester word must place \(3\) before both \(2\)
and \(4\).

\subsection{Forest and grove polynomials}
A \emph{compatible labeling} of \(F\) is a map
\(f:\IN(F)\to\mathbb Z_{>0}\) such that \(f(v)\leq\rho_F(v)\)
for every internal node \(v\), and
\begin{equation}\label{eq:ordinary-forest-labeling}
  \begin{aligned}  
  f(v)&\leq f(v_L)&&\text{if}\ v_L\in\IN(F), \\
  f(v)&<f(v_R)&&\text{if}\ v_R\in\IN(F),
  \end{aligned}
\end{equation}
where \(v_L,v_R\) are the left and right children of \(v\), respectively.
For \(F\in\For\), the \emph{forest polynomial} \(\PP_F\) of
\(F\) is \cite{NT2024,NST2024}
\begin{equation}\label{eq:forest-polynomial}
  \PP_F(\vec x)=\sum_f\prod_{v\in\IN(F)}x_{f(v)},
\end{equation}
where the sum runs over all compatible labelings. We also
write \(\PP_c=\PP_{F(c)}\).

Grove polynomials replace the labels by nonempty finite sets. A
\emph{grove labeling} assigns such a set
\(\kappa(v)\subset\mathbb Z_{>0}\) to each internal node \(v\), with
\(\max\kappa(v)\leq\rho_F(v)\), and satisfies
\begin{equation}\label{eq:grove-labeling-compatibility}
\begin{aligned}
  \max\kappa(v)&\leq\min\kappa(v_L)
    &&\text{if}\ v_L\in\IN(F),\\
  \max\kappa(v)&<\min\kappa(v_R)
    &&\text{if}\ v_R\in\IN(F).
\end{aligned}
\end{equation}
For \(F\in\For\) and an indeterminate \(\beta\), the
\emph{grove polynomial} \(\widetilde\PP_F^{(\beta)}\) of \(F\)
is \cite{NST2026Grove}
\begin{equation}\label{eq:grove-labeling-polynomial}
  \widetilde\PP_F^{(\beta)}(\vec x)
  =\sum_{\kappa}
     \beta^{|\kappa|-|F|}
     \prod_{v\in\IN(F)}\prod_{q\in\kappa(v)}x_q,
  \qquad
  |\kappa|=\sum_{v\in\IN(F)}|\kappa(v)|.
\end{equation}
We also write \(\widetilde\PP_c^{(\beta)}
=\widetilde\PP_{F(c)}^{(\beta)}\).
At \(\beta=0\), only singleton labels contribute, so
\(\widetilde\PP_F^{(0)}=\PP_F\). Both polynomials equal \(1\) for
the empty forest.

\begin{exam}\label{ex:forest-grove-labelings}
For the forest \(F\) in \Cref{fig:indexed-forest-code}, the
compatible labeling on the left of \Cref{fig:forest-grove-labelings}
contributes \(x_2^2x_3x_5\) to \(\PP_F\).
The grove labeling on the right has \(|\kappa|=8\), so it contributes
\[
  \beta^{8-4}x_1^2x_2^2x_3^2x_4x_5
  =\beta^4x_1^2x_2^2x_3^2x_4x_5
\]
to \(\widetilde\PP_F^{(\beta)}\).
The parent label \(\{1,2\}\) has maximum \(2\), which is at most
the minimum of the left-child label \(\{2\}\) and strictly less than
the minimum of the right-child label \(\{3,4\}\).
\end{exam}
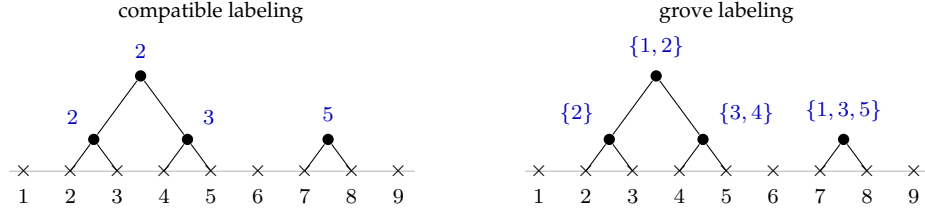
\begin{figure}
  \centering
  \begin{tikzpicture}[x=.62cm,y=.7cm,
    internal/.style={circle,fill=black,inner sep=1.5pt},
    every node/.style={font=\scriptsize}]
    \foreach \shift/\kind in {0/compatible,11/grove} {
      \begin{scope}[xshift=\shift*.62cm]
        \draw[gray!55] (.7,0)--(9.35,0);
        \foreach \i in {1,...,9} {
          \node at (\i,0) {\(\times\)};
          \node[below=3pt] at (\i,0) {\(\i\)};
        }
        \node[internal] (r\kind) at (3.5,1.8) {};
        \node[internal] (l\kind) at (2.5,.6) {};
        \node[internal] (u\kind) at (4.5,.6) {};
        \node[internal] (v\kind) at (7.5,.6) {};
        \draw (r\kind)--(l\kind) (r\kind)--(u\kind)
          (l\kind)--(2,0) (l\kind)--(3,0)
          (u\kind)--(4,0) (u\kind)--(5,0)
          (v\kind)--(7,0) (v\kind)--(8,0);
        \node at (5,3) {\kind\ labeling};
      \end{scope}
    }
    \node[blue!75!black,above=3pt] at (3.5,1.8) {\(2\)};
    \node[blue!75!black,above left=2pt] at (2.5,.6) {\(2\)};
    \node[blue!75!black,above right=2pt] at (4.5,.6) {\(3\)};
    \node[blue!75!black,above=3pt] at (7.5,.6) {\(5\)};
    \begin{scope}[xshift=6.82cm]
      \node[blue!75!black,above=3pt] at (3.5,1.8) {\(\{1,2\}\)};
      \node[blue!75!black,above left=2pt] at (2.5,.6) {\(\{2\}\)};
      \node[blue!75!black,above right=2pt] at (4.5,.6) {\(\{3,4\}\)};
      \node[blue!75!black,above=3pt] at (7.5,.6) {\(\{1,3,5\}\)};
    \end{scope}
  \end{tikzpicture}
  \caption{A compatible labeling and a grove labeling of the
    forest in \Cref{fig:indexed-forest-code}. Blue labels record the
    choices contributing the two terms in \Cref{ex:forest-grove-labelings}.}
  \label{fig:forest-grove-labelings}
\end{figure}

\subsection{Subword formulas for double polynomials}
\label{subsec:vine-subword}
Let \(S_n\) be the symmetric group on \([n]\), with
\(S_n\subset S_{n+1}\) given by adjoining the fixed point \(n+1\).
We write \(S_\infty=\bigcup_{n\geq1}S_n\) for the group of finitely
supported permutations of \(\mathbb N\), generated by the simple
transpositions \(s_i=(i,i+1)\). For \(w\in S_\infty\), let
\(\ell(w)\) denote the Coxeter length of \(w\). The set of reduced
words of \(w\) is
\[
  \Red(w)=\{(a_1,\ldots,a_{\ell(w)})\in\mathbb N^{\ell(w)}:
                    w=s_{a_1}\cdots s_{a_{\ell(w)}}\}.
\]
Fix \(n\geq1\). For \(1\leq q\leq n\), define
\[
  \omega_{[n]}^{(q)}
  =n^{(q)}(n-1)^{(q)}\cdots q^{(q)}
   \overline{q+1}^{(q)}\overline{q+2}^{(q)}\cdots\overline n^{(q)},
  \qquad
  \widetilde\omega_{[n]}
  =\omega_{[n]}^{(1)}\cdots\omega_{[n]}^{(n)}.
\]
A letter \(B\) of \(\widetilde\omega_{[n]}\) has the form
\(j^{(q)}\) with \(q\leq j\leq n\), or
\(\overline j^{(q)}\) with \(q<j\leq n\).
The value \(\val(B)\) is the underlying integer \(j\), obtained by
forgetting the bar and superscript, and \(\wt(B)\) denotes the
polynomial weight assigned to \(B\) in the subword formulas.
As described in \Cref{fig:vine-array}, arrange the unbarred and barred parts of \( \omega_{[n]}^{(q)} \) in rows
\(2q-1\) and \(2q\), respectively, and let \(d(B)\) denote the
row containing \(B\). The value, weight, and row index are given by
\[
  \begin{array}{c|ccc}
    B & \val(B) & \wt(B) & d(B) \\ \hline
    j^{(q)} & j & x_q-t_j & 2q-1 \\
    \overline j^{(q)} & j & t_j-t_q & 2q.
\end{array}
\]

The \emph{vine order} on the letters of
\(\widetilde\omega_{[n]}\) is defined by \(B\prec B'\) if
\(B\) occurs before \(B'\) in the word.

\Cref{fig:vine-array} displays the letters of
\(\widetilde\omega_{[4]}\) in the left array and the corresponding
weights in the right array. The vine order is obtained by
reading the rows from top to bottom, with odd rows read from right
to left and even rows from left to right.
\begin{figure}
  \centering
  \begin{minipage}[t]{.34\textwidth}
    \centering
    Letters\par\smallskip
\begin{tikzpicture}[x=.9cm,y=.55cm,every node/.style={font=\scriptsize}]
\node[left,text=gray] at (-0.12,-0.5) {1};
\fill[purple!9] (0,-1) rectangle (1,0);
\draw[gray!55,line width=.3pt] (0,-1) rectangle (1,0);
\node at (0.5,-0.5) {\(1^{(1)}\)};
\fill[purple!9] (1,-1) rectangle (2,0);
\draw[gray!55,line width=.3pt] (1,-1) rectangle (2,0);
\node at (1.5,-0.5) {\(2^{(1)}\)};
\fill[purple!9] (2,-1) rectangle (3,0);
\draw[gray!55,line width=.3pt] (2,-1) rectangle (3,0);
\node at (2.5,-0.5) {\(3^{(1)}\)};
\fill[purple!9] (3,-1) rectangle (4,0);
\draw[gray!55,line width=.3pt] (3,-1) rectangle (4,0);
\node at (3.5,-0.5) {\(4^{(1)}\)};
\node[left,text=gray] at (-0.12,-1.5) {2};
\fill[orange!13] (1,-2) rectangle (2,-1);
\draw[gray!55,line width=.3pt] (1,-2) rectangle (2,-1);
\node at (1.5,-1.5) {\(\overline{2}^{(1)}\)};
\fill[orange!13] (2,-2) rectangle (3,-1);
\draw[gray!55,line width=.3pt] (2,-2) rectangle (3,-1);
\node at (2.5,-1.5) {\(\overline{3}^{(1)}\)};
\fill[orange!13] (3,-2) rectangle (4,-1);
\draw[gray!55,line width=.3pt] (3,-2) rectangle (4,-1);
\node at (3.5,-1.5) {\(\overline{4}^{(1)}\)};
\node[left,text=gray] at (-0.12,-2.5) {3};
\fill[purple!9] (1,-3) rectangle (2,-2);
\draw[gray!55,line width=.3pt] (1,-3) rectangle (2,-2);
\node at (1.5,-2.5) {\(2^{(2)}\)};
\fill[purple!9] (2,-3) rectangle (3,-2);
\draw[gray!55,line width=.3pt] (2,-3) rectangle (3,-2);
\node at (2.5,-2.5) {\(3^{(2)}\)};
\fill[purple!9] (3,-3) rectangle (4,-2);
\draw[gray!55,line width=.3pt] (3,-3) rectangle (4,-2);
\node at (3.5,-2.5) {\(4^{(2)}\)};
\node[left,text=gray] at (-0.12,-3.5) {4};
\fill[orange!13] (2,-4) rectangle (3,-3);
\draw[gray!55,line width=.3pt] (2,-4) rectangle (3,-3);
\node at (2.5,-3.5) {\(\overline{3}^{(2)}\)};
\fill[orange!13] (3,-4) rectangle (4,-3);
\draw[gray!55,line width=.3pt] (3,-4) rectangle (4,-3);
\node at (3.5,-3.5) {\(\overline{4}^{(2)}\)};
\node[left,text=gray] at (-0.12,-4.5) {5};
\fill[purple!9] (2,-5) rectangle (3,-4);
\draw[gray!55,line width=.3pt] (2,-5) rectangle (3,-4);
\node at (2.5,-4.5) {\(3^{(3)}\)};
\fill[purple!9] (3,-5) rectangle (4,-4);
\draw[gray!55,line width=.3pt] (3,-5) rectangle (4,-4);
\node at (3.5,-4.5) {\(4^{(3)}\)};
\node[left,text=gray] at (-0.12,-5.5) {6};
\fill[orange!13] (3,-6) rectangle (4,-5);
\draw[gray!55,line width=.3pt] (3,-6) rectangle (4,-5);
\node at (3.5,-5.5) {\(\overline{4}^{(3)}\)};
\node[left,text=gray] at (-0.12,-6.5) {7};
\fill[purple!9] (3,-7) rectangle (4,-6);
\draw[gray!55,line width=.3pt] (3,-7) rectangle (4,-6);
\node at (3.5,-6.5) {\(4^{(4)}\)};
\draw[semithick] (0,0)--(4,0)--(4,-7)--(3,-7)--(3,-5)--(2,-5)--(2,-3)--(1,-3)--(1,-1)--(0,-1)--(0,0);
\node[left,text=gray] at (-.12,.35) {\(d\)};
\node[above] at (0.5,0.1) {\(j=1\)};
\node[above] at (1.5,0.1) {\(j=2\)};
\node[above] at (2.5,0.1) {\(j=3\)};
\node[above] at (3.5,0.1) {\(j=4\)};
\end{tikzpicture}
  \end{minipage}\qquad
  \begin{minipage}[t]{.43\textwidth}
    \centering
    Weights\par\smallskip
\begin{tikzpicture}[x=1.25cm,y=.55cm,every node/.style={font=\scriptsize}]
\node[left,text=gray] at (-0.12,-0.5) {1};
\fill[purple!9] (0,-1) rectangle (1,0);
\draw[gray!55,line width=.3pt] (0,-1) rectangle (1,0);
\node at (0.5,-0.5) {\(x_1-t_1\)};
\fill[purple!9] (1,-1) rectangle (2,0);
\draw[gray!55,line width=.3pt] (1,-1) rectangle (2,0);
\node at (1.5,-0.5) {\(x_1-t_2\)};
\fill[purple!9] (2,-1) rectangle (3,0);
\draw[gray!55,line width=.3pt] (2,-1) rectangle (3,0);
\node at (2.5,-0.5) {\(x_1-t_3\)};
\fill[purple!9] (3,-1) rectangle (4,0);
\draw[gray!55,line width=.3pt] (3,-1) rectangle (4,0);
\node at (3.5,-0.5) {\(x_1-t_4\)};
\node[left,text=gray] at (-0.12,-1.5) {2};
\fill[orange!13] (1,-2) rectangle (2,-1);
\draw[gray!55,line width=.3pt] (1,-2) rectangle (2,-1);
\node at (1.5,-1.5) {\(t_2-t_1\)};
\fill[orange!13] (2,-2) rectangle (3,-1);
\draw[gray!55,line width=.3pt] (2,-2) rectangle (3,-1);
\node at (2.5,-1.5) {\(t_3-t_1\)};
\fill[orange!13] (3,-2) rectangle (4,-1);
\draw[gray!55,line width=.3pt] (3,-2) rectangle (4,-1);
\node at (3.5,-1.5) {\(t_4-t_1\)};
\node[left,text=gray] at (-0.12,-2.5) {3};
\fill[purple!9] (1,-3) rectangle (2,-2);
\draw[gray!55,line width=.3pt] (1,-3) rectangle (2,-2);
\node at (1.5,-2.5) {\(x_2-t_2\)};
\fill[purple!9] (2,-3) rectangle (3,-2);
\draw[gray!55,line width=.3pt] (2,-3) rectangle (3,-2);
\node at (2.5,-2.5) {\(x_2-t_3\)};
\fill[purple!9] (3,-3) rectangle (4,-2);
\draw[gray!55,line width=.3pt] (3,-3) rectangle (4,-2);
\node at (3.5,-2.5) {\(x_2-t_4\)};
\node[left,text=gray] at (-0.12,-3.5) {4};
\fill[orange!13] (2,-4) rectangle (3,-3);
\draw[gray!55,line width=.3pt] (2,-4) rectangle (3,-3);
\node at (2.5,-3.5) {\(t_3-t_2\)};
\fill[orange!13] (3,-4) rectangle (4,-3);
\draw[gray!55,line width=.3pt] (3,-4) rectangle (4,-3);
\node at (3.5,-3.5) {\(t_4-t_2\)};
\node[left,text=gray] at (-0.12,-4.5) {5};
\fill[purple!9] (2,-5) rectangle (3,-4);
\draw[gray!55,line width=.3pt] (2,-5) rectangle (3,-4);
\node at (2.5,-4.5) {\(x_3-t_3\)};
\fill[purple!9] (3,-5) rectangle (4,-4);
\draw[gray!55,line width=.3pt] (3,-5) rectangle (4,-4);
\node at (3.5,-4.5) {\(x_3-t_4\)};
\node[left,text=gray] at (-0.12,-5.5) {6};
\fill[orange!13] (3,-6) rectangle (4,-5);
\draw[gray!55,line width=.3pt] (3,-6) rectangle (4,-5);
\node at (3.5,-5.5) {\(t_4-t_3\)};
\node[left,text=gray] at (-0.12,-6.5) {7};
\fill[purple!9] (3,-7) rectangle (4,-6);
\draw[gray!55,line width=.3pt] (3,-7) rectangle (4,-6);
\node at (3.5,-6.5) {\(x_4-t_4\)};
\draw[semithick] (0,0)--(4,0)--(4,-7)--(3,-7)--(3,-5)--(2,-5)--(2,-3)--(1,-3)--(1,-1)--(0,-1)--(0,0);
\node[left,text=gray] at (-.12,.35) {\(d\)};
\node[above] at (0.5,0.1) {\(j=1\)};
\node[above] at (1.5,0.1) {\(j=2\)};
\node[above] at (2.5,0.1) {\(j=3\)};
\node[above] at (3.5,0.1) {\(j=4\)};
\end{tikzpicture}
  \end{minipage}
  \caption{Letters of \(\widetilde\omega_{[4]}\) and their weights.
    Row \(d=2q-1\) contains unbarred letters, and row \(d=2q\)
    contains barred letters.}
  \label{fig:vine-array}
\end{figure}

A \emph{subword}
\(\pi=(B_1,\ldots,B_m)\) of \(\widetilde\omega_{[n]}\) is obtained by selecting letters and retaining
their vine order. We write \(\pi\subseteq\widetilde\omega_{[n]}\) to indicate that \(\pi\) is a subword. For \(\pi=(B_1,\ldots,B_m)\subseteq\widetilde\omega_{[n]}\), we define
\[
  \val(\pi)=(\val(B_1),\ldots,\val(B_m)),\qquad
  |\pi|=m,\qquad
  \wt(\pi)=\prod_{r=1}^m\wt(B_r).
\]
The empty subword has an empty value word and weight \(1\). These
conventions also apply to the staircase word in the next subsection.

For \(w\in S_{n+1}\) and \(F\in\For_{n+1}\), define
\[
\begin{aligned}
  \mathcal R(\widetilde\omega_{[n]};w)
    &=\{\pi\subseteq\widetilde\omega_{[n]}:\val(\pi)\in\Red(w)\},\\
  \mathcal R(\widetilde\omega_{[n]};F)
    &=\{\pi\subseteq\widetilde\omega_{[n]}:\val(\pi)\in\Syl(F)\}.
\end{aligned}
\]
Thus the two formulas differ only in the condition on the value
word.

For \(w\in S_{n+1}\), the \emph{double Schubert polynomial}
\(\mathfrak S_w\) of \(w\) is \cite{BGNST2025a}
\begin{equation}\label{eq:double-schubert-subword}
  \mathfrak S_w(\vec x;\vec t)
  =\sum_{\pi\in\mathcal R(\widetilde\omega_{[n]};w)}\wt(\pi).
\end{equation}

For \(F\in\For_{n+1}\), the \emph{double forest polynomial}
\(\PP_F\) of \(F\) is \cite{BGNST2025a}
\begin{equation}\label{eq:double-forest-subword}
  \PP_F(\vec x;\vec t)
  =\sum_{\pi\in\mathcal R(\widetilde\omega_{[n]};F)}\wt(\pi).
\end{equation}

Setting \(\vec t=\vec0\) gives the Schubert polynomial \(\mathfrak S_w(\vec x)\) and the forest polynomial \(\PP_F(\vec x)\), respectively.
We also write \(\PP_c(\vec x;\vec t)=\PP_{F(c)}(\vec x;\vec t)\).

The subword formulas admit the vine diagrams of
\cite[\S5.3 and \S6.1]{BGNST2025a}, using the tiles in
\Cref{fig:vine-examples}. For a forest subword, at most one box
is selected in each column. Put a T-shaped tile in each selected box, a horizontal
tile in each unselected box above a T-shaped tile in the same column, and an odd or
even elbow in each remaining box according to its row parity.
The branching pattern recovers the indexed forest.
For a Schubert subword, put a cross in every selected box and elbows
in all other boxes. The strands connect bottom position \(i\) to
top position \(w(i)\), with no two strands crossing more than once.
The gray boxes in \Cref{fig:vine-examples} supply only the routing
elbows outside the array and carry no weights.

\begin{exam}\label{ex:forest-vine-subword}
Let \(F_0\) be the forest with code \((1,0,2)\).
Its canonical labels are \(1,3,4\), and
\[
  \Syl(F_0)=\{(1,4,3),(4,1,3),(4,3,1)\}.
\]
The subword
\(\pi_F=(1^{(1)},\overline4^{(1)},3^{(3)})\)
has value word \((1,4,3)\), so it contributes
\[
  \wt(\pi_F)=(x_1-t_1)(t_4-t_1)(x_3-t_3)
\]
to \(\PP_{F_0}(\vec x;\vec t)\).
Its T tiles are in boxes \((d,j)=(1,1),(2,4),(5,3)\), giving the
forest vine on the left of \Cref{fig:vine-examples}.
\end{exam}

\begin{exam}\label{ex:schubert-vine-subword}
Let \(w=21453=s_1s_3s_4\in S_5\).
The subword \(\pi_w=(1^{(1)},3^{(2)},\overline4^{(3)})\)
has reduced value word \((1,3,4)\), so it contributes
\[
  \wt(\pi_w)=(x_1-t_1)(x_2-t_3)(t_4-t_3)
\]
to \(\mathfrak S_w(\vec x;\vec t)\).
The crosses are in boxes \((d,j)=(1,1),(3,3),(6,4)\).
The right-hand vine in \Cref{fig:vine-examples} has bottom labels
\(2,1,4,5,3\), as required by the strand connection rule.
\end{exam}

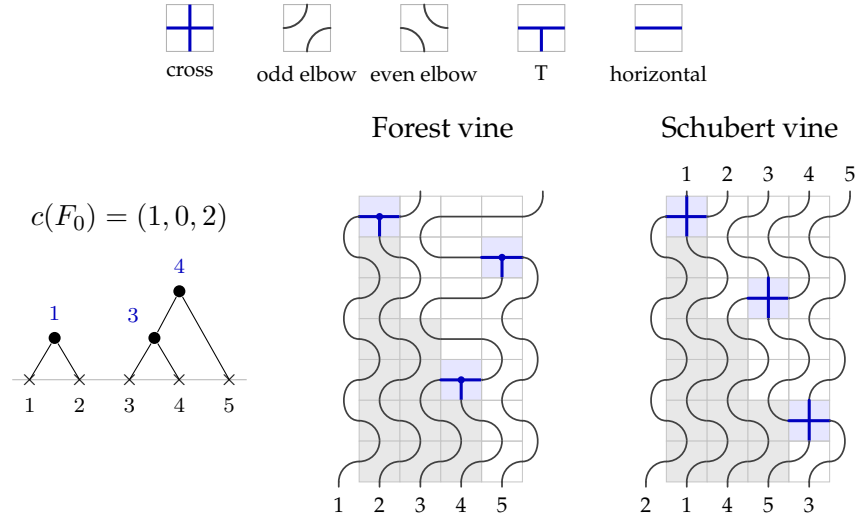
\begin{figure}
  \centering
\begin{tikzpicture}[x=.62cm,y=.62cm,vine/.style={black!75,line width=.75pt,line cap=round},chosen/.style={blue!75!black,line width=1.15pt},every node/.style={font=\scriptsize,align=center}]
\begin{scope}[xshift=0cm]
\draw[gray!60] (0,0) rectangle (1,1);
\draw[chosen] (0,0.5)--(1,0.5);
\draw[chosen] (0.5,0)--(0.5,1);
\node[below] at (.5,-.08) {cross};
\end{scope}
\begin{scope}[xshift=1.55cm]
\draw[gray!60] (0,0) rectangle (1,1);
\draw[vine] (0,0.5) .. controls (0.276,0.5) and (0.5,0.724) .. (0.5,1);
\draw[vine] (0.5,0) .. controls (0.5,0.276) and (0.724,0.5) .. (1,0.5);
\node[below] at (.5,-.08) {odd elbow};
\end{scope}
\begin{scope}[xshift=3.1cm]
\draw[gray!60] (0,0) rectangle (1,1);
\draw[vine] (0,0.5) .. controls (0.276,0.5) and (0.5,0.276) .. (0.5,0);
\draw[vine] (0.5,1) .. controls (0.5,0.724) and (0.724,0.5) .. (1,0.5);
\node[below] at (.5,-.08) {even elbow};
\end{scope}
\begin{scope}[xshift=4.65cm]
\draw[gray!60] (0,0) rectangle (1,1);
\draw[chosen] (0,0.5)--(1,0.5);
\draw[chosen] (0.5,0)--(0.5,0.5);
\node[below] at (.5,-.08) {T};
\end{scope}
\begin{scope}[xshift=6.2cm]
\draw[gray!60] (0,0) rectangle (1,1);
\draw[chosen] (0,0.5)--(1,0.5);
\node[below] at (.5,-.08) {horizontal};
\end{scope}
\end{tikzpicture}
  \par\medskip
  \begin{minipage}[c]{.24\textwidth}
    \centering
    \(c(F_0)=(1,0,2)\)\par\medskip
\begin{tikzpicture}[x=.66cm,y=.65cm,
  internal/.style={circle,fill=black,inner sep=1.6pt},
  every node/.style={font=\scriptsize}]
  \draw[gray!55] (.65,0)--(5.35,0);
  \foreach \i in {1,...,5} {
    \node at (\i,0) {\(\times\)};
    \node[below=3pt] at (\i,0) {\(\i\)};
  }
  \node[internal] (a) at (1.5,.85) {};
  \node[internal] (b) at (3.5,.85) {};
  \node[internal] (c) at (4,1.8) {};
  \draw (a)--(1,0) (a)--(2,0)
    (b)--(3,0) (b)--(4,0)
    (c)--(b) (c)--(5,0);
  \node[blue!75!black,above=3pt] at (a) {\(1\)};
  \node[blue!75!black,above left=2pt] at (b) {\(3\)};
  \node[blue!75!black,above=3pt] at (c) {\(4\)};
\end{tikzpicture}
  \end{minipage}\hspace{.02\textwidth}%
  \begin{minipage}[c]{.23\textwidth}
    \centering
    \strut Forest vine\par\smallskip
\begin{tikzpicture}[x=.54cm,y=.54cm,
  vine/.style={black!75,line width=.65pt,line cap=round},
  chosen/.style={blue!75!black,line width=1.15pt,line cap=round},
  every node/.style={font=\scriptsize}]
\path[use as bounding box] (-.8,-7.85) rectangle (4.9,.9);
\fill[blue!10] (0,-1) rectangle (1,0);
\draw[gray!45,line width=.25pt] (0,-1) rectangle (1,0);
\draw[chosen] (0,-0.5)--(1,-0.5);
\draw[chosen] (0.5,-1)--(0.5,-0.5);
\fill[blue!75!black] (0.5,-0.5) circle (1.35pt);
\draw[gray!45,line width=.25pt] (1,-1) rectangle (2,0);
\draw[vine] (1,-0.5) .. controls (1.276,-0.5) and (1.5,-0.276) .. (1.5,0);
\draw[vine] (1.5,-1) .. controls (1.5,-0.724) and (1.724,-0.5) .. (2,-0.5);
\draw[gray!45,line width=.25pt] (2,-1) rectangle (3,0);
\draw[vine] (2,-0.5)--(3,-0.5);
\draw[gray!45,line width=.25pt] (3,-1) rectangle (4,0);
\draw[vine] (3,-0.5)--(4,-0.5);
\fill[gray!18] (0,-2) rectangle (1,-1);
\draw[gray!45,line width=.25pt] (0,-2) rectangle (1,-1);
\draw[vine] (0,-1.5) .. controls (0.276,-1.5) and (0.5,-1.724) .. (0.5,-2);
\draw[vine] (0.5,-1) .. controls (0.5,-1.276) and (0.724,-1.5) .. (1,-1.5);
\draw[gray!45,line width=.25pt] (1,-2) rectangle (2,-1);
\draw[vine] (1,-1.5) .. controls (1.276,-1.5) and (1.5,-1.724) .. (1.5,-2);
\draw[vine] (1.5,-1) .. controls (1.5,-1.276) and (1.724,-1.5) .. (2,-1.5);
\draw[gray!45,line width=.25pt] (2,-2) rectangle (3,-1);
\draw[vine] (2,-1.5)--(3,-1.5);
\fill[blue!10] (3,-2) rectangle (4,-1);
\draw[gray!45,line width=.25pt] (3,-2) rectangle (4,-1);
\draw[chosen] (3,-1.5)--(4,-1.5);
\draw[chosen] (3.5,-2)--(3.5,-1.5);
\fill[blue!75!black] (3.5,-1.5) circle (1.35pt);
\fill[gray!18] (0,-3) rectangle (1,-2);
\draw[gray!45,line width=.25pt] (0,-3) rectangle (1,-2);
\draw[vine] (0,-2.5) .. controls (0.276,-2.5) and (0.5,-2.276) .. (0.5,-2);
\draw[vine] (0.5,-3) .. controls (0.5,-2.724) and (0.724,-2.5) .. (1,-2.5);
\draw[gray!45,line width=.25pt] (1,-3) rectangle (2,-2);
\draw[vine] (1,-2.5) .. controls (1.276,-2.5) and (1.5,-2.276) .. (1.5,-2);
\draw[vine] (1.5,-3) .. controls (1.5,-2.724) and (1.724,-2.5) .. (2,-2.5);
\draw[gray!45,line width=.25pt] (2,-3) rectangle (3,-2);
\draw[vine] (2,-2.5)--(3,-2.5);
\draw[gray!45,line width=.25pt] (3,-3) rectangle (4,-2);
\draw[vine] (3,-2.5) .. controls (3.276,-2.5) and (3.5,-2.276) .. (3.5,-2);
\draw[vine] (3.5,-3) .. controls (3.5,-2.724) and (3.724,-2.5) .. (4,-2.5);
\fill[gray!18] (0,-4) rectangle (1,-3);
\draw[gray!45,line width=.25pt] (0,-4) rectangle (1,-3);
\draw[vine] (0,-3.5) .. controls (0.276,-3.5) and (0.5,-3.724) .. (0.5,-4);
\draw[vine] (0.5,-3) .. controls (0.5,-3.276) and (0.724,-3.5) .. (1,-3.5);
\fill[gray!18] (1,-4) rectangle (2,-3);
\draw[gray!45,line width=.25pt] (1,-4) rectangle (2,-3);
\draw[vine] (1,-3.5) .. controls (1.276,-3.5) and (1.5,-3.724) .. (1.5,-4);
\draw[vine] (1.5,-3) .. controls (1.5,-3.276) and (1.724,-3.5) .. (2,-3.5);
\draw[gray!45,line width=.25pt] (2,-4) rectangle (3,-3);
\draw[vine] (2,-3.5)--(3,-3.5);
\draw[gray!45,line width=.25pt] (3,-4) rectangle (4,-3);
\draw[vine] (3,-3.5) .. controls (3.276,-3.5) and (3.5,-3.724) .. (3.5,-4);
\draw[vine] (3.5,-3) .. controls (3.5,-3.276) and (3.724,-3.5) .. (4,-3.5);
\fill[gray!18] (0,-5) rectangle (1,-4);
\draw[gray!45,line width=.25pt] (0,-5) rectangle (1,-4);
\draw[vine] (0,-4.5) .. controls (0.276,-4.5) and (0.5,-4.276) .. (0.5,-4);
\draw[vine] (0.5,-5) .. controls (0.5,-4.724) and (0.724,-4.5) .. (1,-4.5);
\fill[gray!18] (1,-5) rectangle (2,-4);
\draw[gray!45,line width=.25pt] (1,-5) rectangle (2,-4);
\draw[vine] (1,-4.5) .. controls (1.276,-4.5) and (1.5,-4.276) .. (1.5,-4);
\draw[vine] (1.5,-5) .. controls (1.5,-4.724) and (1.724,-4.5) .. (2,-4.5);
\fill[blue!10] (2,-5) rectangle (3,-4);
\draw[gray!45,line width=.25pt] (2,-5) rectangle (3,-4);
\draw[chosen] (2,-4.5)--(3,-4.5);
\draw[chosen] (2.5,-5)--(2.5,-4.5);
\fill[blue!75!black] (2.5,-4.5) circle (1.35pt);
\draw[gray!45,line width=.25pt] (3,-5) rectangle (4,-4);
\draw[vine] (3,-4.5) .. controls (3.276,-4.5) and (3.5,-4.276) .. (3.5,-4);
\draw[vine] (3.5,-5) .. controls (3.5,-4.724) and (3.724,-4.5) .. (4,-4.5);
\fill[gray!18] (0,-6) rectangle (1,-5);
\draw[gray!45,line width=.25pt] (0,-6) rectangle (1,-5);
\draw[vine] (0,-5.5) .. controls (0.276,-5.5) and (0.5,-5.724) .. (0.5,-6);
\draw[vine] (0.5,-5) .. controls (0.5,-5.276) and (0.724,-5.5) .. (1,-5.5);
\fill[gray!18] (1,-6) rectangle (2,-5);
\draw[gray!45,line width=.25pt] (1,-6) rectangle (2,-5);
\draw[vine] (1,-5.5) .. controls (1.276,-5.5) and (1.5,-5.724) .. (1.5,-6);
\draw[vine] (1.5,-5) .. controls (1.5,-5.276) and (1.724,-5.5) .. (2,-5.5);
\fill[gray!18] (2,-6) rectangle (3,-5);
\draw[gray!45,line width=.25pt] (2,-6) rectangle (3,-5);
\draw[vine] (2,-5.5) .. controls (2.276,-5.5) and (2.5,-5.724) .. (2.5,-6);
\draw[vine] (2.5,-5) .. controls (2.5,-5.276) and (2.724,-5.5) .. (3,-5.5);
\draw[gray!45,line width=.25pt] (3,-6) rectangle (4,-5);
\draw[vine] (3,-5.5) .. controls (3.276,-5.5) and (3.5,-5.724) .. (3.5,-6);
\draw[vine] (3.5,-5) .. controls (3.5,-5.276) and (3.724,-5.5) .. (4,-5.5);
\fill[gray!18] (0,-7) rectangle (1,-6);
\draw[gray!45,line width=.25pt] (0,-7) rectangle (1,-6);
\draw[vine] (0,-6.5) .. controls (0.276,-6.5) and (0.5,-6.276) .. (0.5,-6);
\draw[vine] (0.5,-7) .. controls (0.5,-6.724) and (0.724,-6.5) .. (1,-6.5);
\fill[gray!18] (1,-7) rectangle (2,-6);
\draw[gray!45,line width=.25pt] (1,-7) rectangle (2,-6);
\draw[vine] (1,-6.5) .. controls (1.276,-6.5) and (1.5,-6.276) .. (1.5,-6);
\draw[vine] (1.5,-7) .. controls (1.5,-6.724) and (1.724,-6.5) .. (2,-6.5);
\fill[gray!18] (2,-7) rectangle (3,-6);
\draw[gray!45,line width=.25pt] (2,-7) rectangle (3,-6);
\draw[vine] (2,-6.5) .. controls (2.276,-6.5) and (2.5,-6.276) .. (2.5,-6);
\draw[vine] (2.5,-7) .. controls (2.5,-6.724) and (2.724,-6.5) .. (3,-6.5);
\draw[gray!45,line width=.25pt] (3,-7) rectangle (4,-6);
\draw[vine] (3,-6.5) .. controls (3.276,-6.5) and (3.5,-6.276) .. (3.5,-6);
\draw[vine] (3.5,-7) .. controls (3.5,-6.724) and (3.724,-6.5) .. (4,-6.5);
\draw[vine] (0,-0.5) .. controls (-0.5,-0.5) and (-0.5,-1.5) .. (0,-1.5);
\draw[vine] (0,-2.5) .. controls (-0.5,-2.5) and (-0.5,-3.5) .. (0,-3.5);
\draw[vine] (0,-4.5) .. controls (-0.5,-4.5) and (-0.5,-5.5) .. (0,-5.5);
\draw[vine] (4,-1.5) .. controls (4.5,-1.5) and (4.5,-2.5) .. (4,-2.5);
\draw[vine] (4,-3.5) .. controls (4.5,-3.5) and (4.5,-4.5) .. (4,-4.5);
\draw[vine] (4,-5.5) .. controls (4.5,-5.5) and (4.5,-6.5) .. (4,-6.5);
\draw[vine] (4,-0.5) .. controls (4.276,-0.5) and (4.5,-0.276) .. (4.5,0);
\draw[vine] (0,-6.5) .. controls (-0.276,-6.5) and (-0.5,-6.724) .. (-0.5,-7);
\draw[vine] (1.5,0)--(1.5,0.12);
\draw[vine] (4.5,0)--(4.5,0.12);
\draw[vine] (-0.5,-7)--(-0.5,-7.12);
\node[below] at (-0.5,-7.12) {1};
\draw[vine] (0.5,-7)--(0.5,-7.12);
\node[below] at (0.5,-7.12) {2};
\draw[vine] (1.5,-7)--(1.5,-7.12);
\node[below] at (1.5,-7.12) {3};
\draw[vine] (2.5,-7)--(2.5,-7.12);
\node[below] at (2.5,-7.12) {4};
\draw[vine] (3.5,-7)--(3.5,-7.12);
\node[below] at (3.5,-7.12) {5};
\end{tikzpicture}
  \end{minipage}\hspace{.02\textwidth}%
  \begin{minipage}[c]{.23\textwidth}
    \centering
    \strut Schubert vine\par\smallskip
\begin{tikzpicture}[x=.54cm,y=.54cm,
  vine/.style={black!75,line width=.65pt,line cap=round},
  chosen/.style={blue!75!black,line width=1.15pt,line cap=round},
  every node/.style={font=\scriptsize}]
\path[use as bounding box] (-.8,-7.85) rectangle (4.9,.9);
\fill[blue!10] (0,-1) rectangle (1,0);
\draw[gray!45,line width=.25pt] (0,-1) rectangle (1,0);
\draw[chosen] (0,-0.5)--(1,-0.5);
\draw[chosen] (0.5,-1)--(0.5,0);
\draw[gray!45,line width=.25pt] (1,-1) rectangle (2,0);
\draw[vine] (1,-0.5) .. controls (1.276,-0.5) and (1.5,-0.276) .. (1.5,0);
\draw[vine] (1.5,-1) .. controls (1.5,-0.724) and (1.724,-0.5) .. (2,-0.5);
\draw[gray!45,line width=.25pt] (2,-1) rectangle (3,0);
\draw[vine] (2,-0.5) .. controls (2.276,-0.5) and (2.5,-0.276) .. (2.5,0);
\draw[vine] (2.5,-1) .. controls (2.5,-0.724) and (2.724,-0.5) .. (3,-0.5);
\draw[gray!45,line width=.25pt] (3,-1) rectangle (4,0);
\draw[vine] (3,-0.5) .. controls (3.276,-0.5) and (3.5,-0.276) .. (3.5,0);
\draw[vine] (3.5,-1) .. controls (3.5,-0.724) and (3.724,-0.5) .. (4,-0.5);
\fill[gray!18] (0,-2) rectangle (1,-1);
\draw[gray!45,line width=.25pt] (0,-2) rectangle (1,-1);
\draw[vine] (0,-1.5) .. controls (0.276,-1.5) and (0.5,-1.724) .. (0.5,-2);
\draw[vine] (0.5,-1) .. controls (0.5,-1.276) and (0.724,-1.5) .. (1,-1.5);
\draw[gray!45,line width=.25pt] (1,-2) rectangle (2,-1);
\draw[vine] (1,-1.5) .. controls (1.276,-1.5) and (1.5,-1.724) .. (1.5,-2);
\draw[vine] (1.5,-1) .. controls (1.5,-1.276) and (1.724,-1.5) .. (2,-1.5);
\draw[gray!45,line width=.25pt] (2,-2) rectangle (3,-1);
\draw[vine] (2,-1.5) .. controls (2.276,-1.5) and (2.5,-1.724) .. (2.5,-2);
\draw[vine] (2.5,-1) .. controls (2.5,-1.276) and (2.724,-1.5) .. (3,-1.5);
\draw[gray!45,line width=.25pt] (3,-2) rectangle (4,-1);
\draw[vine] (3,-1.5) .. controls (3.276,-1.5) and (3.5,-1.724) .. (3.5,-2);
\draw[vine] (3.5,-1) .. controls (3.5,-1.276) and (3.724,-1.5) .. (4,-1.5);
\fill[gray!18] (0,-3) rectangle (1,-2);
\draw[gray!45,line width=.25pt] (0,-3) rectangle (1,-2);
\draw[vine] (0,-2.5) .. controls (0.276,-2.5) and (0.5,-2.276) .. (0.5,-2);
\draw[vine] (0.5,-3) .. controls (0.5,-2.724) and (0.724,-2.5) .. (1,-2.5);
\draw[gray!45,line width=.25pt] (1,-3) rectangle (2,-2);
\draw[vine] (1,-2.5) .. controls (1.276,-2.5) and (1.5,-2.276) .. (1.5,-2);
\draw[vine] (1.5,-3) .. controls (1.5,-2.724) and (1.724,-2.5) .. (2,-2.5);
\fill[blue!10] (2,-3) rectangle (3,-2);
\draw[gray!45,line width=.25pt] (2,-3) rectangle (3,-2);
\draw[chosen] (2,-2.5)--(3,-2.5);
\draw[chosen] (2.5,-3)--(2.5,-2);
\draw[gray!45,line width=.25pt] (3,-3) rectangle (4,-2);
\draw[vine] (3,-2.5) .. controls (3.276,-2.5) and (3.5,-2.276) .. (3.5,-2);
\draw[vine] (3.5,-3) .. controls (3.5,-2.724) and (3.724,-2.5) .. (4,-2.5);
\fill[gray!18] (0,-4) rectangle (1,-3);
\draw[gray!45,line width=.25pt] (0,-4) rectangle (1,-3);
\draw[vine] (0,-3.5) .. controls (0.276,-3.5) and (0.5,-3.724) .. (0.5,-4);
\draw[vine] (0.5,-3) .. controls (0.5,-3.276) and (0.724,-3.5) .. (1,-3.5);
\fill[gray!18] (1,-4) rectangle (2,-3);
\draw[gray!45,line width=.25pt] (1,-4) rectangle (2,-3);
\draw[vine] (1,-3.5) .. controls (1.276,-3.5) and (1.5,-3.724) .. (1.5,-4);
\draw[vine] (1.5,-3) .. controls (1.5,-3.276) and (1.724,-3.5) .. (2,-3.5);
\draw[gray!45,line width=.25pt] (2,-4) rectangle (3,-3);
\draw[vine] (2,-3.5) .. controls (2.276,-3.5) and (2.5,-3.724) .. (2.5,-4);
\draw[vine] (2.5,-3) .. controls (2.5,-3.276) and (2.724,-3.5) .. (3,-3.5);
\draw[gray!45,line width=.25pt] (3,-4) rectangle (4,-3);
\draw[vine] (3,-3.5) .. controls (3.276,-3.5) and (3.5,-3.724) .. (3.5,-4);
\draw[vine] (3.5,-3) .. controls (3.5,-3.276) and (3.724,-3.5) .. (4,-3.5);
\fill[gray!18] (0,-5) rectangle (1,-4);
\draw[gray!45,line width=.25pt] (0,-5) rectangle (1,-4);
\draw[vine] (0,-4.5) .. controls (0.276,-4.5) and (0.5,-4.276) .. (0.5,-4);
\draw[vine] (0.5,-5) .. controls (0.5,-4.724) and (0.724,-4.5) .. (1,-4.5);
\fill[gray!18] (1,-5) rectangle (2,-4);
\draw[gray!45,line width=.25pt] (1,-5) rectangle (2,-4);
\draw[vine] (1,-4.5) .. controls (1.276,-4.5) and (1.5,-4.276) .. (1.5,-4);
\draw[vine] (1.5,-5) .. controls (1.5,-4.724) and (1.724,-4.5) .. (2,-4.5);
\draw[gray!45,line width=.25pt] (2,-5) rectangle (3,-4);
\draw[vine] (2,-4.5) .. controls (2.276,-4.5) and (2.5,-4.276) .. (2.5,-4);
\draw[vine] (2.5,-5) .. controls (2.5,-4.724) and (2.724,-4.5) .. (3,-4.5);
\draw[gray!45,line width=.25pt] (3,-5) rectangle (4,-4);
\draw[vine] (3,-4.5) .. controls (3.276,-4.5) and (3.5,-4.276) .. (3.5,-4);
\draw[vine] (3.5,-5) .. controls (3.5,-4.724) and (3.724,-4.5) .. (4,-4.5);
\fill[gray!18] (0,-6) rectangle (1,-5);
\draw[gray!45,line width=.25pt] (0,-6) rectangle (1,-5);
\draw[vine] (0,-5.5) .. controls (0.276,-5.5) and (0.5,-5.724) .. (0.5,-6);
\draw[vine] (0.5,-5) .. controls (0.5,-5.276) and (0.724,-5.5) .. (1,-5.5);
\fill[gray!18] (1,-6) rectangle (2,-5);
\draw[gray!45,line width=.25pt] (1,-6) rectangle (2,-5);
\draw[vine] (1,-5.5) .. controls (1.276,-5.5) and (1.5,-5.724) .. (1.5,-6);
\draw[vine] (1.5,-5) .. controls (1.5,-5.276) and (1.724,-5.5) .. (2,-5.5);
\fill[gray!18] (2,-6) rectangle (3,-5);
\draw[gray!45,line width=.25pt] (2,-6) rectangle (3,-5);
\draw[vine] (2,-5.5) .. controls (2.276,-5.5) and (2.5,-5.724) .. (2.5,-6);
\draw[vine] (2.5,-5) .. controls (2.5,-5.276) and (2.724,-5.5) .. (3,-5.5);
\fill[blue!10] (3,-6) rectangle (4,-5);
\draw[gray!45,line width=.25pt] (3,-6) rectangle (4,-5);
\draw[chosen] (3,-5.5)--(4,-5.5);
\draw[chosen] (3.5,-6)--(3.5,-5);
\fill[gray!18] (0,-7) rectangle (1,-6);
\draw[gray!45,line width=.25pt] (0,-7) rectangle (1,-6);
\draw[vine] (0,-6.5) .. controls (0.276,-6.5) and (0.5,-6.276) .. (0.5,-6);
\draw[vine] (0.5,-7) .. controls (0.5,-6.724) and (0.724,-6.5) .. (1,-6.5);
\fill[gray!18] (1,-7) rectangle (2,-6);
\draw[gray!45,line width=.25pt] (1,-7) rectangle (2,-6);
\draw[vine] (1,-6.5) .. controls (1.276,-6.5) and (1.5,-6.276) .. (1.5,-6);
\draw[vine] (1.5,-7) .. controls (1.5,-6.724) and (1.724,-6.5) .. (2,-6.5);
\fill[gray!18] (2,-7) rectangle (3,-6);
\draw[gray!45,line width=.25pt] (2,-7) rectangle (3,-6);
\draw[vine] (2,-6.5) .. controls (2.276,-6.5) and (2.5,-6.276) .. (2.5,-6);
\draw[vine] (2.5,-7) .. controls (2.5,-6.724) and (2.724,-6.5) .. (3,-6.5);
\draw[gray!45,line width=.25pt] (3,-7) rectangle (4,-6);
\draw[vine] (3,-6.5) .. controls (3.276,-6.5) and (3.5,-6.276) .. (3.5,-6);
\draw[vine] (3.5,-7) .. controls (3.5,-6.724) and (3.724,-6.5) .. (4,-6.5);
\draw[vine] (0,-0.5) .. controls (-0.5,-0.5) and (-0.5,-1.5) .. (0,-1.5);
\draw[vine] (0,-2.5) .. controls (-0.5,-2.5) and (-0.5,-3.5) .. (0,-3.5);
\draw[vine] (0,-4.5) .. controls (-0.5,-4.5) and (-0.5,-5.5) .. (0,-5.5);
\draw[vine] (4,-1.5) .. controls (4.5,-1.5) and (4.5,-2.5) .. (4,-2.5);
\draw[vine] (4,-3.5) .. controls (4.5,-3.5) and (4.5,-4.5) .. (4,-4.5);
\draw[vine] (4,-5.5) .. controls (4.5,-5.5) and (4.5,-6.5) .. (4,-6.5);
\draw[vine] (4,-0.5) .. controls (4.276,-0.5) and (4.5,-0.276) .. (4.5,0);
\draw[vine] (0,-6.5) .. controls (-0.276,-6.5) and (-0.5,-6.724) .. (-0.5,-7);
\draw[vine] (0.5,0)--(0.5,0.12);
\node[above] at (0.5,0.12) {1};
\draw[vine] (1.5,0)--(1.5,0.12);
\node[above] at (1.5,0.12) {2};
\draw[vine] (2.5,0)--(2.5,0.12);
\node[above] at (2.5,0.12) {3};
\draw[vine] (3.5,0)--(3.5,0.12);
\node[above] at (3.5,0.12) {4};
\draw[vine] (4.5,0)--(4.5,0.12);
\node[above] at (4.5,0.12) {5};
\draw[vine] (-0.5,-7)--(-0.5,-7.12);
\node[below] at (-0.5,-7.12) {2};
\draw[vine] (0.5,-7)--(0.5,-7.12);
\node[below] at (0.5,-7.12) {1};
\draw[vine] (1.5,-7)--(1.5,-7.12);
\node[below] at (1.5,-7.12) {4};
\draw[vine] (2.5,-7)--(2.5,-7.12);
\node[below] at (2.5,-7.12) {5};
\draw[vine] (3.5,-7)--(3.5,-7.12);
\node[below] at (3.5,-7.12) {3};
\end{tikzpicture}
  \end{minipage}
  \caption{The five vine tiles of \cite[\S5.3 and \S6.1]{BGNST2025a}
    and the diagrams for \Cref{ex:forest-vine-subword} and \Cref{ex:schubert-vine-subword}.
    Blue T-shaped tiles encode \(\pi_F\), and blue crosses encode \(\pi_w\).
    The left forest has its canonical labels in blue. In the Schubert
    vine, top labels record strand identities and the bottom labels list
    \(w(1),\ldots,w(5)\).}
  \label{fig:vine-examples}
\end{figure}

The following observation translates the order of letters into the
inequalities used in the path bijections.

\begin{lem}\label{lem:vine-letter-order}
Let \(B,B'\) be distinct letters of \(\widetilde\omega_{[n]}\) with values
\(j,i\), respectively. Then
\begin{equation}\label{eq:vine-letter-order}
B\prec B'\quad\Longleftrightarrow\quad
\begin{array}{c|cc}
 &d(B)\text{ odd}&d(B)\text{ even}\\ \hline
 j\geq i&d(B')\geq d(B)&d(B')>d(B)\\
 j<i&d(B')>d(B)&d(B')\geq d(B).
\end{array}
\end{equation}
\end{lem}
\begin{proof}
If \(d(B')>d(B)\), then \(B\prec B'\), whereas if
\(d(B')<d(B)\), then \(B'\prec B\).  If \(d(B')=d(B)\), the letters
belong to the same part of the same block.  Their values decrease in an
unbarred part and increase in a barred part.  Hence \(B\prec B'\) holds
in the equality case precisely when \(d(B)\) is odd and \(j\geq i\), or
when \(d(B)\) is even and \(j<i\).  These cases give
\eqref{eq:vine-letter-order}.
\end{proof}

\subsection{The Grothendieck subword formula}
The Grothendieck subword formula uses the unbarred letters of
\(\widetilde\omega_{[n]}\), with weights specialized at
\(\vec t=\vec0\), and replaces the reduced word condition by a
condition on the Demazure product. For a word
\(a=(a_1,\ldots,a_m)\) in positive integers, set
\(w^{(0)}=\mathrm{id}\) and define the \emph{Demazure product}
recursively by
\[
  w^{(r)}=
  \begin{cases}
    w^{(r-1)}s_{a_r},&\ell(w^{(r-1)}s_{a_r})>\ell(w^{(r-1)}),\\
    w^{(r-1)},&\ell(w^{(r-1)}s_{a_r})<\ell(w^{(r-1)}),
  \end{cases}
  \qquad
  \operatorname{Dem}(a)=w^{(m)}.
\]
A word \(a\) is a \emph{Hecke word} of \(w\) if
\(\operatorname{Dem}(a)=w\).

For \(n\geq1\), define the \emph{staircase word} and its letter
weights by
\[
  \omega_n^+
  =\prod_{q=1}^{n}\bigl(n^{(q)}(n-1)^{(q)}\cdots q^{(q)}\bigr),
  \qquad \wt\bigl(j^{(q)}\bigr)=x_q,
\]
where the product denotes concatenation in increasing order of \(q\).
The restriction of the vine order to the letters of \(\omega_n^+\)
is denoted by \(\prec_+\).
For \(w\in S_{n+1}\), let
\[
  \mathcal R^K(\omega_n^+;w)
  =\{\pi\subseteq\omega_n^+:\operatorname{Dem}(\val(\pi))=w\}.
\]
For \(w\in S_{n+1}\), the \emph{\(\beta\)-Grothendieck polynomial}
\(G_w^{(\beta)}\) of \(w\) is \cite{KM2005}
\begin{equation}\label{eq:grothendieck-hecke-subword}
  G_w^{(\beta)}(\vec x)
  =\sum_{\pi\in\mathcal R^K(\omega_n^+;w)}
     \beta^{|\pi|-\ell(w)}\wt(\pi).
\end{equation}
At \(\beta=0\), only subwords with reduced value words contribute,
giving \(G_w^{(0)}=\mathfrak S_w\). The specialization \(\beta=-1\)
gives the usual Grothendieck polynomial.

\begin{exam}\label{ex:grothendieck-hecke-subword}
Let \(w=21453\), as in \Cref{ex:schubert-vine-subword}. The staircase subword
\[
  \pi_K=(1^{(1)},3^{(2)},4^{(3)},4^{(4)})\subseteq\omega_4^+
\]
has value word \((1,3,4,4)\). The prefix \((1,3,4)\) is a reduced
word of \(w\), and appending the final \(4\) leaves the Demazure
product unchanged. Hence \(\operatorname{Dem}(\val(\pi_K))=w\).
Since \(|\pi_K|=4\) and \(\ell(w)=3\), the subword \(\pi_K\)
contributes \(\beta x_1x_2x_3x_4\) to \(G_w^{(\beta)}\).
\end{exam}

\subsection{Double Schur and double fundamental polynomials}
\label{subsec:schur-fundamental}
A partition of length at most \(n\) is written
\(\lambda=(\lambda_1\geq\cdots\geq\lambda_n\geq0)\), with
\(|\lambda|=\sum_i\lambda_i\). Its Ferrers diagram is
\[
  D(\lambda)=\{(i,j):1\leq i\leq n,\ 1\leq j\leq\lambda_i\}.
\]
Its elements are called \emph{boxes}. A \emph{descent} of a permutation
\(w\) is an index \(i\) with \(w(i)>w(i+1)\). A permutation is
\emph{Grassmannian} if it has at most one descent. Let
\(w_\lambda\in S_{n+\lambda_1}\) have its first \(n\) values given by
\[
  w_\lambda(i)=i+\lambda_{n+1-i}\qquad(1\leq i\leq n),
\]
with the unused values placed increasingly afterward. Its only possible
descent is at \(n\).
For a partition \(\lambda\) with at most \(n\) parts, the
\emph{double Schur polynomial} \(s_\lambda\) of \(\lambda\) is
\cite{LS1981,LS1982}
\begin{equation}\label{eq:grassmannian-double-schur}
  s_\lambda(x_1,\ldots,x_n;\vec t)
  =\mathfrak S_{w_\lambda}(\vec x;\vec t).
\end{equation}
At \(\vec t=\vec0\), this is the Schur polynomial \( s_\lambda(\vec x_n) \).

A \emph{padded composition} of length \(n\) is a sequence
\[
  c=(0^{n-\ell},\gamma_1,\ldots,\gamma_\ell),\qquad
  0\leq\ell\leq n,\quad \gamma_1,\ldots,\gamma_\ell>0.
\]
Let \(\operatorname{Pad}_n\) be the set of these sequences, regarded
as codes by appending zeros. For \(c\in\operatorname{Pad}_n\), the
\emph{double fundamental quasisymmetric polynomial} \(\mathfrak F_c\)
of \(c\) is \cite{BGNST2025a}
\begin{equation}\label{eq:double-fundamental-polynomial}
  \mathfrak F_c(x_1,\ldots,x_n;\vec t)
  :=\PP_{F(c)}(\vec x;\vec t).
\end{equation}
At \(\vec t=\vec0\), this is the fundamental quasisymmetric
polynomial indexed by \((\gamma_1,\ldots,\gamma_\ell)\); for \(c=(0^n)\), it is \(1\).

\subsection{Equivariant quasisymmetry and basis families}

A polynomial \(f\in\mathbb Z[\vec t][x_1,\ldots,x_n]\) is
\emph{equivariantly quasisymmetric} \cite{BGNST2025a} if, for every
\(1\leq i<n\),
\[
  f(x_1,\ldots,x_{i-1},t_i,x_i,\ldots,x_{n-1};\vec t)
  =f(x_1,\ldots,x_i,t_i,x_{i+1},\ldots,x_{n-1};\vec t).
\]
The resulting \(\mathbb Z[\vec t]\)-algebra is denoted by
\(\operatorname{EQSym}_n\).

The basis properties are summarized in \Cref{tab:basis-families}
\cite{LS1982,KM2005,Gessel1984,NST2024,NST2026Grove,BGNST2025a}.
Each entry lists a basis family followed by its space. For the Schubert
and Grothendieck families, \(w\) ranges over permutations whose descents
lie in \([n]\); for the forest and grove families, \(c(F)\in\Codes_n\).
The Schur and fundamental indices range over partitions of length at
most \(n\) and \(c\in\operatorname{Pad}_n\), respectively.
Write \(\operatorname{Sym}_n\) and \(\operatorname{QSym}_n\) for the
ordinary symmetric and quasisymmetric polynomial rings in
\(x_1,\ldots,x_n\), and set
\(\operatorname{ESym}_n=\mathbb Z[\vec t]\otimes_{\mathbb Z}
\operatorname{Sym}_n\).

\begin{table}
\centering
\renewcommand{\arraystretch}{1.25}
\begin{tabular}{@{}lll@{}}
\toprule
 & Schubert side & Quasisymmetric side\\
\midrule
\(\mathbb Z\)-basis
 & \makecell[l]{
     \(\mathfrak S_w(\vec x):\ \mathbb Z[x_1,\ldots,x_n]\)\\[3pt]
     \(s_\lambda(\vec x;\vec0):\ \operatorname{Sym}_n\)}
 & \makecell[l]{
     \(\PP_F(\vec x):\ \mathbb Z[x_1,\ldots,x_n]\)\\[3pt]
  \(\mathfrak F_c(\vec x;\vec0):\ \operatorname{QSym}_n\)}\\
\addlinespace[2.5pt]
  \hline
\addlinespace[2.5pt]
\(\mathbb Z[\beta]\)-basis
 & \(G_w^{(\beta)}(\vec x):\ \mathbb Z[\beta][x_1,\ldots,x_n]\)
                  & \(\widetilde\PP_F^{(\beta)}(\vec x):\ \mathbb Z[\beta][x_1,\ldots,x_n]\)\\
\addlinespace[2.5pt]
  \hline
\addlinespace[2.5pt]
\(\mathbb Z[\vec t]\)-basis
 & \makecell[l]{
     \(\mathfrak S_w(\vec x;\vec t):\ \mathbb Z[\vec t][x_1,\ldots,x_n]\)\\[3pt]
     \(s_\lambda(\vec x;\vec t):\ \operatorname{ESym}_n\)}
 & \makecell[l]{
     \(\PP_F(\vec x;\vec t):\ \mathbb Z[\vec t][x_1,\ldots,x_n]\)\\[3pt]
     \(\mathfrak F_c(\vec x;\vec t):\ \operatorname{EQSym}_n\)}\\
\bottomrule
\end{tabular}
\caption{Basis families on the Schubert and quasisymmetric sides in
ordinary, \(K\)-theoretic, and equivariant settings.}
\label{tab:basis-families}
\end{table}
The double fundamental basis statement is
\cite[Theorem~4.12]{BGNST2025a}. Since
\(\operatorname{ESym}_n\subseteq\operatorname{EQSym}_n\), every double
Schur polynomial has a unique expansion in the double fundamental basis.

\section{A lattice path model and its polynomial}\label{sec:lattice-paths}

In this section, we introduce labeled lattice paths and define their polynomials as weighted
sums over cell selections governed by successor relations among south steps.
We then establish a weight-preserving bijection with a subword model and
describe the specialization at \(\vec t=\vec0\). These constructions provide
the common framework for the forest and Schubert polynomial models in the
next two sections.

\subsection{Lattice paths and successor forests}
\label{subsec:lattice-path-model}
Write \( W=(-2,0) \) and \( S=(0,-1) \) for west and south steps, respectively. For \(n\geq1\), let \( \Path_n \) be the set of lattice paths of the form
\[
  W^\infty S^{c_n}WS^{c_{n-1}}W \cdots S^{c_1}WS^\infty, \qquad
  c_1,\ldots,c_n\in\ZZ_{\geq0}
\]
whose \( W^\infty \) and \( S^\infty \) lie on the \( x \)-axis and the \( y \)-axis, respectively. Set \(\Path=\bigcup_{n\geq1}\Path_n\).

For \(P\in\Path\), its \emph{code} is the sequence
\(c(P)=(c_i)_{i\geq1}\), where \(c_i\) is the number of south steps on
\(x=2i\).
The map \(c:\Path\to\Codes\) is a bijection; denote its inverse by
\(c\mapsto P(c)\). Also, the map restricts to a bijection
\(\Path_n\to\Codes_n\).
Number the south steps of \(P\) with positive \(x\)-coordinate from top
to bottom as
\( S_1(P),\ldots,S_{|c(P)|}(P) \)
and let \( \mathcal S(P)=\{S_1(P),\ldots,S_{|c(P)|}(P)\} \).
Thus \(S_r(P)\) runs from \(y=-r+1\) to \(y=-r\).
We omit \(P\) from the notation when
the path is clear.

A \emph{labeling} of \(P\) is a map
\(\kappa_P:\mathcal S(P)\to\ZZ_{>0}\). We identify \(\kappa_P\)
with the sequence
\[
 \kappa_P=\bigl(\kappa_P(S_1),\ldots,\kappa_P(S_{|c(P)|})\bigr).
\]
We next define two successor relations on \(\mathcal S(P)\).  If
\(S_r\) and \(S_{r+1}\) lie on the same vertical line, then
\(S_{r+1}\) is the \emph{vertical successor} of \(S_r\), and we write
\( S_r\rav S_{r+1} \).
Suppose that \(S_r\) has no vertical successor.  From its endpoint,
follow the ray in direction \((-2,-1)\) toward the \(y\)-axis.  If this
ray meets the endpoint of a south step, let \(S_s\) be the first such step.
Then \(S_s\) is the \emph{diagonal successor} of \(S_r\), and we write
\( S_r\rad S_s \).
If the ray reaches the \(y\)-axis first, then \(S_r\) has no diagonal
successor.

Interpret the successor arrows as child-to-parent relations: if
\(S_r\rav S_s\), then \(S_r\) is the left child of \(S_s\), while if
\(S_r\rad S_s\), then \(S_r\) is the right child of \(S_s\). By construction,
each step has at most one parent and at most one child of each type.
Since every successor has a larger index, the resulting directed graph
is acyclic. Thus each connected component is a rooted plane tree.

For each south step, attach a leaf as its left child if it has no left
child, and as its right child if it has no right child. This turns each
component into a binary plane tree. Take a baseline whose vertices are
labeled \(1,2,\ldots\). If a component has \(\ell\) internal nodes and
its root lies on \(x=2i\), place its \(\ell+1\) leaves, from left to right,
at the baseline vertices \( i,i+1,\ldots,i+\ell \).
Place an isolated leaf at every remaining baseline position. The resulting
indexed forest is the \emph{successor forest} \(F(P)\). Its
internal nodes are the south steps of \(P\). If \( P=W^{\infty}S^{\infty} \), then \( F(P) \) is the empty forest.
\begin{exam}\label{exa:1}
Let \(P=W^\infty SW^2S^2WSWS^\infty\in\Path_4\). Its code is
\[
  c(P)=(1,2,0,1,0,\ldots).
\]
The south steps \(S_1,S_2,S_3,S_4\) lie on
\(x=8,4,4,2\), respectively. We equip \(P\) with the labeling
\(\kappa_P=(3,2,3,4)\). The successor relations are
\(S_2\rav S_3\rad S_4\) and \(S_1\rad S_3\).
In \(F(P)\), \(S_4\) is a root with right child \(S_3\), and \(S_3\)
has left child \(S_2\) and right child \(S_1\); see \Cref{fig:4}.
\end{exam}

\begin{figure}
  \centering
  \begin{tikzpicture}[
  scale=.55,
  every node/.style={font=\footnotesize},
  stepnode/.style={draw, circle, inner sep=.8pt, minimum size=16pt}
]
  \node at (4.5,6.8) {Lattice path \(P\)};
  \node at (15,6.8) {Successor forest \(F(P)\)};

\begin{scope}[yshift=5.6cm]
  \draw[step=1cm, gray!20, thin] (0,-4) grid (8,0);
  \draw[->, gray!70] (-.3,0) -- (10.5,0) node[right] {\(x\)};
  \draw[->, gray!70] (0,-5.5) -- (0,.6) node[above] {\(y\)};
  \foreach \i in {2,4,6,8} {
    \draw[gray!70] (\i,-.08) -- (\i,.08);
    \node[above=2pt] at (\i,0) {\(\i\)};
  }
  \foreach \j in {1,2,3,4} {
    \draw[gray!70] (-.08,-\j) -- (.08,-\j);
    \node[left=2pt] at (0,-\j) {\(-\j\)};
  }
  \draw[->, black!55, semithick, densely dashed,
    shorten <=1pt, shorten >=2pt] (8,-1) -- (4,-3);
  \draw[->, black!55, semithick, densely dashed,
    shorten <=1pt, shorten >=2pt] (4,-3) -- (2,-4);
  \draw[very thick, blue]
    (9.4,0) -- (8,0) -- (8,-1) -- (4,-1)
    -- (4,-2) -- (4,-3) -- (2,-3) -- (2,-4)
    -- (0,-4) -- (0,-4.7);
  \node[blue] at (9.9,0) {\(\cdots\)};
  \node[blue] at (0,-5.15) {\(\vdots\)};
  \node[right=2pt] at (8,-.5) {\(S_1\)};
  \node[right=2pt] at (4,-1.5) {\(S_2\)};
  \node[left=2pt] at (4,-2.5) {\(S_3\)};
  \node[left=2pt] at (2,-3.5) {\(S_4\)};
\end{scope}

\begin{scope}[xshift=12cm]
  \draw[gray] (-1,0) -- (7,0);

  \coordinate (L1) at (0,0);
  \coordinate (L2) at (1.5,0);
  \coordinate (L3) at (3,0);
  \coordinate (L4) at (4.5,0);
  \coordinate (L5) at (6,0);

  \node[stepnode] (S2) at (2.25,1.4) {\(S_2\)};
  \node[stepnode] (S1) at (5.25,1.4) {\(S_1\)};
  \node[stepnode] (S3) at (3.75,4.2) {\(S_3\)};
  \node[stepnode] (S4) at (3.0,5.6) {\(S_4\)};

  \draw (S4) -- (L1);
  \draw (S4) -- (S3);

  \draw (S3) -- (S2);
  \draw (S3) -- (S1);

  \draw (S2) -- (L2);
  \draw (S2) -- (L3);

  \draw (S1) -- (L4);
  \draw (S1) -- (L5);

  \node at (L1) {\(\times\)};
  \node at (L2) {\(\times\)};
  \node at (L3) {\(\times\)};
  \node at (L4) {\(\times\)};
  \node at (L5) {\(\times\)};
  \node[below=4pt] at (L1) {\(1\)};
  \node[below=4pt] at (L2) {\(2\)};
  \node[below=4pt] at (L3) {\(3\)};
  \node[below=4pt] at (L4) {\(4\)};
  \node[below=4pt] at (L5) {\(5\)};
\end{scope}
\end{tikzpicture}
  \caption{The path \(P\) (left) and its successor forest \(F(P)\)
(right) from \Cref{exa:1}. Matching labels identify south steps with
internal nodes; dashed arrows show diagonal successors, and crosses mark leaves.}
  \label{fig:4}
\end{figure}
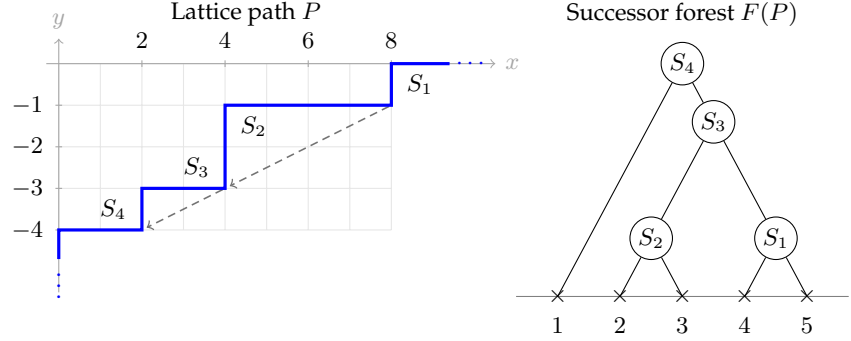

The following
proposition shows that the placements of distinct components do not overlap,
so the successor forest \(F(P)\) is well-defined.

\begin{prop}\label{prop:successor-forest}
For every \(P\in\Path\), the successor forest \(F(P)\) is well-defined.
\end{prop}
\begin{proof}
It suffices to show that the rooted binary plane trees do not overlap.
If \(F(P)\) has fewer than two roots, then it is well-defined. Suppose
there are \(N\geq2\) roots \(S_{r_1},\ldots,S_{r_N}\), where
\(r_1<\cdots<r_N\). Write \(S_r\) on \(x=2h_r\) and set
\(\delta_r=h_r+r\). Since \((h_r)\) is weakly decreasing,
\(\delta_{r+1}\leq\delta_r+1\). If \(S_s\rad S_p\), then \(p\)
is the first later index with \(\delta_p=\delta_s\). Since \(S_s\)
has no vertical successor, \(\delta_{s+1}\leq\delta_s\); the integer
bound therefore gives \(\delta_j<\delta_s\) for \(s<j<p\).

We first claim that the component rooted at \(S_{r_k}\) consists
precisely of \(S_{r_{k-1}+1},\ldots,S_{r_k}\), setting \( r_0=0 \). By definition of the vertical successor and the diagonal successor, every
internal node \(S_s\) in this component satisfies \(s\leq r_k\).
Suppose that there is a diagonal successor relation
\(S_s\rad S_p\) with \(s<r_k<p\). Then
\(\delta_{r_k}<\delta_p\). Since \(S_{r_k}\) has no vertical
successor, \(\delta_{r_k+1}\leq\delta_{r_k}\). The integer bound
therefore forces a first return \(\delta_q=\delta_{r_k}\) with
\(r_k<q<p\), giving \(S_{r_k}\rad S_q\). This contradicts the fact
that \(S_{r_k}\) is a root, and proves the claim.

It follows that the component rooted at \(S_{r_k}\) has
\(r_k-r_{k-1}\) internal nodes. Hence its rightmost leaf is labeled by
\(h_{r_k}+r_k-r_{k-1}\) on the baseline. For \(2\leq k\leq N\),
the same first-return argument, applied to the root \(S_{r_{k-1}}\),
gives \(\delta_{r_k}<\delta_{r_{k-1}}\). Therefore
\(h_{r_k}+r_k-r_{k-1}<h_{r_{k-1}}\). Since \(h_{r_{k-1}}\) is the position of the leftmost leaf of the component rooted at
\(S_{r_{k-1}}\) on the baseline, the two components do not overlap. Thus consecutive components do not overlap, and hence \(F(P)\) is well-defined.
\end{proof}
The component description in the proof of
\Cref{prop:successor-forest} extends to rooted subtrees as follows.

\begin{cor}\label{cor:1}
Let \(P\in\Path\), and suppose that \(F(P)\) has \(N\geq1\)
nontrivial components with roots \(S_{r_1},\ldots,S_{r_N}\), where
\(r_1<\cdots<r_N\). Set \(r_0=0\).
\begin{enumerate}[(i)]
  \item For each \(1\leq k\leq N\), the component rooted at
    \(S_{r_k}\) has internal-node set
    \(\{S_r:r_{k-1}<r\leq r_k\}\).
  \item For each internal node \(S_p\), there is \(q\in[p]\) such that
    the subtree rooted at \(S_p\) has internal-node set
    \(\{S_r:q\leq r\leq p\}\). If \(S_s\rad S_p\), then the right
    and left subtrees of \(S_p\) have internal-node sets
    \[
      \{S_r:q\leq r\leq s\}
      \quad\text{and}\quad
      \{S_r:s<r<p\},
    \]
    respectively.
\end{enumerate}
\end{cor}
\begin{proof}
The assertion about components was proved in
\Cref{prop:successor-forest}. Retain the notation
\(\delta_r=h_r+r\) from its proof. As observed there,
\(S_r\rad S_s\) implies \(\delta_q<\delta_r=\delta_s\) for
\(r<q<s\).

We claim that there are no two successor relations from \(S_p\) to
\(S_r\) and from \(S_q\) to \(S_s\) with \(p<q<r<s\). A vertical
successor relation joins steps with consecutive indices, so both relations would have to be
diagonal. The relation \(S_p\rad S_r\) would give
\(\delta_q<\delta_p=\delta_r\), while \(S_q\rad S_s\) would give
\(\delta_r<\delta_q=\delta_s\), which is a contradiction. Thus the claim holds.

We proceed by induction on the depth from a component root. At depth zero,
if \(S_p=S_{r_k}\), then the proof of \Cref{prop:successor-forest} gives
\(\{S_r:r_{k-1}<r\leq p\}\) as the internal-node set of its rooted
subtree. Now assume that the assertion holds at some depth, and let
\(\{S_q,\ldots,S_p\}\) be the internal-node set of a subtree rooted at
a node \(S_p\) of that depth. We show that the assertion also holds for
its internal children. If \(S_p\) has no internal right child, then all
the other internal nodes belong to its left subtree, whose internal-node
set is \(\{S_q,\ldots,S_{p-1}\}\).

Suppose instead that \(S_s\rad S_p\). Since successor relations strictly increase
indices, every node in the right subtree has index at most \(s\). On the
other hand, no \(S_r\) with \(q\leq r<s\) can lie in the left subtree.
Indeed, its path to the root \(S_{p-1}\) of that subtree would contain a
successor relation from some \(S_{r'}\) to \(S_{s'}\) with \(r'<s<s'<p\),
contradicting the claim above together with \(S_s\rad S_p\). Hence the
right and left subtrees have internal-node sets
\(\{S_q,\ldots,S_s\}\) and \(\{S_{s+1},\ldots,S_{p-1}\}\),
respectively. Both are intervals, completing the induction.
\end{proof}

\subsection{Cell selections and the lattice path polynomial}

Fix a path \(P\in\Path\) with labeling \(\kappa_P\).
We use the unit-square grid in the fourth quadrant and index cells by
column first and row second.  For \(d,r\in\ZZ_{>0}\), let
\[
  C_{d,r}:=[d-1,d]\times[-r,-r+1].
\]
The \(d\)th \emph{column} is the collection
\(\{C_{d,r}:r\geq1\}\), and the \(r\)th \emph{row} is the collection
\(\{C_{d,r}:d\geq1\}\).  Thus the columns are numbered from left to right,
starting at the \(y\)-axis, and the rows are numbered from top to bottom,
starting at the \(x\)-axis. The step \(S_r\) is the unique south step in row \(r\), so
we call \(\kappa_P(S_r)\) the label of that row.
Let \(\mathcal A(P)\) be the set of unit cells above \(P\).  For a
labeling \(\kappa_P\), define the weight of
\(C_{d,r}\in\mathcal{A}(P)\) by
\[
  \wt_{\kappa_P}(C_{d,r})=
  \begin{cases}
    x_{(d+1)/2}-t_{\kappa_P(S_r)},
      & d\leq2\kappa_P(S_r)\text{ and }d\text{ is odd},\\
    t_{\kappa_P(S_r)}-t_{d/2},
      & d\leq2\kappa_P(S_r)\text{ and }d\text{ is even}, \\
    0,
      & d>2\kappa_P(S_r).
\end{cases}
\]
For the labeled path in \Cref{exa:1}, the cell weights are shown
in \Cref{fig:1}.

\begin{figure}
    \centering
    \begin{tikzpicture}[
    scale=0.475,
    dyckpath/.style={very thick, blue},
    tree/.style={blue, line width=2.4pt},
    tick/.style={black, line width=1.7pt},
    lab/.style={red, font=\large},
    orangept/.style={circle, fill=orange, draw=orange, inner sep=1.3pt},
    redpt/.style={circle, fill=red!55!black, draw=red!55!black, inner sep=1.3pt},
    xx/.style={fill=purple!20},
    tt/.style={fill=orange!20}
]
\begin{scope}[xshift=-11cm]
\fill[xx] (0,1) rectangle (1,5);
\fill[tt] (1,1) rectangle (2,5);
\fill[xx] (2,2) rectangle (3,5);
\fill[tt] (3,2) rectangle (4,5);
\fill[xx] (4,4) rectangle (5,5);
\fill[tt] (5,4) rectangle (6,5);
\fill[xx] (6,4) rectangle (7,5);
\fill[tt] (7,4) rectangle (8,5);
  \draw[step=1cm, gray!35, thin] (0,0) grid (10,5);
\draw[dyckpath]
    (10,5)
    -- ++(-2,0)
    -- ++(0,-1)
    -- ++(-2,0)
    -- ++(-2,0)
    -- ++(0,-1)
    -- ++(0,-1)
    -- ++(-2,0)
    -- ++(0,-1)
    -- ++(-2,0)
    -- ++(0,-1);
    
\node at (-2,4.5) {\footnotesize\( \kappa_P(S_1)=3 \)};
\node at (-2,3.5) {\footnotesize\( \kappa_P(S_2)=2 \)};
\node at (-2,2.5) {\footnotesize\( \kappa_P(S_3)=3 \)};
\node at (-2,1.5) {\footnotesize\( \kappa_P(S_4)=4 \)};
\end{scope}
\begin{scope}[xshift=-.5cm ,yshift=-.5cm]
  \fill[xx] (0,.5) rectangle (2,1.5);
  \node at (1,1) {\footnotesize \( x_1-t_4 \)};
  \fill[xx] (0,2) rectangle (2,3);
  \node at (1,2.5) {\footnotesize \( x_1-t_3 \)};
  \fill[xx] (0,3.5) rectangle (2,4.5);
  \node at (1,4) {\footnotesize \( x_1-t_2 \)};
  \fill[xx] (0,5) rectangle (2,6);
  \node at (1,5.5) {\footnotesize \( x_1-t_3 \)};
  \fill[tt] (2.5,.5) rectangle (4.5,1.5);
  \node at (3.5,1) {\footnotesize \( t_4-t_1 \)};
  \fill[tt] (2.5,2) rectangle (4.5,3);
  \node at (3.5,2.5) {\footnotesize \( t_3-t_1 \)};
  \fill[tt] (2.5,3.5) rectangle (4.5,4.5);
  \node at (3.5,4) {\footnotesize \( t_2-t_1 \)};
  \fill[tt] (2.5,5) rectangle (4.5,6);
  \node at (3.5,5.5) {\footnotesize \( t_3-t_1 \)};
  
  \fill[xx] (5,2) rectangle (7,3);
  \node at (6,2.5) {\footnotesize \( x_2-t_3 \)};
  \fill[xx] (5,3.5) rectangle (7,4.5);
  \node at (6,4) {\footnotesize \( x_2-t_2 \)};
  \fill[xx] (5,5) rectangle (7,6);
  \node at (6,5.5) {\footnotesize \( x_2-t_3 \)};
  \fill[tt] (7.5,2) rectangle (9.5,3);
  \node at (8.5,2.5) {\footnotesize \( t_3-t_2 \)};
  \fill[tt] (7.5,3.5) rectangle (9.5,4.5);
  \node at (8.5,4) {\footnotesize \( t_2-t_2 \)};
  \fill[tt] (7.5,5) rectangle (9.5,6);
  \node at (8.5,5.5) {\footnotesize \( t_3-t_2 \)};
  \fill[xx] (10,5) rectangle (12,6);
  \node at (11,5.5) {\footnotesize \( x_3-t_3 \)};
  \fill[tt] (12.5,5) rectangle (14.5,6);
  \node at (13.5,5.5) {\footnotesize \( t_3-t_3 \)};
  \fill[xx] (15,5) rectangle (17,6);
  \node at (16,5.5) {\footnotesize \( 0\)};
  \fill[tt] (17.5,5) rectangle (19.5,6);
  \node at (18.5,5.5) {\footnotesize \( 0 \)};
\end{scope}
\end{tikzpicture}
    \caption{The path \(P=W^\infty SW^2S^2WSWS^\infty\) labeled by
      \(\kappa_P=(3,2,3,4)\) and the corresponding weights
      \(\wt_{\kappa_P}(C_{d,r})\).}
    \label{fig:1}
  \end{figure}
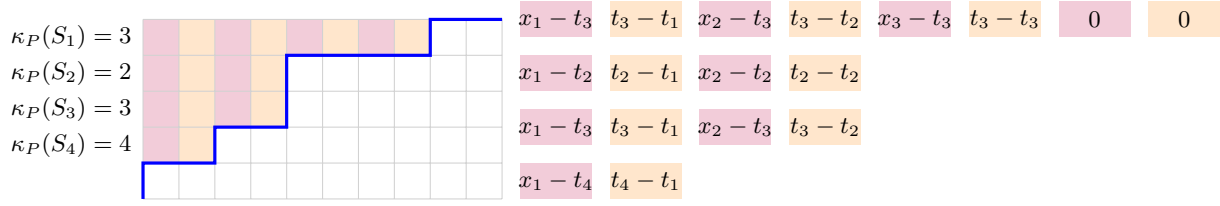

Let \(Q\subseteq\mathcal A(P)\) contain exactly one cell in each row above
\(P\), and write \(\col_r(Q)=d\) when \(C_{d,r}\in Q\). We call \(Q\) a
\emph{cell selection} of \( P \) if every successor relation between \(S_r\) and
\(S_s\) satisfies the corresponding inequality in the following table:
\begin{equation}\label{arr:csel_cond}
\begin{array}{c|cc}
 & \col_s(Q)\text{ odd} & \col_s(Q)\text{ even}\\
\hline
S_r\rav S_s
  & \col_r(Q)\geq\col_s(Q)
  & \col_r(Q)>\col_s(Q)\\
S_r\rad S_s
  & \col_r(Q)>\col_s(Q)
  & \col_r(Q)\geq\col_s(Q).
\end{array}
\end{equation}
Let \(\CSel(P)\) denote the set of cell selections of \(P\).
For a cell selection \(Q\), set
\[
  \wt_{\kappa_P}(Q):=\prod_{C\in Q}\wt_{\kappa_P}(C).
\]
\begin{defn}\label{def:1}
  The \emph{lattice path polynomial} of a labeled path
  \((P,\kappa_P)\) is
  \[
    \FL_{P,\kappa_P}(\vec x;\vec t)
    :=\sum_{Q\in\CSel(P)}\wt_{\kappa_P}(Q).
  \]
\end{defn}
\begin{exam}\label{exa:3}
 For the labeled path \((P,\kappa_P)\) in \Cref{exa:1}, exactly \(12\)
cell selections have nonzero weight. As shown in \Cref{fig:2}, they split
into two families. In each diagram, one cell is chosen from each blue-outlined
block. The left diagram corresponds to
\(2\cdot1\cdot1\cdot4=8\) selections, whose total contribution is
\begin{equation}\label{eq:1}
  \bigl\{(x_1-t_4)+(t_4-t_1)\bigr\}(t_3-t_1)(x_2-t_2)
  \bigl\{(t_3-t_1)+(x_2-t_3)+(t_3-t_2)+(x_3-t_3)\bigr\}.
\end{equation}
The right diagram corresponds to the remaining
\(2\cdot1\cdot1\cdot2=4\) selections, whose total contribution is
\begin{equation}\label{eq:2}
    \bigl\{(x_1-t_4)+(t_4-t_1)\bigr\}(x_2-t_3)(x_2-t_2)
   \bigl\{(t_3-t_2)+(x_3-t_3)\bigr\}.
\end{equation}
Adding \eqref{eq:1} and \eqref{eq:2} and simplifying yields
\[
  \FL_{P,\kappa_P}(\vec{x};\vec{t})
  =(x_1-t_1)(x_2-t_1)(x_2-t_2)(x_3+t_3-t_1-t_2).
\]
  
  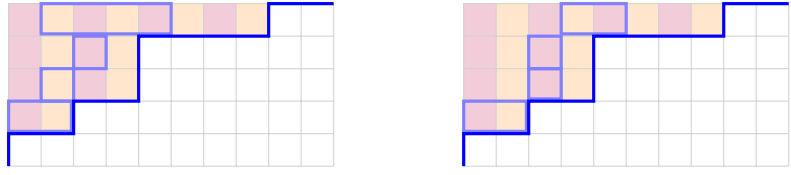
\begin{figure}
    \centering
    \begin{tikzpicture}[
    scale=0.43,
    dyckpath/.style={very thick, blue},
    tree/.style={blue, line width=2.4pt},
    tick/.style={black, line width=1.7pt},
    lab/.style={red, font=\large},
    xx/.style={fill=purple!20},
    tt/.style={fill=orange!20},
    sel/.style={very thick, blue!50}
]
\begin{scope}[xshift=-14cm]
\fill[xx] (0,1) rectangle (1,5);
\fill[tt] (1,1) rectangle (2,5);
\fill[xx] (2,2) rectangle (3,5);
\fill[tt] (3,2) rectangle (4,5);
\fill[xx] (4,4) rectangle (5,5);
\fill[tt] (5,4) rectangle (6,5);
\fill[xx] (6,4) rectangle (7,5);
\fill[tt] (7,4) rectangle (8,5);
  \draw[step=1cm, gray!35, thin] (0,0) grid (10,5);
\draw[sel] (0,1.07) rectangle (1.93,2);
\draw[sel] (1,2) rectangle (2,3);
\draw[sel] (2,3) rectangle (3,4);
\draw[sel] (1,4.07) rectangle (5,5);
\draw[dyckpath]
    (10,5)
    -- ++(-2,0)
    -- ++(0,-1)
    -- ++(-2,0)
    -- ++(-2,0)
    -- ++(0,-1)
    -- ++(0,-1)
    -- ++(-2,0)
    -- ++(0,-1)
    -- ++(-2,0)
    -- ++(0,-1);
\end{scope}
\begin{scope}
\fill[xx] (0,1) rectangle (1,5);
\fill[tt] (1,1) rectangle (2,5);
\fill[xx] (2,2) rectangle (3,5);
\fill[tt] (3,2) rectangle (4,5);
\fill[xx] (4,4) rectangle (5,5);
\fill[tt] (5,4) rectangle (6,5);
\fill[xx] (6,4) rectangle (7,5);
\fill[tt] (7,4) rectangle (8,5);
  \draw[step=1cm, gray!35, thin] (0,0) grid (10,5);
\draw[sel] (0,1.07) rectangle (1.93,2);
\draw[sel] (2,2.07) rectangle (3,3);
\draw[sel] (2,3) rectangle (3,4);
\draw[sel] (3,4.07) rectangle (5,5);
\draw[dyckpath]
    (10,5)
    -- ++(-2,0)
    -- ++(0,-1)
    -- ++(-2,0)
    -- ++(-2,0)
    -- ++(0,-1)
    -- ++(0,-1)
    -- ++(-2,0)
    -- ++(0,-1)
    -- ++(-2,0)
    -- ++(0,-1);
\end{scope}
\end{tikzpicture}
    \caption{The two families of nonzero cell selections for the labeled
      path in \Cref{exa:1}.  A selection chooses one cell from each
      blue-outlined block.  The left diagram encodes eight selections and the
      right diagram the remaining four.}
    \label{fig:2}
  \end{figure}
\end{exam}
\subsection{A subword--path bijection}
\label{subsec:common-subword-path}

Fix \(n\geq1\). Let \(P\in\Path\), put \(m=|\mathcal S(P)|\), and write
\(\kappa_r=\kappa_P(S_r)\in[n]\).
Let \(\mathcal H(P)=([m],\prec_P)\) be the poset in which \(s\) covers
\(r\) if and only if \(S_r\rav S_s\) or \(S_r\rad S_s\).
We define
\[
  \CSel^\times(P,\kappa_P)
  =\{Q\in\CSel(P):\wt_{\kappa_P}(Q)\neq0\}.
\]
We will use labelings with the following properties:
\begin{enumerate}[(i)]
  \item \(\kappa_r<\kappa_s\) if \(S_r\rav S_s\), and
    \(\kappa_r>\kappa_s\) if \(S_r\rad S_s\);
  \item If \(S_r\) is the topmost south step on \(x=2i\), then
    \(\kappa_r\leq i\);
  \item rows with equal labels are comparable in \(\mathcal H(P)\).
\end{enumerate}
For a word \(a=(a_1,\ldots,a_m)\), write
\(\rev(a)=(a_m,\ldots,a_1)\). Define the corresponding subword model by
\[
  \mathcal R(P,\kappa_P)
  =\left\{\pi\subseteq\widetilde\omega_{[n]}:
    \begin{array}{l}
      \rev(\val(\pi))=(\kappa_{\theta(1)},\ldots,
                          \kappa_{\theta(m)})\\
      \text{for a linear extension }(\theta(1),\ldots,\theta(m))
        \text{ of }\mathcal H(P)
    \end{array}\right\}.
\]
By (iii), equal labels occur in every linear extension in their
row order, so \(\theta\) is uniquely determined by \(\pi\).
For \(\pi=(\pi_1,\ldots,\pi_m)\in\mathcal R(P,\kappa_P)\), put
\(B_{\theta(k)}=\pi_{m-k+1}\) and define
\[
  \Phi_{P,\kappa_P}(\pi)
  =\{C_{d(B_r),r}:r\in[m]\}.
\]

\Cref{fig:common-weight-preserving-bijection} shows the correspondence
for \(n=4\), \(c(P)=(0,0,2,1)\), \(\kappa_P=(4,3,4)\), and
\(\pi=(4^{(2)},3^{(2)},\overline4^{(2)})\).
Here \(S_1\rad S_2\rav S_3\), so \(\theta=(1,2,3)\).

\begin{figure}
  \centering
\begin{minipage}[c]{.46\textwidth}
\centering
\begin{tikzpicture}[x=.56cm,y=.76cm,every node/.style={font=\scriptsize},
  path/.style={blue!70!black,line width=1.25pt,line cap=round},
  gridline/.style={gray!45,line width=.3pt}]
\foreach \r/\H/\selected/\rowlabel/\rowcolor in
  {1/8/4/4/orange!80!black,2/6/3/3/blue!70!black,3/6/3/4/teal!70!black} {
  \fill[gray!8] (0,-\r) rectangle (8,1-\r);
  \fill[white] (0,-\r) rectangle (\H,1-\r);
  \foreach \j in {0,...,7}
    \draw[gridline] (\j,-\r) rectangle (\j+1,1-\r);
  \fill[fill=\rowcolor!12]
    (\selected-1,-\r) rectangle (\selected,1-\r);
  \draw[draw=\rowcolor,line width=1.1pt]
    (\selected-1,-\r) rectangle (\selected,1-\r);
  \node[text=\rowcolor] at (\selected-.5,.5-\r)
    {\(C_{\selected,\r}\)};
  \draw[draw=\rowcolor,densely dotted]
    (\selected,.5-\r)--(\H,.5-\r);
  \node[right,text=\rowcolor] at (\H+.12,.5-\r)
    {\(S_{\r}:\,\rowlabel\)};
  \node[left,text=gray] at (-.12,.5-\r) {\(\r\)};
}
\foreach \j in {1,...,8}
  \node[above,text=gray] at (\j-.5,.08) {\(\j\)};
\node[above left,text=gray] at (-.12,.08) {\(r\backslash d\)};
\draw[path] (8,0)--(8,-1)--(6,-1)--(6,-3)--(0,-3);
\end{tikzpicture}
\end{minipage}\hfill
\begin{minipage}[c]{.14\textwidth}
\centering
\begin{tikzpicture}[every node/.style={font=\scriptsize}]
\draw[<-,semithick] (0,.2)--(1.55,.2)
  node[midway,above] {\(\Phi_{P,\kappa_P}\)};
\draw[->,semithick] (0,-.2)--(1.55,-.2)
  node[midway,below] {\(\Phi_{P,\kappa_P}^{-1}\)};
\end{tikzpicture}
\end{minipage}\hfill
\begin{minipage}[c]{.34\textwidth}
\centering
\begin{tikzpicture}[x=1.00cm,y=.50cm,every node/.style={font=\scriptsize}]
\foreach \q in {1,...,4} {
  \pgfmathtruncatemacro{\d}{2*\q-1}
  \foreach \j in {\q,...,4} {
    \fill[purple!7] (\j-1,-\d) rectangle (\j,1-\d);
    \draw[gray!45,line width=.3pt] (\j-1,-\d) rectangle (\j,1-\d);
    \def\lettercolor{black!55}
    \ifnum\q=2\relax
      \ifnum\j=3\relax
        \def\lettercolor{blue!70!black}
        \fill[blue!70!black!12] (\j-1,-\d) rectangle (\j,1-\d);
        \draw[blue!70!black,line width=1.1pt]
          (\j-1,-\d) rectangle (\j,1-\d);
      \fi
      \ifnum\j=4\relax
        \def\lettercolor{teal!70!black}
        \fill[teal!70!black!12] (\j-1,-\d) rectangle (\j,1-\d);
        \draw[teal!70!black,line width=1.1pt]
          (\j-1,-\d) rectangle (\j,1-\d);
      \fi
    \fi
    \node[text=\lettercolor] at (\j-.5,.5-\d) {\(\j^{(\q)}\)};
  }
  \ifnum\q<4\relax
    \pgfmathtruncatemacro{\d}{2*\q}
    \pgfmathtruncatemacro{\firstbar}{\q+1}
    \foreach \j in {\firstbar,...,4} {
      \fill[orange!9] (\j-1,-\d) rectangle (\j,1-\d);
      \draw[gray!45,line width=.3pt] (\j-1,-\d) rectangle (\j,1-\d);
      \def\lettercolor{black!55}
      \ifnum\q=2\relax
        \ifnum\j=4\relax
          \def\lettercolor{orange!80!black}
          \fill[orange!80!black!12] (\j-1,-\d) rectangle (\j,1-\d);
          \draw[orange!80!black,line width=1.1pt]
            (\j-1,-\d) rectangle (\j,1-\d);
        \fi
      \fi
      \node[text=\lettercolor] at (\j-.5,.5-\d)
        {\(\overline\j^{(\q)}\)};
    }
  \fi
}
\draw[semithick] (0,0)--(4,0)--(4,-7)--(3,-7)--(3,-5)
  --(2,-5)--(2,-3)--(1,-3)--(1,-1)--(0,-1)--cycle;
\foreach \d in {1,...,7}
  \node[left,text=gray] at (-.12,.5-\d) {\(\d\)};
\foreach \j in {1,...,4}
  \node[above,text=gray] at (\j-.5,.08) {\(\j\)};
\node[above left,text=gray] at (-.12,.08) {\(d\backslash j\)};
\end{tikzpicture}
\end{minipage}
\par\medskip
\((B_1,B_2,B_3)=(\overline4^{(2)},3^{(2)},4^{(2)})\)
\par\smallskip
\(\wt(\pi)=\wt_{\kappa_P}(Q)
  =(t_4-t_2)(x_2-t_3)(x_2-t_4)\)
\caption{The bijection \(\Phi_{P,\kappa_P}\) for \(n=4\).
Matching colors pair letters with selected cells; the inverse reads the
letters in vine order.}
\label{fig:common-weight-preserving-bijection}
\end{figure}

\begin{lem}\label{lem:subword-path-bijection}
Under assumptions \textup{(i)--(iii)}, the map
\(\Phi_{P,\kappa_P}\) is a weight-preserving bijection from
\(\mathcal R(P,\kappa_P)\) to \(\CSel^\times(P,\kappa_P)\).
Consequently,
\[
  \FL_{P,\kappa_P}(\vec x;\vec t)
  =\sum_{\pi\in\mathcal R(P,\kappa_P)}\wt(\pi).
\]
\end{lem}
\begin{proof}
For \(\pi\in\mathcal R(P,\kappa_P)\), let \(B_r\) denote the letter
assigned to row \(r\). Then \(\val(B_r)=\kappa_r\), and
\(B_s\prec B_r\) whenever \(S_r\) has successor \(S_s\).
By (i) and \Cref{lem:vine-letter-order}, these letter inequalities
are equivalent to \eqref{arr:csel_cond}.
To check that the selected cell in row \(r\) lies above \(P\),
let \(S_t\) be the topmost south step on the vertical line \(x=2i\)
containing \(S_r\). The vertical chain from \(S_t\) to \(S_r\) and (ii) give
\[
  d(B_r)\leq d(B_t)<2\kappa_t\leq2i.
\]
Moreover, \(d(B_r)<2\kappa_r\) makes its weight nonzero. Hence
\(\Phi_{P,\kappa_P}(\pi)\) is a nonzero cell selection of \(P\).

Conversely, let \(Q\in\CSel^\times(P,\kappa_P)\).
Nonzero weight gives \(\col_r(Q)<2\kappa_r\), so there is a unique
letter \(B_r\in\widetilde\omega_{[n]}\) of value \(\kappa_r\)
with \(d(B_r)=\col_r(Q)\).
Rows related by a successor relation have different labels by (i),
so \Cref{lem:vine-letter-order} and \eqref{arr:csel_cond} give
\(B_s\prec B_r\) along every successor relation.
If \(\kappa_r\neq\kappa_s\), then \(B_r\neq B_s\).
If \(\kappa_r=\kappa_s\) with \(r<s\), then \(r\prec_P s\) by (iii).
A successor chain from \(S_r\) to \(S_s\) gives \(B_s\prec B_r\).
Thus all the letters \(B_r\) are distinct.
Read the \(B_r\) in vine order to obtain a subword \(\pi\).
Their reversed order is a linear extension of \(\mathcal H(P)\),
so \(\pi\in\mathcal R(P,\kappa_P)\).
Both constructions retain the row label and the column index of each
letter, and (iii) fixes the matching of equal occurrences. They are
therefore inverse. Finally,
\(\wt_{\kappa_P}(C_{d(B_r),r})=\wt(B_r)\), proving weight
preservation and the summation identity. The empty path and subword
satisfy the same assertions by convention.
\end{proof}

\subsection{\texorpdfstring{The specialization \(\vec t=\vec0\)}
{The specialization t=0}}
\label{subsec:ordinary-lattice-path-polynomial}

At \(\vec t=\vec0\), cells in even columns have weight zero.
For \(P\in\Path\), replace each west step by \((-1,0)\) to obtain
\(P^\circ\), retaining its south steps and their labels \(\kappa_P\).
The compression \(C_{2q-1,r}\mapsto C_{q,r}\) preserves vertical
successors and sends rays of direction \((-2,-1)\) to rays of
direction \((-1,-1)\), so diagonal successors also correspond.
Write \(C_{q,r}\) for the unit cell in column \(q\) and row \(r\)
of the compressed grid.
Its weight at \(\vec t=\vec0\) is
\begin{equation}\label{eq:compressed-cell-weight}
  \wt_{\kappa_P}(C_{q,r})
  :=
  \begin{cases}
    x_q, & q\leq\kappa_P(S_r),\\
    0,   & q>\kappa_P(S_r).
  \end{cases}
\end{equation}

Let \(Q^\circ\subseteq\mathcal A(P^\circ)\) contain
exactly one cell in each row and write \(\col_r(Q^\circ)=q\) when
\(C_{q,r}\in Q^\circ\). We call \(Q^\circ\) a \emph{cell selection} of
\(P^\circ\) if
\begin{equation}\label{eq:compressed-cell-inequ}
\begin{array}{ll}
  \col_r(Q^\circ)\geq\col_s(Q^\circ),
    & \text{if }S_r\rav S_s,\\
  \col_r(Q^\circ)>\col_s(Q^\circ),
    & \text{if }S_r\rad S_s.
\end{array}
\end{equation}
Let \(\CSel(P^\circ)\) denote the set of cell selections of \(P^\circ\).

\begin{defn}\label{def:ordinary-lattice-path-polynomial}
The \emph{lattice path polynomial} of the labeled path
\((P^\circ,\kappa_P)\) is
\[
  \FL_{P^\circ,\kappa_P}(\vec x)
  :=
  \sum_{Q^\circ\in\CSel(P^\circ)}\wt_{\kappa_P}(Q^\circ).
\]
\end{defn}

\begin{prop}\label{prop:compressed-specialization}
For every labeled path \((P,\kappa_P)\),
\[
  \FL_{P,\kappa_P}(\vec x;\vec0)
  =
  \FL_{P^\circ,\kappa_P}(\vec x).
\]
\end{prop}
\begin{proof}
The odd-column compression is a weight-preserving bijection with
\(\CSel(P^\circ)\), by \eqref{arr:csel_cond} and
\eqref{eq:compressed-cell-inequ}; selections using even columns have
weight zero.
\end{proof}
\section{A lattice path model for forest polynomials}
\label{sec:forest-path-models}

For every \(P\in\Path\), we define a canonical labeling \(\kappacan_P\)
and prove
\(\FL_{P,\kappacan_P}(\vec x;\vec t)=\PP_{F(P)}(\vec x;\vec t)\).
We also express forest and grove polynomials as weighted sums over
cell selections and multi-cell selections of \(P^\circ\), respectively.

\subsection{A lattice path model for double forest polynomials}

For \(P\in\Path\), define its \emph{canonical labeling}
\(\kappacan_P\) recursively.  If \(S_r\) is the topmost south step on
\(x=2h\), set \(\kappacan_P(S_r)=h\). Otherwise \(S_{r-1}\rav S_r\), and \( \kappacan_P(S_r) \) is the smallest positive integer greater than \( \kappacan_P(S_{r-1}) \) that has not appeared among
\( \kappacan_P(S_1),\ldots, \kappacan_P(S_{r-1}) \).
\Cref{lem:forest-path-incidence} below shows that the canonical labeling
of \(P\) agrees with the canonical labeling of \(F(P)\)
under the correspondence between south steps and internal nodes.

For an indexed forest \(F\), we define
\( P(F):=P(c(F)) \). Fix an indexed forest \(F\) and write
\(\ell_F(v)\) for the canonical label of \(v\in\IN(F)\). For each
\(h\geq1\), order the internal nodes with \(\rho_F(v)=h\) by increasing
\(\ell_F(v)\), and match them with the south steps of \(P=P(F)\) on \(x=2h\)
from top to bottom. This defines a map
\[
  S_F:\IN(F)\longrightarrow\mathcal S(P).
\]
When \(F\) is fixed, we write \(S(v)\) for \(S_F(v)\).

\begin{lem}\label{lem:forest-path-incidence}
For every indexed forest \(F\), the map
\(S_F:\IN(F)\to\mathcal S(P(F))\) is a bijection satisfying
\[
  \kappacan_{P(F)}(S_F(v))=\ell_F(v)
  \qquad (v\in\IN(F)),
\]
and it identifies \(F(P(F))\) with \(F\) as indexed forests. In
particular, the maps \(P:\For\to\Path\) and \(F:\Path\to\For\) are inverse
bijections.
\end{lem}
\begin{proof}
Put \(P=P(F)\).
Write \(c(F)=(c_i)_{i\geq1}\). For each \(h\), there are \(c_h\)
internal nodes \(v\) with \(\rho_F(v)=h\) and \(c_h\) south steps of
\(P\) on \(x=2h\). Hence \(S_F\) is a bijection.

We next show that \(F(P)=F\). Let \(S_r\) lie on \(x=2h_r\).
For every successor relation from \(S_r\) to \(S_s\),
\[
  \rho_{F(P)}(S_r)-\rho_{F(P)}(S_s)
  =h_r-h_s.
\]
This is immediate for a vertical successor. For a diagonal successor,
both sides equal \(s-r\), since \Cref{cor:1} shows that the left subtree
of \(S_s\) has \(s-r\) leaves.
For each root \(S_r\) of a component, by construction of \(F(P)\), we have
\(\rho_{F(P)}(S_r)=h_r\). If \(S_r\) has successor \(S_s\) and
\(\rho_{F(P)}(S_s)=h_s\), the preceding equality gives
\(\rho_{F(P)}(S_r)=h_r\). Induction from the roots to their
descendants proves this equality for every internal node.
Consequently, \(c(F(P))=c(P)=c(F)\), and the code bijection
\cite[Proposition~3.3]{NT2024} gives \(F(P)=F\).
Since \(P:\For\to\Path\) is a bijection, the maps \(P\) and \(F\) are
inverse.

It remains to compare the canonical labels. We use induction on the row
index \(r\). Let \(v\) satisfy \(S(v)=S_r\), and set
\(h=\rho_F(v)\). If \(S_r\) is the topmost south step on \(x=2h\),
then \(v\) is the bottom node of its left-child chain, so
\(\ell_F(v)=h\). Hence
\(\kappacan_P(S_r)=\ell_F(v)\).

Otherwise \(S_{r-1}\rav S_r\). Let \(v'\) satisfy
\(S(v')=S_{r-1}\). By induction,
\(\kappacan_P(S_{r-1})=\ell_F(v')\). For every integer
\(j\) with \(\ell_F(v')<j<\ell_F(v)\), the right subtree of \(v'\)
contains a unique internal node \(z\) with \(\ell_F(z)=j\). Since
\(\rho_F(z)>h\), the south step \(S(z)\) occurs before \(S_r\), and the
induction hypothesis gives \(\kappacan_P(S(z))=j\).
Thus all integers strictly between \(\ell_F(v')\) and \(\ell_F(v)\)
have already been used, whereas \(\ell_F(v)\) has not. Therefore
\(\kappacan_P(S_r)=\ell_F(v)\), completing the induction.
\end{proof}

For the canonical labeling of \(P\), we write
\(\FL_{P,\kappacan_P}\xt\) simply as \(\FL_P\xt\).

\begin{thm}\label{thm:1}
For every \(P\in\Path\), we have
\( \FL_P\xt=\PP_{F(P)}\xt \).
\end{thm}
\begin{proof}
Put \(F=F(P)\) and choose \(n\) such that \(F\in\For_{n+1}\).
By \Cref{lem:forest-path-incidence}, the canonical labels on \(P\)
are the distinct labels \(\ell_F(v)\). In \(F\), the node \(S_r\)
is the left child of \(S_s\) if \(S_r\rav S_s\), and the right child
if \(S_r\rad S_s\). Thus the labels increase along vertical arrows
and decrease along diagonal arrows. The topmost step on \(x=2h\)
has label \(h\) by definition. Thus all the hypotheses of
\Cref{lem:subword-path-bijection} hold.
Moreover, \(s\) covers \(r\) in \(\mathcal H(P)\) if and only if
\(S_s\) is the parent of \(S_r\) in \(F\). Reading the canonical
labels along reversed linear extensions therefore gives precisely
the Sylvester words. Hence
\[
  \mathcal R(P,\kappacan_P)
  =\mathcal R(\widetilde\omega_{[n]};F).
\]
By \Cref{lem:subword-path-bijection} and \eqref{eq:double-forest-subword},
\[
  \FL_P(\vec x;\vec t)
  =\sum_{\pi\in\mathcal R(\widetilde\omega_{[n]};F)}\wt(\pi)
  =\PP_F(\vec x;\vec t). \qedhere
\]
\end{proof}

For \(F\in\For_{n+1}\) and \(P=P(F)\), write
\(\Phi_F:=\Phi_{P,\kappacan_P}\).
This specialization sends the letter \(B_v\) of value \(\ell_F(v)\)
to the cell \(C_{d(B_v),r}\) in the row \(S_F(v)=S_r\).

The double fundamental polynomials recalled in
\Cref{subsec:schur-fundamental} therefore have the following path model.

\begin{cor}\label{cor:double-fundamental-path}
For every \(c\in\operatorname{Pad}_n\),
\( \mathfrak{F}_c(x_1,\ldots,x_n;\vec t) =\FL_{P(c)}(\vec x;\vec t) \).
\end{cor}

\subsection{Specialization to ordinary forest polynomials}
\label{subsec:forest-correspondence}

For the canonical labeling of \(P\), transported to \(P^\circ\), we
abbreviate the lattice path polynomial
\(\FL_{P^\circ,\kappacan_P}(\vec x)\) to
\(\FL_{P^\circ}(\vec x)\).

Combining \Cref{thm:1} and \Cref{prop:compressed-specialization} gives, for every
\(P\in\Path\),
\begin{equation}\label{eq:ordinary-forest-path-model}
  \PP_{F(P)}(\vec x)
  =\FL_P(\vec x;\vec0)
  =\FL_{P^\circ}(\vec x).
\end{equation}
Thus ordinary forest polynomials are realized by the compressed paths
\(P^\circ\), whose west steps have unit length.

We now relate \eqref{eq:ordinary-forest-path-model} to the inequalities for a compatible labeling \eqref{eq:ordinary-forest-labeling}. Fix an indexed forest \(F\), and
put \(P=P(F)\). For a cell selection of \(P^\circ\), denote the selected
column in the row containing \(S(v)\) by \(f(v)\). By
\Cref{lem:forest-path-incidence}, the successor conditions \eqref{eq:compressed-cell-inequ} become
\[
  \begin{aligned}  
  f(v)&\leq f(v_L)&&\text{if}\ v_L\in\IN(F), \\
  f(v)&<f(v_R)&&\text{if}\ v_R\in\IN(F).
  \end{aligned}
\]
Since \(S(v)\) lies on \(x=\rho_F(v)\) in \(P^\circ\), we have
\(f(v)\leq\rho_F(v)\). Therefore the map sending a cell selection to
\(f\) is a bijection onto the compatible labelings in
\eqref{eq:ordinary-forest-labeling}. The bijection preserves weights because
the selected cell in row \(S(v)\) has weight \(x_{f(v)}\).
In particular, if
\(c\in\operatorname{Pad}_n\), then
\[
  \mathfrak{F}_c(x_1,\ldots,x_n;\vec0)
  =\FL_{P(c)^\circ}(\vec x),
\]
which is the corresponding fundamental quasisymmetric polynomial.

\subsection{A lattice path model for grove polynomials}

Fix \(P\in\Path\), and write \(m=|\mathcal S(P)|\). For
\(Q\subseteq\mathcal A(P^\circ)\), define
\( Q_r=\{C_{q,r}\in Q:q\geq1\} \) for \(r\in[m]\) and \(\col(Q_r)=\{q:C_{q,r}\in Q_r\}
\).
We call \(Q\) a \emph{multi-cell selection} of \(P^\circ\) if every
\(Q_r\) is nonempty and
\begin{equation}\label{eq:grove-multicell-compatibility}
\begin{aligned}
  S_r\rav S_s&\quad\Longrightarrow\quad
     \max\col(Q_s)\leq\min\col(Q_r),\\
  S_r\rad S_s&\quad\Longrightarrow\quad
     \max\col(Q_s)<\min\col(Q_r).
\end{aligned}
\end{equation}
Let \(\MCSel(P^\circ)\) denote the set of multi-cell selections.

\begin{thm}\label{thm:grove-multicell-model}
For \(P\in\Path\), let \(m=|\mathcal S(P)|\). Then the grove polynomial \( \widetilde\PP_{F(P)}^{(\beta)}(\vec x) \) is
\begin{equation}\label{eq:grove-multicell-formula}
  \widetilde\PP_{F(P)}^{(\beta)}(\vec x)
  =
  \sum_{Q\in\MCSel(P^\circ)}
    \beta^{|Q|-m}
    \prod_{C\in Q}\wt_{\kappacan_P}(C).
\end{equation}
\end{thm}
\begin{proof}
  Put \(F=F(P)\). For a grove labeling \(\kappa\) of \(F\), define
  \[
    Q_r=\{C_{q,r}:q\in\kappa(v)\}
    \qquad\text{when }S_r=S_F(v).
  \]
  By \Cref{lem:forest-path-incidence}, \(S_r\) lies on
  \(x=\rho_F(v)\) in \(P^\circ\). Hence
\(\max\kappa(v)\leq\rho_F(v)\) is equivalent to
\(Q_r\subseteq\mathcal A(P^\circ)\). Moreover an internal left child
corresponds to a vertical successor, while an internal right child
corresponds to a diagonal successor. Thus the two conditions in
\eqref{eq:grove-labeling-compatibility} are exactly those in
\eqref{eq:grove-multicell-compatibility}. Therefore
  \(\kappa\mapsto Q\) is a bijection from the grove labelings of \(F\) to
  \(\MCSel(P^\circ)\).

  If \(C_{q,r}\in Q_r\), then
  \( q\leq\rho_F(v)\leq\ell_F(v)=\kappacan_P(S_r) \),
  so \(\wt_{\kappacan_P}(C_{q,r})=x_q\). Moreover,
  \( |Q|=|\kappa| \) and \( m=|F| \).
  Hence the summand indexed by \(\kappa\) in
  \eqref{eq:grove-labeling-polynomial} equals the summand indexed by \(Q\)
  in \eqref{eq:grove-multicell-formula}.
  \end{proof}

\section{A lattice path model for Schubert polynomials}
\label{sec:schubert-path-models}
\raggedbottom

In this section, we express double Schubert
polynomials as sums of lattice path polynomials indexed by endpoint-decreasing
reduced words, using \(\Omega\)-parking and endpoint sorting. Ordinary Schubert polynomials are obtained at
\(\vec t=\vec0\), and Grothendieck polynomials are expressed using
endpoint-decreasing Hecke words and injective cell selections.

\subsection{The \texorpdfstring{\(\Omega^\bullet\)}{Omega}-correspondence}
\label{subsec:nt-omega}

In this subsection, we summarize the results of Nadeau and Tewari \cite{NT2024} used below.

In order to define the \( \Omega \)-parking procedure, we allow leaves to be indexed by \(\ZZ\),
with only finitely many nontrivial trees. Following
\cite{BGNST2025a}, set
\[
  L(F)=\{\ell_F(v):v\in\IN(F)\}.
\]
Equivalently, \(i\in L(F)\) if and only if the leaves indexed by
\(i\) and \(i+1\) belong to the same tree of \(F\).
This is the support used in \cite{NT2024}. In particular, for
\(F\in\For\) and \(n\geq1\),
\[
  F\in\For_n\quad\Longleftrightarrow\quad L(F)\subseteq[n-1].
\]
For an internal node \(v\), let \(\operatorname{INT}(v,F)\) be the
set of canonical labels of internal nodes in the subtree rooted at
\(v\). This is an interval of integers with minimum \(\rho_F(v)\).
For a labeled forest \(\mathsf P\) with underlying forest \(F\), set
\(\operatorname{INT}(v,\mathsf P)=\operatorname{INT}(v,F)\).

We use decorated letters \(i^{[k]}\), where \(i\in\ZZ\) and
\(k\geq1\), ordered lexicographically by \((i,k)\), and set
\(\val(i^{[k]})=i\). The original correspondence accepts any word
whose decorated letters are pairwise distinct.
For an integer word \(a=(a_1,\ldots,a_m)\), we follow the
convention in \cite[Section~5.1]{NT2024} by reading from left to
right and replacing the \(k\)-th occurrence of \(i\) by \(i^{[k]}\).
The resulting word
\(\widetilde a=(\widetilde a_1,\ldots,\widetilde a_m)\)
has pairwise distinct letters.

\begin{defn}\label{def:omega-correspondence}
The \emph{\(\Omega^\bullet\)-correspondence} associates the decorated
word \(\widetilde a\) with a pair of labeled forests, its \(P\)-symbol
and \(Q\)-symbol. We first construct their common underlying forest
by inserting the letters of \(\widetilde a\) from left to right.
Start with one leaf at every integer and no internal nodes.
At step \(r\), let \(F\) be the current forest.
When \(a_r\in L(F)\), let \([L,R]\) be the maximal interval in
\(L(F)\) containing \(a_r\), and let \(v_s\), with \(s<r\), be
the root of the corresponding tree. Set
\[
  p=
  \begin{cases}
    a_r, & a_r\notin L(F),\\
    L-1, & a_r\in L(F)\text{ and }\widetilde a_r<\widetilde a_s,\\
    R+1, & a_r\in L(F)\text{ and }\widetilde a_r>\widetilde a_s.
  \end{cases}
\]

Create a new internal node \(v_r\). To choose its left child,
first check whether \(p-1\in L(F)\). If so, find the tree
containing the internal node with canonical label \(p-1\), and
attach its root as the left child of \(v_r\). If not, attach the leaf indexed by \(p\) as its left child.
Next check whether \(p+1\in L(F)\). If so, attach the root of
the tree containing the internal node with canonical label \(p+1\)
as the right child of \(v_r\). If not, attach the leaf indexed by \(p+1\) as its right child.
Then \(v_r\) has canonical label \(p\), so this insertion adds \(p\) to \(L(F)\).

Let \(F_\Omega(a)\) be the forest obtained after all insertions.
Labeling each \(v_r\) by \(\widetilde a_r\) gives the
\emph{\(P\)-symbol} \(\mathsf P_\Omega(a)\), and labeling it by
\(r\) gives the \emph{\(Q\)-symbol} \(\mathsf Q_\Omega(a)\).
Thus
\[
  \Omega^\bullet(\widetilde a)
    =(\mathsf P_\Omega(a),\mathsf Q_\Omega(a)).
\]
\end{defn}

In \( P_\Omega(a) \), each internal left
child has a smaller label than its parent, and each internal right
child has a larger label. The value of the label at \(v\) belongs
to \(\operatorname{INT}(v,F_\Omega(a))\).
In \( Q_\Omega(a) \), every integer in \([m]\) is used exactly once
as a label, and each parent has a larger label than its internal children. Two words \(a,b\) are \emph{\(\Omega\)-equivalent} if \(\mathsf P_\Omega(a)=\mathsf P_\Omega(b)\).

The words are obtained from the \(P\)-symbol \(\mathsf P_\Omega(a)\) by taking the node order of a Sylvester word of \(F_\Omega(a)\), reversing it,
and reading the corresponding labels of \(\mathsf P_\Omega(a)\)
with decorations removed. These orders are exactly the \emph{linear extensions} of the internal-node poset with children preceding parents. Equal values lie on a
chain, so every such reading preserves their occurrence order.

We define the \emph{rightmost reading} of a forest \(F\) as follows.
We read its nontrivial trees from right to left. Within each tree, we
recursively read the internal nodes of the right subtree before
those of the left subtree and read the root last.
The flags \(\rho_F(v)\) appear in weakly decreasing order along
this reading. Nodes with the same flag lie on a left-child chain
and are read from descendant to ancestor.

Let \(\mathsf{IndFor}_+\) be the set of indexed forests whose
canonical labels are all positive, equivalently
\(L(F)\subseteq\ZZ_{>0}\). Forests outside this set
have zero forest polynomial.

\begin{exam}\label{ex:nt-omega-insertion}
Let \(a=(4,3,1,2,4,3,4)\), so
\[
  \widetilde a=(4^{[1]},3^{[1]},1^{[1]},2^{[1]},
    4^{[2]},3^{[2]},4^{[3]}).
\]
In \Cref{fig:nt-omega-insertion}, the first three letters produce
two nontrivial trees, rooted at \(3^{[1]}\) and \(1^{[1]}\).
Inserting \(2^{[1]}\) joins them under a new root, whose left and
right children are the nodes labeled \(1^{[1]}\) and \(3^{[1]}\),
respectively. This node is labeled \(4\) in the \(Q\)-symbol
because it is created at the fourth step.
\end{exam}

\begin{figure}
\centering
\begin{tikzpicture}[x=1cm,y=1cm]
\begin{scope}[shift={(0.2cm,6.3cm)},x=0.68cm,y=0.68cm]
\draw[gray!65,line width=.35pt] (-.2,0)--(7.2,0);
\draw[black!80,line width=.4pt] (-0.065,-.065)--(0.065,.065) (-0.065,.065)--(0.065,-.065);
\node[font=\scriptsize,anchor=north,inner sep=1pt] at (0,-.10) {0};
\draw[black!80,line width=.4pt] (0.935,-.065)--(1.065,.065) (0.935,.065)--(1.065,-.065);
\node[font=\scriptsize,anchor=north,inner sep=1pt] at (1,-.10) {1};
\draw[black!80,line width=.4pt] (1.935,-.065)--(2.065,.065) (1.935,.065)--(2.065,-.065);
\node[font=\scriptsize,anchor=north,inner sep=1pt] at (2,-.10) {2};
\draw[black!80,line width=.4pt] (2.935,-.065)--(3.065,.065) (2.935,.065)--(3.065,-.065);
\node[font=\scriptsize,anchor=north,inner sep=1pt] at (3,-.10) {3};
\draw[black!80,line width=.4pt] (3.935,-.065)--(4.065,.065) (3.935,.065)--(4.065,-.065);
\node[font=\scriptsize,anchor=north,inner sep=1pt] at (4,-.10) {4};
\draw[black!80,line width=.4pt] (4.935,-.065)--(5.065,.065) (4.935,.065)--(5.065,-.065);
\node[font=\scriptsize,anchor=north,inner sep=1pt] at (5,-.10) {5};
\draw[black!80,line width=.4pt] (5.935,-.065)--(6.065,.065) (5.935,.065)--(6.065,-.065);
\node[font=\scriptsize,anchor=north,inner sep=1pt] at (6,-.10) {6};
\draw[black!80,line width=.4pt] (6.935,-.065)--(7.065,.065) (6.935,.065)--(7.065,-.065);
\node[font=\scriptsize,anchor=north,inner sep=1pt] at (7,-.10) {7};
\draw[blue!75!black,line width=.55pt] (4.5,0.5)--(4,0);
\draw[blue!75!black,line width=.55pt] (4.5,0.5)--(5,0);
\fill[blue!75!black] (4.5,0.5) circle[radius=1.4pt];
\node[font=\fontsize{7}{8}\selectfont,blue!75!black,anchor=south,inner sep=.3pt] at (4.5,0.63) {\(4^{[1]}\)};
\end{scope}
\begin{scope}[shift={(5.55cm,6.3cm)},x=0.68cm,y=0.68cm]
\draw[gray!65,line width=.35pt] (-.2,0)--(7.2,0);
\draw[black!80,line width=.4pt] (-0.065,-.065)--(0.065,.065) (-0.065,.065)--(0.065,-.065);
\node[font=\scriptsize,anchor=north,inner sep=1pt] at (0,-.10) {0};
\draw[black!80,line width=.4pt] (0.935,-.065)--(1.065,.065) (0.935,.065)--(1.065,-.065);
\node[font=\scriptsize,anchor=north,inner sep=1pt] at (1,-.10) {1};
\draw[black!80,line width=.4pt] (1.935,-.065)--(2.065,.065) (1.935,.065)--(2.065,-.065);
\node[font=\scriptsize,anchor=north,inner sep=1pt] at (2,-.10) {2};
\draw[black!80,line width=.4pt] (2.935,-.065)--(3.065,.065) (2.935,.065)--(3.065,-.065);
\node[font=\scriptsize,anchor=north,inner sep=1pt] at (3,-.10) {3};
\draw[black!80,line width=.4pt] (3.935,-.065)--(4.065,.065) (3.935,.065)--(4.065,-.065);
\node[font=\scriptsize,anchor=north,inner sep=1pt] at (4,-.10) {4};
\draw[black!80,line width=.4pt] (4.935,-.065)--(5.065,.065) (4.935,.065)--(5.065,-.065);
\node[font=\scriptsize,anchor=north,inner sep=1pt] at (5,-.10) {5};
\draw[black!80,line width=.4pt] (5.935,-.065)--(6.065,.065) (5.935,.065)--(6.065,-.065);
\node[font=\scriptsize,anchor=north,inner sep=1pt] at (6,-.10) {6};
\draw[black!80,line width=.4pt] (6.935,-.065)--(7.065,.065) (6.935,.065)--(7.065,-.065);
\node[font=\scriptsize,anchor=north,inner sep=1pt] at (7,-.10) {7};
\draw[black,line width=.55pt] (4.5,0.5)--(4,0);
\draw[black,line width=.55pt] (4.5,0.5)--(5,0);
\draw[blue!75!black,line width=.55pt] (4,1)--(3,0);
\draw[blue!75!black,line width=.55pt] (4,1)--(4.5,0.5);
\fill[black] (4.5,0.5) circle[radius=1.4pt];
\node[font=\fontsize{7}{8}\selectfont,black,anchor=west,inner sep=.3pt] at (4.67,0.5) {\(4^{[1]}\)};
\fill[blue!75!black] (4,1) circle[radius=1.4pt];
\node[font=\fontsize{7}{8}\selectfont,blue!75!black,anchor=south,inner sep=.3pt] at (4,1.13) {\(3^{[1]}\)};
\end{scope}
\begin{scope}[shift={(10.9cm,6.3cm)},x=0.68cm,y=0.68cm]
\draw[gray!65,line width=.35pt] (-.2,0)--(7.2,0);
\draw[black!80,line width=.4pt] (-0.065,-.065)--(0.065,.065) (-0.065,.065)--(0.065,-.065);
\node[font=\scriptsize,anchor=north,inner sep=1pt] at (0,-.10) {0};
\draw[black!80,line width=.4pt] (0.935,-.065)--(1.065,.065) (0.935,.065)--(1.065,-.065);
\node[font=\scriptsize,anchor=north,inner sep=1pt] at (1,-.10) {1};
\draw[black!80,line width=.4pt] (1.935,-.065)--(2.065,.065) (1.935,.065)--(2.065,-.065);
\node[font=\scriptsize,anchor=north,inner sep=1pt] at (2,-.10) {2};
\draw[black!80,line width=.4pt] (2.935,-.065)--(3.065,.065) (2.935,.065)--(3.065,-.065);
\node[font=\scriptsize,anchor=north,inner sep=1pt] at (3,-.10) {3};
\draw[black!80,line width=.4pt] (3.935,-.065)--(4.065,.065) (3.935,.065)--(4.065,-.065);
\node[font=\scriptsize,anchor=north,inner sep=1pt] at (4,-.10) {4};
\draw[black!80,line width=.4pt] (4.935,-.065)--(5.065,.065) (4.935,.065)--(5.065,-.065);
\node[font=\scriptsize,anchor=north,inner sep=1pt] at (5,-.10) {5};
\draw[black!80,line width=.4pt] (5.935,-.065)--(6.065,.065) (5.935,.065)--(6.065,-.065);
\node[font=\scriptsize,anchor=north,inner sep=1pt] at (6,-.10) {6};
\draw[black!80,line width=.4pt] (6.935,-.065)--(7.065,.065) (6.935,.065)--(7.065,-.065);
\node[font=\scriptsize,anchor=north,inner sep=1pt] at (7,-.10) {7};
\draw[black,line width=.55pt] (4.5,0.5)--(4,0);
\draw[black,line width=.55pt] (4.5,0.5)--(5,0);
\draw[black,line width=.55pt] (4,1)--(3,0);
\draw[black,line width=.55pt] (4,1)--(4.5,0.5);
\draw[blue!75!black,line width=.55pt] (1.5,0.5)--(1,0);
\draw[blue!75!black,line width=.55pt] (1.5,0.5)--(2,0);
\fill[black] (4.5,0.5) circle[radius=1.4pt];
\node[font=\fontsize{7}{8}\selectfont,black,anchor=west,inner sep=.3pt] at (4.67,0.5) {\(4^{[1]}\)};
\fill[black] (4,1) circle[radius=1.4pt];
\node[font=\fontsize{7}{8}\selectfont,black,anchor=south,inner sep=.3pt] at (4,1.13) {\(3^{[1]}\)};
\fill[blue!75!black] (1.5,0.5) circle[radius=1.4pt];
\node[font=\fontsize{7}{8}\selectfont,blue!75!black,anchor=south,inner sep=.3pt] at (1.5,0.63) {\(1^{[1]}\)};
\end{scope}
\begin{scope}[shift={(0.2cm,3.55cm)},x=0.68cm,y=0.68cm]
\draw[gray!65,line width=.35pt] (-.2,0)--(7.2,0);
\draw[black!80,line width=.4pt] (-0.065,-.065)--(0.065,.065) (-0.065,.065)--(0.065,-.065);
\node[font=\scriptsize,anchor=north,inner sep=1pt] at (0,-.10) {0};
\draw[black!80,line width=.4pt] (0.935,-.065)--(1.065,.065) (0.935,.065)--(1.065,-.065);
\node[font=\scriptsize,anchor=north,inner sep=1pt] at (1,-.10) {1};
\draw[black!80,line width=.4pt] (1.935,-.065)--(2.065,.065) (1.935,.065)--(2.065,-.065);
\node[font=\scriptsize,anchor=north,inner sep=1pt] at (2,-.10) {2};
\draw[black!80,line width=.4pt] (2.935,-.065)--(3.065,.065) (2.935,.065)--(3.065,-.065);
\node[font=\scriptsize,anchor=north,inner sep=1pt] at (3,-.10) {3};
\draw[black!80,line width=.4pt] (3.935,-.065)--(4.065,.065) (3.935,.065)--(4.065,-.065);
\node[font=\scriptsize,anchor=north,inner sep=1pt] at (4,-.10) {4};
\draw[black!80,line width=.4pt] (4.935,-.065)--(5.065,.065) (4.935,.065)--(5.065,-.065);
\node[font=\scriptsize,anchor=north,inner sep=1pt] at (5,-.10) {5};
\draw[black!80,line width=.4pt] (5.935,-.065)--(6.065,.065) (5.935,.065)--(6.065,-.065);
\node[font=\scriptsize,anchor=north,inner sep=1pt] at (6,-.10) {6};
\draw[black!80,line width=.4pt] (6.935,-.065)--(7.065,.065) (6.935,.065)--(7.065,-.065);
\node[font=\scriptsize,anchor=north,inner sep=1pt] at (7,-.10) {7};
\draw[black,line width=.55pt] (4.5,0.5)--(4,0);
\draw[black,line width=.55pt] (4.5,0.5)--(5,0);
\draw[black,line width=.55pt] (4,1)--(3,0);
\draw[black,line width=.55pt] (4,1)--(4.5,0.5);
\draw[black,line width=.55pt] (1.5,0.5)--(1,0);
\draw[black,line width=.55pt] (1.5,0.5)--(2,0);
\draw[blue!75!black,line width=.55pt] (3,2)--(1.5,0.5);
\draw[blue!75!black,line width=.55pt] (3,2)--(4,1);
\fill[black] (4.5,0.5) circle[radius=1.4pt];
\node[font=\fontsize{7}{8}\selectfont,black,anchor=west,inner sep=.3pt] at (4.67,0.5) {\(4^{[1]}\)};
\fill[black] (4,1) circle[radius=1.4pt];
\node[font=\fontsize{7}{8}\selectfont,black,anchor=west,inner sep=.3pt] at (4.17,1) {\(3^{[1]}\)};
\fill[black] (1.5,0.5) circle[radius=1.4pt];
\node[font=\fontsize{7}{8}\selectfont,black,anchor=east,inner sep=.3pt] at (1.33,0.5) {\(1^{[1]}\)};
\fill[blue!75!black] (3,2) circle[radius=1.4pt];
\node[font=\fontsize{7}{8}\selectfont,blue!75!black,anchor=south,inner sep=.3pt] at (3,2.13) {\(2^{[1]}\)};
\end{scope}
\begin{scope}[shift={(5.55cm,3.55cm)},x=0.68cm,y=0.68cm]
\draw[gray!65,line width=.35pt] (-.2,0)--(7.2,0);
\draw[black!80,line width=.4pt] (-0.065,-.065)--(0.065,.065) (-0.065,.065)--(0.065,-.065);
\node[font=\scriptsize,anchor=north,inner sep=1pt] at (0,-.10) {0};
\draw[black!80,line width=.4pt] (0.935,-.065)--(1.065,.065) (0.935,.065)--(1.065,-.065);
\node[font=\scriptsize,anchor=north,inner sep=1pt] at (1,-.10) {1};
\draw[black!80,line width=.4pt] (1.935,-.065)--(2.065,.065) (1.935,.065)--(2.065,-.065);
\node[font=\scriptsize,anchor=north,inner sep=1pt] at (2,-.10) {2};
\draw[black!80,line width=.4pt] (2.935,-.065)--(3.065,.065) (2.935,.065)--(3.065,-.065);
\node[font=\scriptsize,anchor=north,inner sep=1pt] at (3,-.10) {3};
\draw[black!80,line width=.4pt] (3.935,-.065)--(4.065,.065) (3.935,.065)--(4.065,-.065);
\node[font=\scriptsize,anchor=north,inner sep=1pt] at (4,-.10) {4};
\draw[black!80,line width=.4pt] (4.935,-.065)--(5.065,.065) (4.935,.065)--(5.065,-.065);
\node[font=\scriptsize,anchor=north,inner sep=1pt] at (5,-.10) {5};
\draw[black!80,line width=.4pt] (5.935,-.065)--(6.065,.065) (5.935,.065)--(6.065,-.065);
\node[font=\scriptsize,anchor=north,inner sep=1pt] at (6,-.10) {6};
\draw[black!80,line width=.4pt] (6.935,-.065)--(7.065,.065) (6.935,.065)--(7.065,-.065);
\node[font=\scriptsize,anchor=north,inner sep=1pt] at (7,-.10) {7};
\draw[black,line width=.55pt] (4.5,0.5)--(4,0);
\draw[black,line width=.55pt] (4.5,0.5)--(5,0);
\draw[black,line width=.55pt] (4,1)--(3,0);
\draw[black,line width=.55pt] (4,1)--(4.5,0.5);
\draw[black,line width=.55pt] (1.5,0.5)--(1,0);
\draw[black,line width=.55pt] (1.5,0.5)--(2,0);
\draw[black,line width=.55pt] (3,2)--(1.5,0.5);
\draw[black,line width=.55pt] (3,2)--(4,1);
\draw[blue!75!black,line width=.55pt] (3.5,2.5)--(3,2);
\draw[blue!75!black,line width=.55pt] (3.5,2.5)--(6,0);
\fill[black] (4.5,0.5) circle[radius=1.4pt];
\node[font=\fontsize{7}{8}\selectfont,black,anchor=west,inner sep=.3pt] at (4.67,0.5) {\(4^{[1]}\)};
\fill[black] (4,1) circle[radius=1.4pt];
\node[font=\fontsize{7}{8}\selectfont,black,anchor=west,inner sep=.3pt] at (4.17,1) {\(3^{[1]}\)};
\fill[black] (1.5,0.5) circle[radius=1.4pt];
\node[font=\fontsize{7}{8}\selectfont,black,anchor=east,inner sep=.3pt] at (1.33,0.5) {\(1^{[1]}\)};
\fill[black] (3,2) circle[radius=1.4pt];
\node[font=\fontsize{7}{8}\selectfont,black,anchor=east,inner sep=.3pt] at (2.83,2) {\(2^{[1]}\)};
\fill[blue!75!black] (3.5,2.5) circle[radius=1.4pt];
\node[font=\fontsize{7}{8}\selectfont,blue!75!black,anchor=south,inner sep=.3pt] at (3.5,2.63) {\(4^{[2]}\)};
\end{scope}
\begin{scope}[shift={(10.9cm,3.55cm)},x=0.68cm,y=0.68cm]
\draw[gray!65,line width=.35pt] (-.2,0)--(7.2,0);
\draw[black!80,line width=.4pt] (-0.065,-.065)--(0.065,.065) (-0.065,.065)--(0.065,-.065);
\node[font=\scriptsize,anchor=north,inner sep=1pt] at (0,-.10) {0};
\draw[black!80,line width=.4pt] (0.935,-.065)--(1.065,.065) (0.935,.065)--(1.065,-.065);
\node[font=\scriptsize,anchor=north,inner sep=1pt] at (1,-.10) {1};
\draw[black!80,line width=.4pt] (1.935,-.065)--(2.065,.065) (1.935,.065)--(2.065,-.065);
\node[font=\scriptsize,anchor=north,inner sep=1pt] at (2,-.10) {2};
\draw[black!80,line width=.4pt] (2.935,-.065)--(3.065,.065) (2.935,.065)--(3.065,-.065);
\node[font=\scriptsize,anchor=north,inner sep=1pt] at (3,-.10) {3};
\draw[black!80,line width=.4pt] (3.935,-.065)--(4.065,.065) (3.935,.065)--(4.065,-.065);
\node[font=\scriptsize,anchor=north,inner sep=1pt] at (4,-.10) {4};
\draw[black!80,line width=.4pt] (4.935,-.065)--(5.065,.065) (4.935,.065)--(5.065,-.065);
\node[font=\scriptsize,anchor=north,inner sep=1pt] at (5,-.10) {5};
\draw[black!80,line width=.4pt] (5.935,-.065)--(6.065,.065) (5.935,.065)--(6.065,-.065);
\node[font=\scriptsize,anchor=north,inner sep=1pt] at (6,-.10) {6};
\draw[black!80,line width=.4pt] (6.935,-.065)--(7.065,.065) (6.935,.065)--(7.065,-.065);
\node[font=\scriptsize,anchor=north,inner sep=1pt] at (7,-.10) {7};
\draw[black,line width=.55pt] (4.5,0.5)--(4,0);
\draw[black,line width=.55pt] (4.5,0.5)--(5,0);
\draw[black,line width=.55pt] (4,1)--(3,0);
\draw[black,line width=.55pt] (4,1)--(4.5,0.5);
\draw[black,line width=.55pt] (1.5,0.5)--(1,0);
\draw[black,line width=.55pt] (1.5,0.5)--(2,0);
\draw[black,line width=.55pt] (3,2)--(1.5,0.5);
\draw[black,line width=.55pt] (3,2)--(4,1);
\draw[black,line width=.55pt] (3.5,2.5)--(3,2);
\draw[black,line width=.55pt] (3.5,2.5)--(6,0);
\draw[blue!75!black,line width=.55pt] (3,3)--(0,0);
\draw[blue!75!black,line width=.55pt] (3,3)--(3.5,2.5);
\fill[black] (4.5,0.5) circle[radius=1.4pt];
\node[font=\fontsize{7}{8}\selectfont,black,anchor=west,inner sep=.3pt] at (4.67,0.5) {\(4^{[1]}\)};
\fill[black] (4,1) circle[radius=1.4pt];
\node[font=\fontsize{7}{8}\selectfont,black,anchor=west,inner sep=.3pt] at (4.17,1) {\(3^{[1]}\)};
\fill[black] (1.5,0.5) circle[radius=1.4pt];
\node[font=\fontsize{7}{8}\selectfont,black,anchor=east,inner sep=.3pt] at (1.33,0.5) {\(1^{[1]}\)};
\fill[black] (3,2) circle[radius=1.4pt];
\node[font=\fontsize{7}{8}\selectfont,black,anchor=east,inner sep=.3pt] at (2.83,2) {\(2^{[1]}\)};
\fill[black] (3.5,2.5) circle[radius=1.4pt];
\node[font=\fontsize{7}{8}\selectfont,black,anchor=west,inner sep=.3pt] at (3.67,2.5) {\(4^{[2]}\)};
\fill[blue!75!black] (3,3) circle[radius=1.4pt];
\node[font=\fontsize{7}{8}\selectfont,blue!75!black,anchor=south,inner sep=.3pt] at (3,3.13) {\(3^{[2]}\)};
\end{scope}
\draw[gray!70,dashed,line width=.35pt] (2.0,-.72) rectangle (13.9,3.05);
\begin{scope}[shift={(2.3cm,0.25cm)},x=0.68cm,y=0.68cm]
\draw[gray!65,line width=.35pt] (-.2,0)--(7.2,0);
\draw[black!80,line width=.4pt] (-0.065,-.065)--(0.065,.065) (-0.065,.065)--(0.065,-.065);
\node[font=\scriptsize,anchor=north,inner sep=1pt] at (0,-.10) {0};
\draw[black!80,line width=.4pt] (0.935,-.065)--(1.065,.065) (0.935,.065)--(1.065,-.065);
\node[font=\scriptsize,anchor=north,inner sep=1pt] at (1,-.10) {1};
\draw[black!80,line width=.4pt] (1.935,-.065)--(2.065,.065) (1.935,.065)--(2.065,-.065);
\node[font=\scriptsize,anchor=north,inner sep=1pt] at (2,-.10) {2};
\draw[black!80,line width=.4pt] (2.935,-.065)--(3.065,.065) (2.935,.065)--(3.065,-.065);
\node[font=\scriptsize,anchor=north,inner sep=1pt] at (3,-.10) {3};
\draw[black!80,line width=.4pt] (3.935,-.065)--(4.065,.065) (3.935,.065)--(4.065,-.065);
\node[font=\scriptsize,anchor=north,inner sep=1pt] at (4,-.10) {4};
\draw[black!80,line width=.4pt] (4.935,-.065)--(5.065,.065) (4.935,.065)--(5.065,-.065);
\node[font=\scriptsize,anchor=north,inner sep=1pt] at (5,-.10) {5};
\draw[black!80,line width=.4pt] (5.935,-.065)--(6.065,.065) (5.935,.065)--(6.065,-.065);
\node[font=\scriptsize,anchor=north,inner sep=1pt] at (6,-.10) {6};
\draw[black!80,line width=.4pt] (6.935,-.065)--(7.065,.065) (6.935,.065)--(7.065,-.065);
\node[font=\scriptsize,anchor=north,inner sep=1pt] at (7,-.10) {7};
\draw[black,line width=.55pt] (4.5,0.5)--(4,0);
\draw[black,line width=.55pt] (4.5,0.5)--(5,0);
\draw[black,line width=.55pt] (4,1)--(3,0);
\draw[black,line width=.55pt] (4,1)--(4.5,0.5);
\draw[black,line width=.55pt] (1.5,0.5)--(1,0);
\draw[black,line width=.55pt] (1.5,0.5)--(2,0);
\draw[black,line width=.55pt] (3,2)--(1.5,0.5);
\draw[black,line width=.55pt] (3,2)--(4,1);
\draw[black,line width=.55pt] (3.5,2.5)--(3,2);
\draw[black,line width=.55pt] (3.5,2.5)--(6,0);
\draw[black,line width=.55pt] (3,3)--(0,0);
\draw[black,line width=.55pt] (3,3)--(3.5,2.5);
\draw[black,line width=.55pt] (3.5,3.5)--(3,3);
\draw[black,line width=.55pt] (3.5,3.5)--(7,0);
\fill[black] (4.5,0.5) circle[radius=1.4pt];
\node[font=\fontsize{7}{8}\selectfont,black,anchor=west,inner sep=.3pt] at (4.67,0.5) {\(4^{[1]}\)};
\fill[black] (4,1) circle[radius=1.4pt];
\node[font=\fontsize{7}{8}\selectfont,black,anchor=west,inner sep=.3pt] at (4.17,1) {\(3^{[1]}\)};
\fill[black] (1.5,0.5) circle[radius=1.4pt];
\node[font=\fontsize{7}{8}\selectfont,black,anchor=east,inner sep=.3pt] at (1.33,0.5) {\(1^{[1]}\)};
\fill[black] (3,2) circle[radius=1.4pt];
\node[font=\fontsize{7}{8}\selectfont,black,anchor=east,inner sep=.3pt] at (2.83,2) {\(2^{[1]}\)};
\fill[black] (3.5,2.5) circle[radius=1.4pt];
\node[font=\fontsize{7}{8}\selectfont,black,anchor=west,inner sep=.3pt] at (3.67,2.5) {\(4^{[2]}\)};
\fill[black] (3,3) circle[radius=1.4pt];
\node[font=\fontsize{7}{8}\selectfont,black,anchor=east,inner sep=.3pt] at (2.83,3) {\(3^{[2]}\)};
\fill[black] (3.5,3.5) circle[radius=1.4pt];
\node[font=\fontsize{7}{8}\selectfont,black,anchor=south,inner sep=.3pt] at (3.5,3.63) {\(4^{[3]}\)};
\end{scope}
\begin{scope}[shift={(8.8cm,0.25cm)},x=0.68cm,y=0.68cm]
\draw[gray!65,line width=.35pt] (-.2,0)--(7.2,0);
\draw[black!80,line width=.4pt] (-0.065,-.065)--(0.065,.065) (-0.065,.065)--(0.065,-.065);
\node[font=\scriptsize,anchor=north,inner sep=1pt] at (0,-.10) {0};
\draw[black!80,line width=.4pt] (0.935,-.065)--(1.065,.065) (0.935,.065)--(1.065,-.065);
\node[font=\scriptsize,anchor=north,inner sep=1pt] at (1,-.10) {1};
\draw[black!80,line width=.4pt] (1.935,-.065)--(2.065,.065) (1.935,.065)--(2.065,-.065);
\node[font=\scriptsize,anchor=north,inner sep=1pt] at (2,-.10) {2};
\draw[black!80,line width=.4pt] (2.935,-.065)--(3.065,.065) (2.935,.065)--(3.065,-.065);
\node[font=\scriptsize,anchor=north,inner sep=1pt] at (3,-.10) {3};
\draw[black!80,line width=.4pt] (3.935,-.065)--(4.065,.065) (3.935,.065)--(4.065,-.065);
\node[font=\scriptsize,anchor=north,inner sep=1pt] at (4,-.10) {4};
\draw[black!80,line width=.4pt] (4.935,-.065)--(5.065,.065) (4.935,.065)--(5.065,-.065);
\node[font=\scriptsize,anchor=north,inner sep=1pt] at (5,-.10) {5};
\draw[black!80,line width=.4pt] (5.935,-.065)--(6.065,.065) (5.935,.065)--(6.065,-.065);
\node[font=\scriptsize,anchor=north,inner sep=1pt] at (6,-.10) {6};
\draw[black!80,line width=.4pt] (6.935,-.065)--(7.065,.065) (6.935,.065)--(7.065,-.065);
\node[font=\scriptsize,anchor=north,inner sep=1pt] at (7,-.10) {7};
\draw[black,line width=.55pt] (4.5,0.5)--(4,0);
\draw[black,line width=.55pt] (4.5,0.5)--(5,0);
\draw[black,line width=.55pt] (4,1)--(3,0);
\draw[black,line width=.55pt] (4,1)--(4.5,0.5);
\draw[black,line width=.55pt] (1.5,0.5)--(1,0);
\draw[black,line width=.55pt] (1.5,0.5)--(2,0);
\draw[black,line width=.55pt] (3,2)--(1.5,0.5);
\draw[black,line width=.55pt] (3,2)--(4,1);
\draw[black,line width=.55pt] (3.5,2.5)--(3,2);
\draw[black,line width=.55pt] (3.5,2.5)--(6,0);
\draw[black,line width=.55pt] (3,3)--(0,0);
\draw[black,line width=.55pt] (3,3)--(3.5,2.5);
\draw[black,line width=.55pt] (3.5,3.5)--(3,3);
\draw[black,line width=.55pt] (3.5,3.5)--(7,0);
\fill[black] (4.5,0.5) circle[radius=1.4pt];
\node[font=\fontsize{7}{8}\selectfont,black,anchor=west,inner sep=.3pt] at (4.67,0.5) {\(1\)};
\fill[black] (4,1) circle[radius=1.4pt];
\node[font=\fontsize{7}{8}\selectfont,black,anchor=west,inner sep=.3pt] at (4.17,1) {\(2\)};
\fill[black] (1.5,0.5) circle[radius=1.4pt];
\node[font=\fontsize{7}{8}\selectfont,black,anchor=east,inner sep=.3pt] at (1.33,0.5) {\(3\)};
\fill[black] (3,2) circle[radius=1.4pt];
\node[font=\fontsize{7}{8}\selectfont,black,anchor=east,inner sep=.3pt] at (2.83,2) {\(4\)};
\fill[black] (3.5,2.5) circle[radius=1.4pt];
\node[font=\fontsize{7}{8}\selectfont,black,anchor=west,inner sep=.3pt] at (3.67,2.5) {\(5\)};
\fill[black] (3,3) circle[radius=1.4pt];
\node[font=\fontsize{7}{8}\selectfont,black,anchor=east,inner sep=.3pt] at (2.83,3) {\(6\)};
\fill[black] (3.5,3.5) circle[radius=1.4pt];
\node[font=\fontsize{7}{8}\selectfont,black,anchor=south,inner sep=.3pt] at (3.5,3.63) {\(7\)};
\end{scope}
\node[font=\small] at (4.68,-.42) {\(\mathsf P_\Omega(a)\)};
\node[font=\small] at (11.18,-.42) {\(\mathsf Q_\Omega(a)\)};
\end{tikzpicture}
\caption{The \(\Omega^\bullet\)-correspondence for \(a=(4,3,1,2,4,3,4)\); new nodes are blue.}
\label{fig:nt-omega-insertion}
\end{figure}
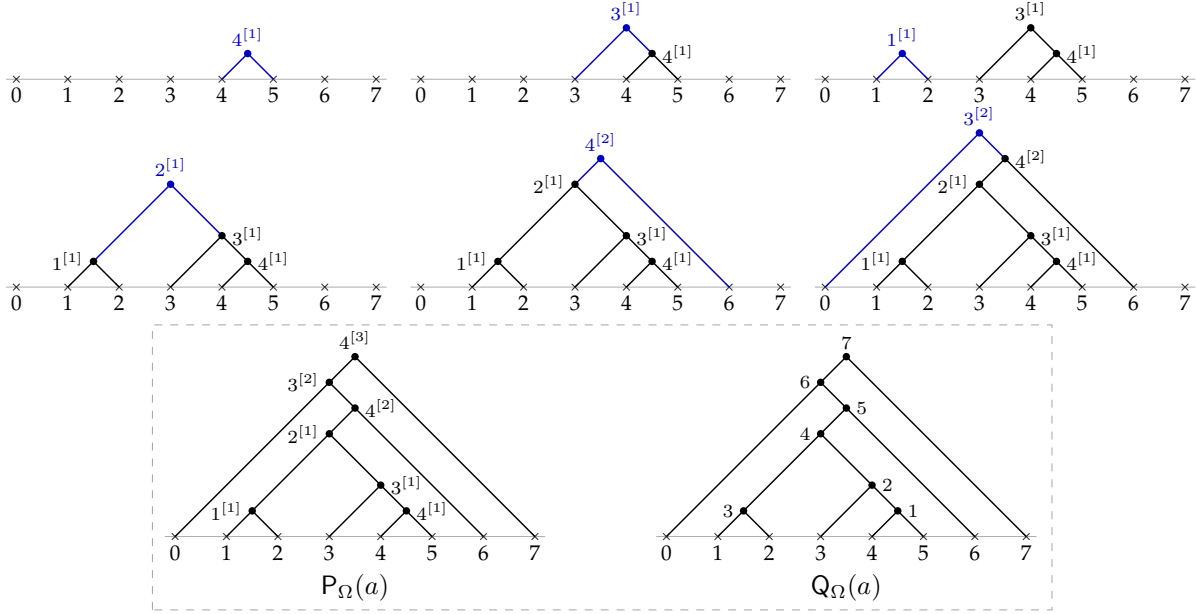

\subsection{\texorpdfstring{\(\Omega\)}{Omega}-parking and word paths}
\label{subsec:endpoint-decreasing}

We realize \(\Omega\)-parking
\cite[Section 4.2]{NT2024} by marked blocks, and use it to
construct word paths. The correspondence with
\Cref{subsec:nt-omega} is recorded in \Cref{tab:omega-dictionary}
and illustrated in \Cref{fig:3}.

For \(m\geq0\), let \(\Word_m\) be the set of words of length
\(m\) in positive integers.
For \(m>0\), fix
\(a=(a_1,\ldots,a_m)\in\Word_m\).
For \(i\in\ZZ\), let
\( I_i:=[i,i+1]\times[0,1]\subset\RR^2 \)
be the block at \(i\).  For a finite set \(\mathcal I\) of blocks, a
\emph{component} is a maximal subset of the form
\( \{I_L,I_{L+1},\ldots,I_R\}\subseteq\mathcal I \).
We construct \(\mathcal I_r\) recursively, marking one block in each
component and calling it that component's \emph{marked block}. Each mark records the index of the
insertion that created it. Set
\(\mathcal I_0=\varnothing\).  For \(r=1,\ldots,m\), define
\(\mathcal I_r\) as follows.  If
\(I_{a_r}\notin\mathcal I_{r-1}\), set
\[
  \mathcal I_r=\mathcal I_{r-1}\cup\{I_{a_r}\}.
\]
Otherwise, let
\( \{I_L,I_{L+1},\ldots,I_R\} \)
be the component containing \(I_{a_r}\), and let \(s<r\) be the index
recorded on its marked block. Set
\begin{equation}\label{eq:block-rule}
  \mathcal I_r=
  \begin{cases}
    \mathcal I_{r-1}\cup\{I_{L-1}\},
      &a_r<a_s,\\
    \mathcal I_{r-1}\cup\{I_{R+1}\},
      &a_r\geq a_s.
  \end{cases}
\end{equation}
In the component containing the unique block
\(I\in\mathcal I_r\setminus\mathcal I_{r-1}\), remove all existing marks,
mark \(I\) with the index \(r\), and keep all marks in the other
components. Let \(J_r(a)\) be the set of positions of the blocks in this
component and set \(e_r(a):=\min J_r(a)\).

Each component mark records its most recent insertion.
For \(1\leq r<s\leq m\), define a relation \(\trl_a\) on \([m]\) by
\begin{equation}\label{eq:mark-removal-poset}
  r\trl_a s
  \quad\Longleftrightarrow\quad
  \text{the \(s\)-th insertion removes the mark labeled \(r\)}.
\end{equation}
Let \(\prec_a\) be its transitive closure and put
\(\mathcal H(a)=([m],\prec_a)\).

We call \(a\) \emph{endpoint-decreasing} if
\(e_1(a)\geq e_2(a)\geq\cdots\geq e_m(a)\).
The empty word is also endpoint-decreasing. Denote the set of such
words of length \(m\) by \(\EWord_m\).

For \(r\in[m]\), let \(v_r\) be the node created at step \(r\)
in the forest construction, and let \(I_{p_r}\) be the block added
at the same step by \(\Omega\)-parking.
When the new block joins an existing component, the forest insertion
makes the root of the corresponding tree a child of \(v_r\).
The mark on that component is removed, and the new mark identifies
\(v_r\). Later insertions leave the subtree rooted at \(v_r\)
unchanged. Since the positions occupied in the two constructions
agree at each step by \cite[Proposition~4.3]{NT2024}, we obtain,
for \(r,s\in[m]\),
\begin{equation}\label{eq:omega-insertion-data}
  \begin{aligned}
    J_r(a)&=\operatorname{INT}(v_r,\mathsf P_\Omega(a)),\\
    e_r(a)&=\rho_{F_\Omega(a)}(v_r),\\
    r\trl_a s&\quad\Longleftrightarrow\quad
      v_r\text{ is a child of }v_s.
  \end{aligned}
\end{equation}
The last equivalence identifies orders that respect \(\mathcal H(a)\)
with orders in which each child is read before its parent.
Reading the \(P\)-symbol in these orders gives all the words in its
\(\Omega\)-class by
\cite[Theorem~4.1 and Equation~(4.1)]{NT2024}.
To keep the endpoints weakly decreasing, we must read nodes with
larger flags first. When the flags agree, the nodes lie on a
left-child chain, so a descendant must be read before its ancestor.
These requirements determine exactly the rightmost reading.
Consequently, the rightmost reading of \(\mathsf P_\Omega(a)\),
with decorations removed, gives the unique endpoint-decreasing
representative of the \(\Omega\)-class of \(a\).

\begin{table}
  \centering
  \small
  \renewcommand{\arraystretch}{1.25}
  \begin{tabular}{@{}ll@{}}
    \toprule
    Marked blocks & \(P\)-symbol\\
    \midrule
    \(J_r(a)\) & \(\operatorname{INT}(v_r,F)\)\\
    \(e_r(a)\) & \(\rho_F(v_r)\)\\
    New block position \(p_r\)
      & \(\ell_F(v_r)\)\\
    \(r\trl_a s\) & \(v_r\) is an internal child of \(v_s\)\\
    \(\mathcal H(a)\) & Internal-node poset, children before parents\\
    \bottomrule
  \end{tabular}
  \caption{The dictionary with the \(P\)-symbol, where
    \(F=F_\Omega(a)\).}
  \label{tab:omega-dictionary}
\end{table}

In particular, the dictionary gives
\begin{equation}\label{eq:omega-positive-support}
  e_r(a)>0\text{ for all }r
  \quad\Longleftrightarrow\quad
  F_\Omega(a)\in\mathsf{IndFor}_+.
\end{equation}
For a nonempty endpoint-decreasing word, this is equivalent to
positivity of its final endpoint.

\begin{exam}\label{exa:4}
For the word \(a=(4,3,1,2,4,3,4)\) of
\Cref{ex:nt-omega-insertion}, \Cref{fig:3} pairs each \(P\)-symbol
with its marked-block state. At insertion \(4\), the two roots
\(v_2,v_3\) become children of \(v_4\), just as the marks
\(2,3\) are removed. The new component is
\(J_4(a)=\{1,2,3,4\}\), with \(e_4(a)=1\).
The full endpoint sequence is
\[
  (e_1(a),\ldots,e_7(a))=(4,3,1,1,1,0,0),
\]
so \(a\) is endpoint-decreasing.
\end{exam}

\begin{figure}
\centering
\begin{tikzpicture}[x=1cm,y=1cm]
\node[font=\small] at (.85,.55) {Insertion};
\node[font=\small] at (4.88,.55) {Nadeau--Tewari \(P\)-symbol};
\node[font=\small] at (11.45,.55) {Marked blocks};
\begin{scope}[shift={(2.5cm,-0.87cm)},x=0.68cm,y=0.68cm]
\draw[gray!65,line width=.35pt] (-.2,0)--(7.2,0);
\draw[black!80,line width=.4pt] (-0.065,-.065)--(0.065,.065) (-0.065,.065)--(0.065,-.065);
\node[font=\scriptsize,anchor=north,inner sep=1pt] at (0,-.10) {0};
\draw[black!80,line width=.4pt] (0.935,-.065)--(1.065,.065) (0.935,.065)--(1.065,-.065);
\node[font=\scriptsize,anchor=north,inner sep=1pt] at (1,-.10) {1};
\draw[black!80,line width=.4pt] (1.935,-.065)--(2.065,.065) (1.935,.065)--(2.065,-.065);
\node[font=\scriptsize,anchor=north,inner sep=1pt] at (2,-.10) {2};
\draw[black!80,line width=.4pt] (2.935,-.065)--(3.065,.065) (2.935,.065)--(3.065,-.065);
\node[font=\scriptsize,anchor=north,inner sep=1pt] at (3,-.10) {3};
\draw[black!80,line width=.4pt] (3.935,-.065)--(4.065,.065) (3.935,.065)--(4.065,-.065);
\node[font=\scriptsize,anchor=north,inner sep=1pt] at (4,-.10) {4};
\draw[black!80,line width=.4pt] (4.935,-.065)--(5.065,.065) (4.935,.065)--(5.065,-.065);
\node[font=\scriptsize,anchor=north,inner sep=1pt] at (5,-.10) {5};
\draw[black!80,line width=.4pt] (5.935,-.065)--(6.065,.065) (5.935,.065)--(6.065,-.065);
\node[font=\scriptsize,anchor=north,inner sep=1pt] at (6,-.10) {6};
\draw[black!80,line width=.4pt] (6.935,-.065)--(7.065,.065) (6.935,.065)--(7.065,-.065);
\node[font=\scriptsize,anchor=north,inner sep=1pt] at (7,-.10) {7};
\draw[blue!75!black,line width=.55pt] (4.5,0.5)--(4,0);
\draw[blue!75!black,line width=.55pt] (4.5,0.5)--(5,0);
\fill[blue!75!black] (4.5,0.5) circle[radius=1.4pt];
\node[font=\fontsize{7}{8}\selectfont,blue!75!black,anchor=south,inner sep=.3pt] at (4.5,0.63) {\(4^{[1]}\)};
\end{scope}
\begin{scope}[shift={(9cm,-0.7cm)},x=0.7cm,y=1cm]
\path[draw=gray!40,densely dotted,line width=.35pt] (0,-.21) rectangle (1,.21);
\node[font=\scriptsize,gray!80,anchor=north] at (0.5,-.23) {0};
\path[draw=gray!40,densely dotted,line width=.35pt] (1,-.21) rectangle (2,.21);
\node[font=\scriptsize,gray!80,anchor=north] at (1.5,-.23) {1};
\path[draw=gray!40,densely dotted,line width=.35pt] (2,-.21) rectangle (3,.21);
\node[font=\scriptsize,gray!80,anchor=north] at (2.5,-.23) {2};
\path[draw=gray!40,densely dotted,line width=.35pt] (3,-.21) rectangle (4,.21);
\node[font=\scriptsize,gray!80,anchor=north] at (3.5,-.23) {3};
\path[draw=blue!75!black,fill=blue!12,line width=.8pt] (4,-.21) rectangle (5,.21);
\node[font=\scriptsize,blue!75!black] at (4.5,0) {1};
\node[font=\scriptsize,gray!80,anchor=north] at (4.5,-.23) {4};
\path[draw=gray!40,densely dotted,line width=.35pt] (5,-.21) rectangle (6,.21);
\node[font=\scriptsize,gray!80,anchor=north] at (5.5,-.23) {5};
\path[draw=gray!40,densely dotted,line width=.35pt] (6,-.21) rectangle (7,.21);
\node[font=\scriptsize,gray!80,anchor=north] at (6.5,-.23) {6};
\draw[blue!75!black,line width=.45pt] (4,-.61)--(4,-.67)--(5,-.67)--(5,-.61);
\node[font=\scriptsize,anchor=north] at (3.5,-.64) {\(J_1(a)=\{4\},\quad e_1(a)=4\)};
\end{scope}
\node[font=\small,align=center] at (.85,-0.7) {\(r=1\)\\[2pt]\(4^{[1]}\)};
\node[font=\small] at (7.85,-0.7) {\(\longleftrightarrow\)};
\draw[gray!20] (0,-1.78)--(14.25,-1.78);
\begin{scope}[shift={(2.5cm,-3.03cm)},x=0.68cm,y=0.68cm]
\draw[gray!65,line width=.35pt] (-.2,0)--(7.2,0);
\draw[black!80,line width=.4pt] (-0.065,-.065)--(0.065,.065) (-0.065,.065)--(0.065,-.065);
\node[font=\scriptsize,anchor=north,inner sep=1pt] at (0,-.10) {0};
\draw[black!80,line width=.4pt] (0.935,-.065)--(1.065,.065) (0.935,.065)--(1.065,-.065);
\node[font=\scriptsize,anchor=north,inner sep=1pt] at (1,-.10) {1};
\draw[black!80,line width=.4pt] (1.935,-.065)--(2.065,.065) (1.935,.065)--(2.065,-.065);
\node[font=\scriptsize,anchor=north,inner sep=1pt] at (2,-.10) {2};
\draw[black!80,line width=.4pt] (2.935,-.065)--(3.065,.065) (2.935,.065)--(3.065,-.065);
\node[font=\scriptsize,anchor=north,inner sep=1pt] at (3,-.10) {3};
\draw[black!80,line width=.4pt] (3.935,-.065)--(4.065,.065) (3.935,.065)--(4.065,-.065);
\node[font=\scriptsize,anchor=north,inner sep=1pt] at (4,-.10) {4};
\draw[black!80,line width=.4pt] (4.935,-.065)--(5.065,.065) (4.935,.065)--(5.065,-.065);
\node[font=\scriptsize,anchor=north,inner sep=1pt] at (5,-.10) {5};
\draw[black!80,line width=.4pt] (5.935,-.065)--(6.065,.065) (5.935,.065)--(6.065,-.065);
\node[font=\scriptsize,anchor=north,inner sep=1pt] at (6,-.10) {6};
\draw[black!80,line width=.4pt] (6.935,-.065)--(7.065,.065) (6.935,.065)--(7.065,-.065);
\node[font=\scriptsize,anchor=north,inner sep=1pt] at (7,-.10) {7};
\draw[black,line width=.55pt] (4.5,0.5)--(4,0);
\draw[black,line width=.55pt] (4.5,0.5)--(5,0);
\draw[blue!75!black,line width=.55pt] (4,1)--(3,0);
\draw[blue!75!black,line width=.55pt] (4,1)--(4.5,0.5);
\fill[black] (4.5,0.5) circle[radius=1.4pt];
\node[font=\fontsize{7}{8}\selectfont,black,anchor=west,inner sep=.3pt] at (4.67,0.5) {\(4^{[1]}\)};
\fill[blue!75!black] (4,1) circle[radius=1.4pt];
\node[font=\fontsize{7}{8}\selectfont,blue!75!black,anchor=south,inner sep=.3pt] at (4,1.13) {\(3^{[1]}\)};
\end{scope}
\begin{scope}[shift={(9cm,-2.69cm)},x=0.7cm,y=1cm]
\path[draw=gray!40,densely dotted,line width=.35pt] (0,-.21) rectangle (1,.21);
\node[font=\scriptsize,gray!80,anchor=north] at (0.5,-.23) {0};
\path[draw=gray!40,densely dotted,line width=.35pt] (1,-.21) rectangle (2,.21);
\node[font=\scriptsize,gray!80,anchor=north] at (1.5,-.23) {1};
\path[draw=gray!40,densely dotted,line width=.35pt] (2,-.21) rectangle (3,.21);
\node[font=\scriptsize,gray!80,anchor=north] at (2.5,-.23) {2};
\path[draw=blue!75!black,fill=blue!12,line width=.8pt] (3,-.21) rectangle (4,.21);
\node[font=\scriptsize,blue!75!black] at (3.5,0) {2};
\node[font=\scriptsize,gray!80,anchor=north] at (3.5,-.23) {3};
\path[draw=black,fill=white,line width=.5pt] (4,-.21) rectangle (5,.21);
\node[font=\scriptsize,gray!80,anchor=north] at (4.5,-.23) {4};
\path[draw=gray!40,densely dotted,line width=.35pt] (5,-.21) rectangle (6,.21);
\node[font=\scriptsize,gray!80,anchor=north] at (5.5,-.23) {5};
\path[draw=gray!40,densely dotted,line width=.35pt] (6,-.21) rectangle (7,.21);
\node[font=\scriptsize,gray!80,anchor=north] at (6.5,-.23) {6};
\draw[blue!75!black,line width=.45pt] (3,-.61)--(3,-.67)--(5,-.67)--(5,-.61);
\node[font=\scriptsize,anchor=north] at (3.5,-.64) {\(J_2(a)=\{3,4\},\quad e_2(a)=3\)};
\end{scope}
\node[font=\small,align=center] at (.85,-2.69) {\(r=2\)\\[2pt]\(3^{[1]}\)};
\node[font=\small] at (7.85,-2.69) {\(\longleftrightarrow\)};
\draw[gray!20] (0,-3.77)--(14.25,-3.77);
\begin{scope}[shift={(2.5cm,-5.02cm)},x=0.68cm,y=0.68cm]
\draw[gray!65,line width=.35pt] (-.2,0)--(7.2,0);
\draw[black!80,line width=.4pt] (-0.065,-.065)--(0.065,.065) (-0.065,.065)--(0.065,-.065);
\node[font=\scriptsize,anchor=north,inner sep=1pt] at (0,-.10) {0};
\draw[black!80,line width=.4pt] (0.935,-.065)--(1.065,.065) (0.935,.065)--(1.065,-.065);
\node[font=\scriptsize,anchor=north,inner sep=1pt] at (1,-.10) {1};
\draw[black!80,line width=.4pt] (1.935,-.065)--(2.065,.065) (1.935,.065)--(2.065,-.065);
\node[font=\scriptsize,anchor=north,inner sep=1pt] at (2,-.10) {2};
\draw[black!80,line width=.4pt] (2.935,-.065)--(3.065,.065) (2.935,.065)--(3.065,-.065);
\node[font=\scriptsize,anchor=north,inner sep=1pt] at (3,-.10) {3};
\draw[black!80,line width=.4pt] (3.935,-.065)--(4.065,.065) (3.935,.065)--(4.065,-.065);
\node[font=\scriptsize,anchor=north,inner sep=1pt] at (4,-.10) {4};
\draw[black!80,line width=.4pt] (4.935,-.065)--(5.065,.065) (4.935,.065)--(5.065,-.065);
\node[font=\scriptsize,anchor=north,inner sep=1pt] at (5,-.10) {5};
\draw[black!80,line width=.4pt] (5.935,-.065)--(6.065,.065) (5.935,.065)--(6.065,-.065);
\node[font=\scriptsize,anchor=north,inner sep=1pt] at (6,-.10) {6};
\draw[black!80,line width=.4pt] (6.935,-.065)--(7.065,.065) (6.935,.065)--(7.065,-.065);
\node[font=\scriptsize,anchor=north,inner sep=1pt] at (7,-.10) {7};
\draw[black,line width=.55pt] (4.5,0.5)--(4,0);
\draw[black,line width=.55pt] (4.5,0.5)--(5,0);
\draw[black,line width=.55pt] (4,1)--(3,0);
\draw[black,line width=.55pt] (4,1)--(4.5,0.5);
\draw[blue!75!black,line width=.55pt] (1.5,0.5)--(1,0);
\draw[blue!75!black,line width=.55pt] (1.5,0.5)--(2,0);
\fill[black] (4.5,0.5) circle[radius=1.4pt];
\node[font=\fontsize{7}{8}\selectfont,black,anchor=west,inner sep=.3pt] at (4.67,0.5) {\(4^{[1]}\)};
\draw[black,fill=gray!30] (4,1) circle[radius=2.1pt];
\node[font=\fontsize{7}{8}\selectfont,black,anchor=south,inner sep=.3pt] at (4,1.13) {\(3^{[1]}\)};
\fill[blue!75!black] (1.5,0.5) circle[radius=1.4pt];
\node[font=\fontsize{7}{8}\selectfont,blue!75!black,anchor=south,inner sep=.3pt] at (1.5,0.63) {\(1^{[1]}\)};
\end{scope}
\begin{scope}[shift={(9cm,-4.68cm)},x=0.7cm,y=1cm]
\path[draw=gray!40,densely dotted,line width=.35pt] (0,-.21) rectangle (1,.21);
\node[font=\scriptsize,gray!80,anchor=north] at (0.5,-.23) {0};
\path[draw=blue!75!black,fill=blue!12,line width=.8pt] (1,-.21) rectangle (2,.21);
\node[font=\scriptsize,blue!75!black] at (1.5,0) {3};
\node[font=\scriptsize,gray!80,anchor=north] at (1.5,-.23) {1};
\path[draw=gray!40,densely dotted,line width=.35pt] (2,-.21) rectangle (3,.21);
\node[font=\scriptsize,gray!80,anchor=north] at (2.5,-.23) {2};
\path[draw=black,fill=gray!25,line width=.5pt] (3,-.21) rectangle (4,.21);
\node[font=\scriptsize,black] at (3.5,0) {2};
\node[font=\scriptsize,gray!80,anchor=north] at (3.5,-.23) {3};
\path[draw=black,fill=white,line width=.5pt] (4,-.21) rectangle (5,.21);
\node[font=\scriptsize,gray!80,anchor=north] at (4.5,-.23) {4};
\path[draw=gray!40,densely dotted,line width=.35pt] (5,-.21) rectangle (6,.21);
\node[font=\scriptsize,gray!80,anchor=north] at (5.5,-.23) {5};
\path[draw=gray!40,densely dotted,line width=.35pt] (6,-.21) rectangle (7,.21);
\node[font=\scriptsize,gray!80,anchor=north] at (6.5,-.23) {6};
\draw[blue!75!black,line width=.45pt] (1,-.61)--(1,-.67)--(2,-.67)--(2,-.61);
\node[font=\scriptsize,anchor=north] at (3.5,-.64) {\(J_3(a)=\{1\},\quad e_3(a)=1\)};
\end{scope}
\node[font=\small,align=center] at (.85,-4.68) {\(r=3\)\\[2pt]\(1^{[1]}\)};
\node[font=\small] at (7.85,-4.68) {\(\longleftrightarrow\)};
\draw[gray!20] (0,-5.76)--(14.25,-5.76);
\begin{scope}[shift={(2.5cm,-7.69cm)},x=0.68cm,y=0.68cm]
\draw[gray!65,line width=.35pt] (-.2,0)--(7.2,0);
\draw[black!80,line width=.4pt] (-0.065,-.065)--(0.065,.065) (-0.065,.065)--(0.065,-.065);
\node[font=\scriptsize,anchor=north,inner sep=1pt] at (0,-.10) {0};
\draw[black!80,line width=.4pt] (0.935,-.065)--(1.065,.065) (0.935,.065)--(1.065,-.065);
\node[font=\scriptsize,anchor=north,inner sep=1pt] at (1,-.10) {1};
\draw[black!80,line width=.4pt] (1.935,-.065)--(2.065,.065) (1.935,.065)--(2.065,-.065);
\node[font=\scriptsize,anchor=north,inner sep=1pt] at (2,-.10) {2};
\draw[black!80,line width=.4pt] (2.935,-.065)--(3.065,.065) (2.935,.065)--(3.065,-.065);
\node[font=\scriptsize,anchor=north,inner sep=1pt] at (3,-.10) {3};
\draw[black!80,line width=.4pt] (3.935,-.065)--(4.065,.065) (3.935,.065)--(4.065,-.065);
\node[font=\scriptsize,anchor=north,inner sep=1pt] at (4,-.10) {4};
\draw[black!80,line width=.4pt] (4.935,-.065)--(5.065,.065) (4.935,.065)--(5.065,-.065);
\node[font=\scriptsize,anchor=north,inner sep=1pt] at (5,-.10) {5};
\draw[black!80,line width=.4pt] (5.935,-.065)--(6.065,.065) (5.935,.065)--(6.065,-.065);
\node[font=\scriptsize,anchor=north,inner sep=1pt] at (6,-.10) {6};
\draw[black!80,line width=.4pt] (6.935,-.065)--(7.065,.065) (6.935,.065)--(7.065,-.065);
\node[font=\scriptsize,anchor=north,inner sep=1pt] at (7,-.10) {7};
\draw[black,line width=.55pt] (4.5,0.5)--(4,0);
\draw[black,line width=.55pt] (4.5,0.5)--(5,0);
\draw[black,line width=.55pt] (4,1)--(3,0);
\draw[black,line width=.55pt] (4,1)--(4.5,0.5);
\draw[black,line width=.55pt] (1.5,0.5)--(1,0);
\draw[black,line width=.55pt] (1.5,0.5)--(2,0);
\draw[blue!75!black,line width=.55pt] (3,2)--(1.5,0.5);
\draw[blue!75!black,line width=.55pt] (3,2)--(4,1);
\fill[black] (4.5,0.5) circle[radius=1.4pt];
\node[font=\fontsize{7}{8}\selectfont,black,anchor=west,inner sep=.3pt] at (4.67,0.5) {\(4^{[1]}\)};
\fill[black] (4,1) circle[radius=1.4pt];
\node[font=\fontsize{7}{8}\selectfont,black,anchor=west,inner sep=.3pt] at (4.17,1) {\(3^{[1]}\)};
\fill[black] (1.5,0.5) circle[radius=1.4pt];
\node[font=\fontsize{7}{8}\selectfont,black,anchor=east,inner sep=.3pt] at (1.33,0.5) {\(1^{[1]}\)};
\fill[blue!75!black] (3,2) circle[radius=1.4pt];
\node[font=\fontsize{7}{8}\selectfont,blue!75!black,anchor=south,inner sep=.3pt] at (3,2.13) {\(2^{[1]}\)};
\end{scope}
\begin{scope}[shift={(9cm,-7.01cm)},x=0.7cm,y=1cm]
\path[draw=gray!40,densely dotted,line width=.35pt] (0,-.21) rectangle (1,.21);
\node[font=\scriptsize,gray!80,anchor=north] at (0.5,-.23) {0};
\path[draw=black,fill=white,line width=.5pt] (1,-.21) rectangle (2,.21);
\node[font=\scriptsize,gray!80,anchor=north] at (1.5,-.23) {1};
\path[draw=blue!75!black,fill=blue!12,line width=.8pt] (2,-.21) rectangle (3,.21);
\node[font=\scriptsize,blue!75!black] at (2.5,0) {4};
\node[font=\scriptsize,gray!80,anchor=north] at (2.5,-.23) {2};
\path[draw=black,fill=white,line width=.5pt] (3,-.21) rectangle (4,.21);
\node[font=\scriptsize,gray!80,anchor=north] at (3.5,-.23) {3};
\path[draw=black,fill=white,line width=.5pt] (4,-.21) rectangle (5,.21);
\node[font=\scriptsize,gray!80,anchor=north] at (4.5,-.23) {4};
\path[draw=gray!40,densely dotted,line width=.35pt] (5,-.21) rectangle (6,.21);
\node[font=\scriptsize,gray!80,anchor=north] at (5.5,-.23) {5};
\path[draw=gray!40,densely dotted,line width=.35pt] (6,-.21) rectangle (7,.21);
\node[font=\scriptsize,gray!80,anchor=north] at (6.5,-.23) {6};
\draw[blue!75!black,line width=.45pt] (1,-.61)--(1,-.67)--(5,-.67)--(5,-.61);
\node[font=\scriptsize,anchor=north] at (3.5,-.64) {\(J_4(a)=\{1,\ldots,4\},\quad e_4(a)=1\)};
\end{scope}
\node[font=\small,align=center] at (.85,-7.01) {\(r=4\)\\[2pt]\(2^{[1]}\)};
\node[font=\small] at (7.85,-7.01) {\(\longleftrightarrow\)};
\draw[gray!20] (0,-8.11)--(14.25,-8.11);
\begin{scope}[shift={(2.5cm,-10.38cm)},x=0.68cm,y=0.68cm]
\draw[gray!65,line width=.35pt] (-.2,0)--(7.2,0);
\draw[black!80,line width=.4pt] (-0.065,-.065)--(0.065,.065) (-0.065,.065)--(0.065,-.065);
\node[font=\scriptsize,anchor=north,inner sep=1pt] at (0,-.10) {0};
\draw[black!80,line width=.4pt] (0.935,-.065)--(1.065,.065) (0.935,.065)--(1.065,-.065);
\node[font=\scriptsize,anchor=north,inner sep=1pt] at (1,-.10) {1};
\draw[black!80,line width=.4pt] (1.935,-.065)--(2.065,.065) (1.935,.065)--(2.065,-.065);
\node[font=\scriptsize,anchor=north,inner sep=1pt] at (2,-.10) {2};
\draw[black!80,line width=.4pt] (2.935,-.065)--(3.065,.065) (2.935,.065)--(3.065,-.065);
\node[font=\scriptsize,anchor=north,inner sep=1pt] at (3,-.10) {3};
\draw[black!80,line width=.4pt] (3.935,-.065)--(4.065,.065) (3.935,.065)--(4.065,-.065);
\node[font=\scriptsize,anchor=north,inner sep=1pt] at (4,-.10) {4};
\draw[black!80,line width=.4pt] (4.935,-.065)--(5.065,.065) (4.935,.065)--(5.065,-.065);
\node[font=\scriptsize,anchor=north,inner sep=1pt] at (5,-.10) {5};
\draw[black!80,line width=.4pt] (5.935,-.065)--(6.065,.065) (5.935,.065)--(6.065,-.065);
\node[font=\scriptsize,anchor=north,inner sep=1pt] at (6,-.10) {6};
\draw[black!80,line width=.4pt] (6.935,-.065)--(7.065,.065) (6.935,.065)--(7.065,-.065);
\node[font=\scriptsize,anchor=north,inner sep=1pt] at (7,-.10) {7};
\draw[black,line width=.55pt] (4.5,0.5)--(4,0);
\draw[black,line width=.55pt] (4.5,0.5)--(5,0);
\draw[black,line width=.55pt] (4,1)--(3,0);
\draw[black,line width=.55pt] (4,1)--(4.5,0.5);
\draw[black,line width=.55pt] (1.5,0.5)--(1,0);
\draw[black,line width=.55pt] (1.5,0.5)--(2,0);
\draw[black,line width=.55pt] (3,2)--(1.5,0.5);
\draw[black,line width=.55pt] (3,2)--(4,1);
\draw[blue!75!black,line width=.55pt] (3.5,2.5)--(3,2);
\draw[blue!75!black,line width=.55pt] (3.5,2.5)--(6,0);
\fill[black] (4.5,0.5) circle[radius=1.4pt];
\node[font=\fontsize{7}{8}\selectfont,black,anchor=west,inner sep=.3pt] at (4.67,0.5) {\(4^{[1]}\)};
\fill[black] (4,1) circle[radius=1.4pt];
\node[font=\fontsize{7}{8}\selectfont,black,anchor=west,inner sep=.3pt] at (4.17,1) {\(3^{[1]}\)};
\fill[black] (1.5,0.5) circle[radius=1.4pt];
\node[font=\fontsize{7}{8}\selectfont,black,anchor=east,inner sep=.3pt] at (1.33,0.5) {\(1^{[1]}\)};
\fill[black] (3,2) circle[radius=1.4pt];
\node[font=\fontsize{7}{8}\selectfont,black,anchor=east,inner sep=.3pt] at (2.83,2) {\(2^{[1]}\)};
\fill[blue!75!black] (3.5,2.5) circle[radius=1.4pt];
\node[font=\fontsize{7}{8}\selectfont,blue!75!black,anchor=south,inner sep=.3pt] at (3.5,2.63) {\(4^{[2]}\)};
\end{scope}
\begin{scope}[shift={(9cm,-9.53cm)},x=0.7cm,y=1cm]
\path[draw=gray!40,densely dotted,line width=.35pt] (0,-.21) rectangle (1,.21);
\node[font=\scriptsize,gray!80,anchor=north] at (0.5,-.23) {0};
\path[draw=black,fill=white,line width=.5pt] (1,-.21) rectangle (2,.21);
\node[font=\scriptsize,gray!80,anchor=north] at (1.5,-.23) {1};
\path[draw=black,fill=white,line width=.5pt] (2,-.21) rectangle (3,.21);
\node[font=\scriptsize,gray!80,anchor=north] at (2.5,-.23) {2};
\path[draw=black,fill=white,line width=.5pt] (3,-.21) rectangle (4,.21);
\node[font=\scriptsize,gray!80,anchor=north] at (3.5,-.23) {3};
\path[draw=black,fill=white,line width=.5pt] (4,-.21) rectangle (5,.21);
\node[font=\scriptsize,gray!80,anchor=north] at (4.5,-.23) {4};
\path[draw=blue!75!black,fill=blue!12,line width=.8pt] (5,-.21) rectangle (6,.21);
\node[font=\scriptsize,blue!75!black] at (5.5,0) {5};
\node[font=\scriptsize,gray!80,anchor=north] at (5.5,-.23) {5};
\path[draw=gray!40,densely dotted,line width=.35pt] (6,-.21) rectangle (7,.21);
\node[font=\scriptsize,gray!80,anchor=north] at (6.5,-.23) {6};
\draw[blue!75!black,line width=.45pt] (1,-.61)--(1,-.67)--(6,-.67)--(6,-.61);
\node[font=\scriptsize,anchor=north] at (3.5,-.64) {\(J_5(a)=\{1,\ldots,5\},\quad e_5(a)=1\)};
\end{scope}
\node[font=\small,align=center] at (.85,-9.53) {\(r=5\)\\[2pt]\(4^{[2]}\)};
\node[font=\small] at (7.85,-9.53) {\(\longleftrightarrow\)};
\draw[gray!20] (0,-10.8)--(14.25,-10.8);
\begin{scope}[shift={(2.5cm,-13.41cm)},x=0.68cm,y=0.68cm]
\draw[gray!65,line width=.35pt] (-.2,0)--(7.2,0);
\draw[black!80,line width=.4pt] (-0.065,-.065)--(0.065,.065) (-0.065,.065)--(0.065,-.065);
\node[font=\scriptsize,anchor=north,inner sep=1pt] at (0,-.10) {0};
\draw[black!80,line width=.4pt] (0.935,-.065)--(1.065,.065) (0.935,.065)--(1.065,-.065);
\node[font=\scriptsize,anchor=north,inner sep=1pt] at (1,-.10) {1};
\draw[black!80,line width=.4pt] (1.935,-.065)--(2.065,.065) (1.935,.065)--(2.065,-.065);
\node[font=\scriptsize,anchor=north,inner sep=1pt] at (2,-.10) {2};
\draw[black!80,line width=.4pt] (2.935,-.065)--(3.065,.065) (2.935,.065)--(3.065,-.065);
\node[font=\scriptsize,anchor=north,inner sep=1pt] at (3,-.10) {3};
\draw[black!80,line width=.4pt] (3.935,-.065)--(4.065,.065) (3.935,.065)--(4.065,-.065);
\node[font=\scriptsize,anchor=north,inner sep=1pt] at (4,-.10) {4};
\draw[black!80,line width=.4pt] (4.935,-.065)--(5.065,.065) (4.935,.065)--(5.065,-.065);
\node[font=\scriptsize,anchor=north,inner sep=1pt] at (5,-.10) {5};
\draw[black!80,line width=.4pt] (5.935,-.065)--(6.065,.065) (5.935,.065)--(6.065,-.065);
\node[font=\scriptsize,anchor=north,inner sep=1pt] at (6,-.10) {6};
\draw[black!80,line width=.4pt] (6.935,-.065)--(7.065,.065) (6.935,.065)--(7.065,-.065);
\node[font=\scriptsize,anchor=north,inner sep=1pt] at (7,-.10) {7};
\draw[black,line width=.55pt] (4.5,0.5)--(4,0);
\draw[black,line width=.55pt] (4.5,0.5)--(5,0);
\draw[black,line width=.55pt] (4,1)--(3,0);
\draw[black,line width=.55pt] (4,1)--(4.5,0.5);
\draw[black,line width=.55pt] (1.5,0.5)--(1,0);
\draw[black,line width=.55pt] (1.5,0.5)--(2,0);
\draw[black,line width=.55pt] (3,2)--(1.5,0.5);
\draw[black,line width=.55pt] (3,2)--(4,1);
\draw[black,line width=.55pt] (3.5,2.5)--(3,2);
\draw[black,line width=.55pt] (3.5,2.5)--(6,0);
\draw[blue!75!black,line width=.55pt] (3,3)--(0,0);
\draw[blue!75!black,line width=.55pt] (3,3)--(3.5,2.5);
\fill[black] (4.5,0.5) circle[radius=1.4pt];
\node[font=\fontsize{7}{8}\selectfont,black,anchor=west,inner sep=.3pt] at (4.67,0.5) {\(4^{[1]}\)};
\fill[black] (4,1) circle[radius=1.4pt];
\node[font=\fontsize{7}{8}\selectfont,black,anchor=west,inner sep=.3pt] at (4.17,1) {\(3^{[1]}\)};
\fill[black] (1.5,0.5) circle[radius=1.4pt];
\node[font=\fontsize{7}{8}\selectfont,black,anchor=east,inner sep=.3pt] at (1.33,0.5) {\(1^{[1]}\)};
\fill[black] (3,2) circle[radius=1.4pt];
\node[font=\fontsize{7}{8}\selectfont,black,anchor=east,inner sep=.3pt] at (2.83,2) {\(2^{[1]}\)};
\fill[black] (3.5,2.5) circle[radius=1.4pt];
\node[font=\fontsize{7}{8}\selectfont,black,anchor=west,inner sep=.3pt] at (3.67,2.5) {\(4^{[2]}\)};
\fill[blue!75!black] (3,3) circle[radius=1.4pt];
\node[font=\fontsize{7}{8}\selectfont,blue!75!black,anchor=south,inner sep=.3pt] at (3,3.13) {\(3^{[2]}\)};
\end{scope}
\begin{scope}[shift={(9cm,-12.39cm)},x=0.7cm,y=1cm]
\path[draw=blue!75!black,fill=blue!12,line width=.8pt] (0,-.21) rectangle (1,.21);
\node[font=\scriptsize,blue!75!black] at (0.5,0) {6};
\node[font=\scriptsize,gray!80,anchor=north] at (0.5,-.23) {0};
\path[draw=black,fill=white,line width=.5pt] (1,-.21) rectangle (2,.21);
\node[font=\scriptsize,gray!80,anchor=north] at (1.5,-.23) {1};
\path[draw=black,fill=white,line width=.5pt] (2,-.21) rectangle (3,.21);
\node[font=\scriptsize,gray!80,anchor=north] at (2.5,-.23) {2};
\path[draw=black,fill=white,line width=.5pt] (3,-.21) rectangle (4,.21);
\node[font=\scriptsize,gray!80,anchor=north] at (3.5,-.23) {3};
\path[draw=black,fill=white,line width=.5pt] (4,-.21) rectangle (5,.21);
\node[font=\scriptsize,gray!80,anchor=north] at (4.5,-.23) {4};
\path[draw=black,fill=white,line width=.5pt] (5,-.21) rectangle (6,.21);
\node[font=\scriptsize,gray!80,anchor=north] at (5.5,-.23) {5};
\path[draw=gray!40,densely dotted,line width=.35pt] (6,-.21) rectangle (7,.21);
\node[font=\scriptsize,gray!80,anchor=north] at (6.5,-.23) {6};
\draw[blue!75!black,line width=.45pt] (0,-.61)--(0,-.67)--(6,-.67)--(6,-.61);
\node[font=\scriptsize,anchor=north] at (3.5,-.64) {\(J_6(a)=\{0,\ldots,5\},\quad e_6(a)=0\)};
\end{scope}
\node[font=\small,align=center] at (.85,-12.39) {\(r=6\)\\[2pt]\(3^{[2]}\)};
\node[font=\small] at (7.85,-12.39) {\(\longleftrightarrow\)};
\draw[gray!20] (0,-13.83)--(14.25,-13.83);
\begin{scope}[shift={(2.5cm,-16.78cm)},x=0.68cm,y=0.68cm]
\draw[gray!65,line width=.35pt] (-.2,0)--(7.2,0);
\draw[black!80,line width=.4pt] (-0.065,-.065)--(0.065,.065) (-0.065,.065)--(0.065,-.065);
\node[font=\scriptsize,anchor=north,inner sep=1pt] at (0,-.10) {0};
\draw[black!80,line width=.4pt] (0.935,-.065)--(1.065,.065) (0.935,.065)--(1.065,-.065);
\node[font=\scriptsize,anchor=north,inner sep=1pt] at (1,-.10) {1};
\draw[black!80,line width=.4pt] (1.935,-.065)--(2.065,.065) (1.935,.065)--(2.065,-.065);
\node[font=\scriptsize,anchor=north,inner sep=1pt] at (2,-.10) {2};
\draw[black!80,line width=.4pt] (2.935,-.065)--(3.065,.065) (2.935,.065)--(3.065,-.065);
\node[font=\scriptsize,anchor=north,inner sep=1pt] at (3,-.10) {3};
\draw[black!80,line width=.4pt] (3.935,-.065)--(4.065,.065) (3.935,.065)--(4.065,-.065);
\node[font=\scriptsize,anchor=north,inner sep=1pt] at (4,-.10) {4};
\draw[black!80,line width=.4pt] (4.935,-.065)--(5.065,.065) (4.935,.065)--(5.065,-.065);
\node[font=\scriptsize,anchor=north,inner sep=1pt] at (5,-.10) {5};
\draw[black!80,line width=.4pt] (5.935,-.065)--(6.065,.065) (5.935,.065)--(6.065,-.065);
\node[font=\scriptsize,anchor=north,inner sep=1pt] at (6,-.10) {6};
\draw[black!80,line width=.4pt] (6.935,-.065)--(7.065,.065) (6.935,.065)--(7.065,-.065);
\node[font=\scriptsize,anchor=north,inner sep=1pt] at (7,-.10) {7};
\draw[black,line width=.55pt] (4.5,0.5)--(4,0);
\draw[black,line width=.55pt] (4.5,0.5)--(5,0);
\draw[black,line width=.55pt] (4,1)--(3,0);
\draw[black,line width=.55pt] (4,1)--(4.5,0.5);
\draw[black,line width=.55pt] (1.5,0.5)--(1,0);
\draw[black,line width=.55pt] (1.5,0.5)--(2,0);
\draw[black,line width=.55pt] (3,2)--(1.5,0.5);
\draw[black,line width=.55pt] (3,2)--(4,1);
\draw[black,line width=.55pt] (3.5,2.5)--(3,2);
\draw[black,line width=.55pt] (3.5,2.5)--(6,0);
\draw[black,line width=.55pt] (3,3)--(0,0);
\draw[black,line width=.55pt] (3,3)--(3.5,2.5);
\draw[blue!75!black,line width=.55pt] (3.5,3.5)--(3,3);
\draw[blue!75!black,line width=.55pt] (3.5,3.5)--(7,0);
\fill[black] (4.5,0.5) circle[radius=1.4pt];
\node[font=\fontsize{7}{8}\selectfont,black,anchor=west,inner sep=.3pt] at (4.67,0.5) {\(4^{[1]}\)};
\fill[black] (4,1) circle[radius=1.4pt];
\node[font=\fontsize{7}{8}\selectfont,black,anchor=west,inner sep=.3pt] at (4.17,1) {\(3^{[1]}\)};
\fill[black] (1.5,0.5) circle[radius=1.4pt];
\node[font=\fontsize{7}{8}\selectfont,black,anchor=east,inner sep=.3pt] at (1.33,0.5) {\(1^{[1]}\)};
\fill[black] (3,2) circle[radius=1.4pt];
\node[font=\fontsize{7}{8}\selectfont,black,anchor=east,inner sep=.3pt] at (2.83,2) {\(2^{[1]}\)};
\fill[black] (3.5,2.5) circle[radius=1.4pt];
\node[font=\fontsize{7}{8}\selectfont,black,anchor=west,inner sep=.3pt] at (3.67,2.5) {\(4^{[2]}\)};
\fill[black] (3,3) circle[radius=1.4pt];
\node[font=\fontsize{7}{8}\selectfont,black,anchor=east,inner sep=.3pt] at (2.83,3) {\(3^{[2]}\)};
\fill[blue!75!black] (3.5,3.5) circle[radius=1.4pt];
\node[font=\fontsize{7}{8}\selectfont,blue!75!black,anchor=south,inner sep=.3pt] at (3.5,3.63) {\(4^{[3]}\)};
\end{scope}
\begin{scope}[shift={(9cm,-15.59cm)},x=0.7cm,y=1cm]
\path[draw=black,fill=white,line width=.5pt] (0,-.21) rectangle (1,.21);
\node[font=\scriptsize,gray!80,anchor=north] at (0.5,-.23) {0};
\path[draw=black,fill=white,line width=.5pt] (1,-.21) rectangle (2,.21);
\node[font=\scriptsize,gray!80,anchor=north] at (1.5,-.23) {1};
\path[draw=black,fill=white,line width=.5pt] (2,-.21) rectangle (3,.21);
\node[font=\scriptsize,gray!80,anchor=north] at (2.5,-.23) {2};
\path[draw=black,fill=white,line width=.5pt] (3,-.21) rectangle (4,.21);
\node[font=\scriptsize,gray!80,anchor=north] at (3.5,-.23) {3};
\path[draw=black,fill=white,line width=.5pt] (4,-.21) rectangle (5,.21);
\node[font=\scriptsize,gray!80,anchor=north] at (4.5,-.23) {4};
\path[draw=black,fill=white,line width=.5pt] (5,-.21) rectangle (6,.21);
\node[font=\scriptsize,gray!80,anchor=north] at (5.5,-.23) {5};
\path[draw=blue!75!black,fill=blue!12,line width=.8pt] (6,-.21) rectangle (7,.21);
\node[font=\scriptsize,blue!75!black] at (6.5,0) {7};
\node[font=\scriptsize,gray!80,anchor=north] at (6.5,-.23) {6};
\draw[blue!75!black,line width=.45pt] (0,-.61)--(0,-.67)--(7,-.67)--(7,-.61);
\node[font=\scriptsize,anchor=north] at (3.5,-.64) {\(J_7(a)=\{0,\ldots,6\},\quad e_7(a)=0\)};
\end{scope}
\node[font=\small,align=center] at (.85,-15.59) {\(r=7\)\\[2pt]\(4^{[3]}\)};
\node[font=\small] at (7.85,-15.59) {\(\longleftrightarrow\)};
\end{tikzpicture}
\caption{The same insertion states as in
  \Cref{fig:nt-omega-insertion}, paired with marked \(\Omega\)-parking.
  Blue indicates the new node and new marked block; gray roots and
  gray marked blocks correspond. Each position of blocks is the canonical label
  of the corresponding internal node. The blue bracket indicates \(J_r(a)\),
  whose minimum is \(e_r(a)\).}
\label{fig:3}
\end{figure}
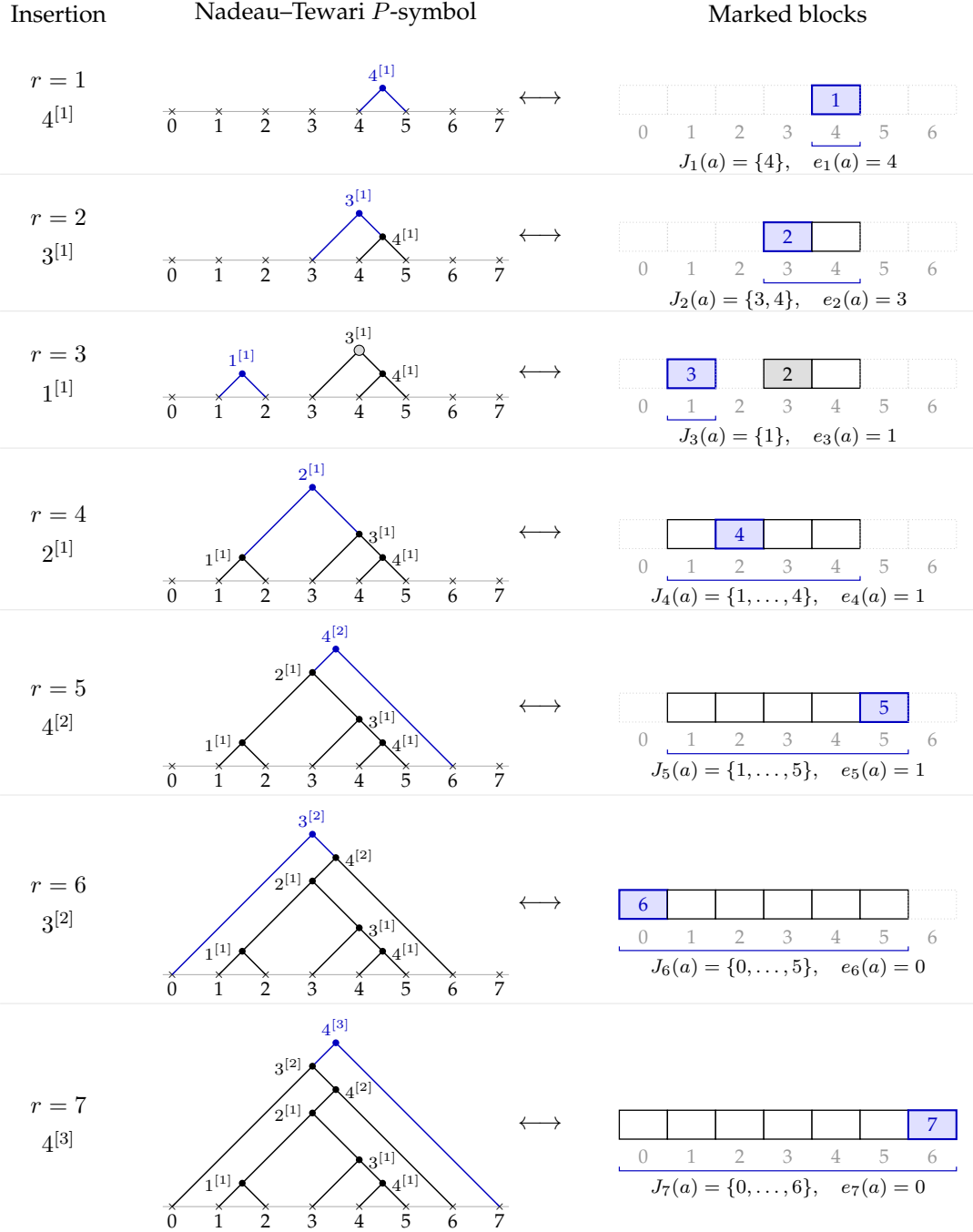

Let \(a=(a_1,\ldots,a_m)\in\EWord_m\) be nonempty. Define its shift by
\begin{equation}\label{eq:component-word-shift}
  s(a)=\max_{1\leq r\leq m}(|J_r(a)|-1),
\end{equation}
and put
\begin{equation}\label{eq:word-horizontal-position}
  h_r(a)=e_r(a)+s(a)
  \qquad(1\leq r\leq m).
\end{equation}
Since \(e_1(a)\geq\cdots\geq e_m(a)\), the sequence
\(h_1(a),\ldots,h_m(a)\) is weakly decreasing. Since
\(a_r\in J_r(a)\), we also have
\( 1\leq a_r\leq\max J_r(a) =e_r(a)+|J_r(a)|-1\leq h_r(a) \).
In particular, every \(h_r(a)\) is positive. For the empty word,
set \(s(\varnothing)=0\).

Define \(c(a)=(c_i(a))_{i\geq1}\in\Codes\) by
\( c_i(a)=|\{r:h_r(a)=i\}| \),
and let \(P_a=P(c(a))\in\Path\). The \emph{word labeling} \(\kappa_a\)
of \(P_a\) is given by
\[
  \kappa_a(S_r)=\kappa_{P_a}(S_r)=a_r
  \qquad(1\leq r\leq m),
\]
where \(S_1,\ldots,S_m\in\mathcal S(P_a)\). Thus \(S_r\) lies on
\(x=2h_r(a)\). We call the labeled path \((P_a,\kappa_a)\) the
\emph{word path} of \(a\).

\begin{lem}\label{lem:block-path-successors}
Let \(a=(a_1,\ldots,a_m)\in\EWord_m\). For \( 1\le r<s\le m \),
\begin{equation}\label{eq:block-path-successor}
  \begin{aligned}  
  S_r\rav S_s
  \quad&\Longleftrightarrow\quad
  r\trl_a s \ \text{and}\ a_r\leq a_s,\\
  S_r\rad S_s
  \quad&\Longleftrightarrow\quad
  r\trl_a s \ \text{and}\ a_r>a_s.
  \end{aligned}
\end{equation}
\end{lem}
\begin{proof}
By \eqref{eq:omega-insertion-data} and
\Cref{lem:forest-path-incidence}, vertical and diagonal successors
correspond to left and right children in \(\mathsf P_\Omega(a)\),
respectively. The label values satisfy \(a_r\leq a_s\) for a left
child and \(a_r>a_s\) for a right child.
\end{proof}

For an endpoint-decreasing word \(a=(a_1,\ldots,a_m)\), define
\( \CSel_a^\times :=\CSel^\times(P_a,\kappa_a) \).
Then
\begin{equation}\label{eq:word-path-nonzero-selections}
  \CSel_a^\times
  =\left\{Q:
    \begin{array}{l}
      Q\text{ chooses one cell in each row and satisfies }
        \eqref{arr:csel_cond},\\
      \col_r(Q)<2a_r\quad(1\leq r\leq m)
    \end{array}
  \right\}.
\end{equation}
Indeed, \(\col_r(Q)<2a_r\leq2h_r(a)\) places every selected cell
above \(P_a\), and \(\col_r(Q)<2a_r\) makes its weight nonzero.
Conversely, every cell whose column index is at least \(2a_r\) has
zero weight under the word labeling. This proves
\eqref{eq:word-path-nonzero-selections}.

\begin{remark}\label{rem:word-path-shift-choice}
The shift \(s(a)\) in \eqref{eq:component-word-shift} need not be
minimal. For example, \(a=(4,3,1)\) gives \(s(a)=1\), although
a shift of zero already gives \(h_r(a)\geq a_r\) for every \(r\).
For any integer \(d\geq0\), we may also replace every \(h_r(a)\)
by \(h_r(a)+d\), keeping the word labels \(a_r\) fixed.
This uniform translation preserves the successor relations, and all
nonzero selections already satisfy
\(\col_r(Q)<2a_r\leq2h_r(a)\) by
\eqref{eq:word-path-nonzero-selections}. Thus the resulting path has
the same nonzero cell selections and weights, and defines the same
lattice path polynomial.
\end{remark}

\subsection{Endpoint sorting and linear extensions of posets}
\label{subsec:endpoint-sorting}

For words in positive integers, a \emph{commutation move} is the interchange
\[
  (\ldots,i,j,\ldots)\longleftrightarrow(\ldots,j,i,\ldots)
  \qquad\text{when } |i-j|>1.
\]
Two words are \emph{commutation-equivalent} if they are related by a sequence of commutation moves.

Let \(a=(a_1,\ldots,a_m)\in\Word_m\) and let \(\sigma\in S_m\) be the unique permutation such that
\begin{equation}\label{eq:uniq_perm}
  e_{\sigma(1)}(a)\geq\cdots\geq e_{\sigma(m)}(a),
  \qquad
  e_{\sigma(r)}(a)=e_{\sigma(s)}(a)
  \ \Longrightarrow\ 
  \sigma(r)<\sigma(s)
  \quad(1\leq r<s\leq m).
\end{equation}
Define the map 
\[
  \ESort:\bigsqcup_{m\ge0}\Word_m\rightarrow\bigsqcup_{m\ge0}\EWord_m
\]
by \( \ESort(a)=(a_{\sigma(1)},\ldots,a_{\sigma(m)}) \) and \(\ESort(\varnothing)=\varnothing\).
We call \( \ESort(a) \) the \emph{endpoint sort} of \(a\).
If \(r<s\) and \(a_r=a_s\), then
\(J_r(a)\subseteq J_s(a)\), and hence \(e_r(a)\geq e_s(a)\).
Therefore \(\ESort(a)\) preserves the relative order of equal letters.

Let \(r<s\) be incomparable in \(\mathcal H(a)\). The blocks
\(I_{a_r}\) and \(I_{a_s}\) lie in distinct components after the
\(s\)-th insertion; otherwise the successive removals give \(r\prec_a s\). Therefore
\begin{equation}\label{eq:incomparable-supports-commute}
  |a_r-a_s|>1.
\end{equation}
Moreover, every relation \(r\trl_a s\) has \(r<s\), so the
natural order \((1,\ldots,m)\) is a linear extension of
\(\mathcal H(a)\).

\begin{lem}\label{lem:endpoint-sort-well-defined}
For every \( a\in\Word_m \), the word \( \ESort(a) \) is endpoint-decreasing and commutation-equivalent to \( a \).
In particular, the map \( \ESort \) is well-defined.
\end{lem}
\begin{proof}
The correspondence in \Cref{subsec:endpoint-decreasing} identifies
\(\ESort(a)\) with the word obtained from the rightmost reading of
\(\mathsf P_\Omega(a)\) by removing the decorations.
The claim then follows from
\cite[Proposition~4.8 and Lemma~4.9]{NT2024}.
\end{proof}

Since this rightmost reading is determined by the \(P\)-symbol,
for words \(u,v\) in positive integers,
\begin{equation}\label{eq:endpoint-sort-omega-classes}
  \ESort(u)=\ESort(v)
  \quad\Longleftrightarrow\quad
  \mathsf P_\Omega(u)=\mathsf P_\Omega(v).
\end{equation}
Combining \eqref{eq:omega-insertion-data} with
\cite[Theorem~4.1 and Equation~(4.1)]{NT2024} yields the following
characterization of words with the same endpoint sort.

\begin{prop}\label{prop:endpoint-sort-linear-extensions}
Let \(a=(a_1,\ldots,a_m)\) be endpoint-decreasing, and let
\(u\) be a rearrangement of \(a\). Then \(\ESort(u)=a\) if and only if
\( u=(a_{\theta(1)},\ldots,a_{\theta(m)}) \)
for some linear extension
\((\theta(1),\ldots,\theta(m))\) of \(\mathcal H(a)\).
\end{prop}

Let \(a\) be endpoint-decreasing, let
\((\sigma(1),\ldots,\sigma(m))\) be a linear extension of
\(\mathcal H(a)\), and put
\(b=(a_{\sigma(1)},\ldots,a_{\sigma(m)})\).
Reading the same \(P\)-symbol and using
\eqref{eq:omega-insertion-data} gives
\begin{equation}\label{eq:linear-extension-induction-data}
    \ESort(b)=a,\quad J_p(b)=J_{\sigma(p)}(a), \quad
    p\trl_b q
    \Longleftrightarrow
    \sigma(p)\trl_a\sigma(q)
    \qquad(p,q\in[m]).
\end{equation}
In particular, \(e_p(b)=e_{\sigma(p)}(a)\).

\subsection{A lattice path model for double Schubert polynomials}
\label{subsec:double-schubert-path}
\leavevmode\par\noindent

For \(w\in S_\infty\), define
\[
  \ERed(w)
  :=\{a\in\Red(w):a\text{ is endpoint-decreasing}\}.
\]
Its elements are called \emph{endpoint-decreasing reduced words} of
\(w\). For every nonempty \(a\in\ERed(w)\), we write
\(\FL_{P_a,\kappa_a}\xt\) simply as \(\FL_a\xt\).
We also set \(\FL_\varnothing\xt=1\).
\begin{thm}\label{thm:double-schubert-lattice-path}
For every \(w\in S_\infty\),
\[
  \mathfrak S_w(\vec{x};\vec{t})
  =\sum_{a\in\ERed(w^{-1})}\FL_a(\vec{x};\vec{t}).
\]
\end{thm}

Specializing to \(\vec t=\vec0\) and applying
\Cref{prop:compressed-specialization}, we identify each word-path summand
with its compressed lattice path polynomial. Thus the theorem also gives
a lattice path model for ordinary Schubert polynomials.

We apply \Cref{lem:subword-path-bijection} to the word paths.
If
\(b\in\Red(w^{-1})\), then \Cref{lem:endpoint-sort-well-defined} shows
that \(\ESort(b)\) is in \( \ERed(w^{-1}) \) and is commutation-equivalent to \(b\).

Fix \(w\in S_{n+1}\). Recall that
\(\rev(b)=(b_m,\ldots,b_1)\) for a word
\(b=(b_1,\ldots,b_m)\). For
\(a\in\ERed(w^{-1})\), define
\[
  \mathcal R_a(\widetilde\omega_{[n]};w)
  =\{\pi\in\mathcal R(\widetilde\omega_{[n]};w):
      \ESort\bigl(\operatorname{rev}(\val(\pi))\bigr)=a\}.
\]
When \( n \) and \( w \) are clear from context, we simply write \( \mathcal R_a \).
These sets form a disjoint union
\begin{equation}\label{eq:vine-subword-decomposition}
  \mathcal R(\widetilde\omega_{[n]};w)
  =\bigsqcup_{a\in\ERed(w^{-1})}\mathcal R_a.
\end{equation}

Write \(\Phi_a:=\Phi_{P_a,\kappa_a}\) for this specialization.
For \(\pi\in\mathcal R_a\), put \(u=\rev(\val(\pi))\) and
let \(\sigma_u\) define \(\ESort(u)\).
The linear extension in \Cref{subsec:common-subword-path} is
\((\sigma_u^{-1}(1),\ldots,\sigma_u^{-1}(m))\), so row \(r\)
receives the letter \(B_r=\pi_{m-\sigma_u(r)+1}\).
The cell--letter assignment is the same as in
\Cref{fig:common-weight-preserving-bijection}.

\begin{proof}[Proof of \Cref{thm:double-schubert-lattice-path}]
Choose \(n\) such that \(w\in S_{n+1}\), and let
\(a\in\ERed(w^{-1})\) be nonempty.
By \Cref{lem:block-path-successors}, the successor poset
\(\mathcal H(P_a)\) is \(\mathcal H(a)\), and
\(a_r\leq a_s\) along vertical successors, while \(a_r>a_s\) along
diagonal successors. A vertical successor joins consecutive rows, so equality would contradict reducedness. The word path construction gives
\(a_r\leq h_r(a)\) in every row, and
\eqref{eq:incomparable-supports-commute} makes rows labeled by the same number
comparable. Thus the labeling \(\kappa_a\) satisfies all the
hypotheses of \Cref{lem:subword-path-bijection}.

By \Cref{prop:endpoint-sort-linear-extensions}, the reversed value
words in \(\mathcal R(P_a,\kappa_a)\) are exactly the rearrangements
\(u\) with \(\ESort(u)=a\). These are commutation-equivalent to
\(a\), hence reduced words of \(w^{-1}\). Therefore, we obtain
\( \mathcal R(P_a,\kappa_a)=\mathcal R_a \).
Now \Cref{lem:subword-path-bijection}, \eqref{eq:double-schubert-subword}, and
\eqref{eq:vine-subword-decomposition} give
\[
  \mathfrak S_w\xt
  =\sum_{a\in\ERed(w^{-1})}
      \sum_{\pi\in\mathcal R_a}\wt(\pi)
  =\sum_{a\in\ERed(w^{-1})}\FL_a\xt.
\]
The empty word contributes \(1\) by convention.
\end{proof}

For a partition
\(\lambda=(\lambda_1\geq\cdots\geq\lambda_n\geq0)\), combining
\Cref{thm:double-schubert-lattice-path} with
\eqref{eq:grassmannian-double-schur} gives
\begin{equation}\label{eq:double-schur-lattice-path}
  s_\lambda(x_1,\ldots,x_n;\vec t)
  =\sum_{a\in\ERed(w_\lambda^{-1})}\FL_a(\vec x;\vec t).
\end{equation}

We next describe the word paths for Grassmannian permutations.  For a reduced word
\(a=(a_1,\ldots,a_m)\), its \emph{heap} is the poset on \([m]\)
whose order is the transitive closure of the relation \( r\prec s \) if \( r<s\) and \( |a_r-a_s|\leq1 \).
Thus a linear extension
\((r_1,\ldots,r_m)\) has label sequence
\((a_{r_1},\ldots,a_{r_m})\).  A permutation is
\emph{fully commutative} if any two of its reduced words are commutation-equivalent.  In this case their
heaps are isomorphic as labeled posets, and we call this labeled poset the heap of the permutation.

From \cite{Winkel1996}, for a partition
\(\lambda=(\lambda_1\geq\cdots\geq\lambda_n\geq0)\), the heap of \(w_\lambda^{-1}\) is isomorphic to \(D(\lambda)\), with the box
\((i,j)\) labeled by \(n-i+j\) and ordered coordinatewise:
\[
  (i,j)\leq(i',j')
  \quad\Longleftrightarrow\quad
  i\leq i'\ \text{and}\ j\leq j'.
\]
By \cite{Stembridge1996}, its reduced words are the
label sequences of the linear extensions of this heap, since \(w_\lambda^{-1}\) is fully commutative.

For a finite word \(b=(b_1,\ldots,b_\ell)\), let
\(\asc(b)\) and \(\des(b)\) denote its numbers
of ascents and descents, that is,
\( \asc(b)=|\{1\leq k<\ell:b_k<b_{k+1}\}| \) and \(
\des(b)=|\{1\leq k<\ell:b_k>b_{k+1}\}| \).
Both statistics are zero on a word of length at most one.

\begin{prop}\label{prop:grassmannian-block-construction}
Let \(\lambda\) be a partition of length at most \( n \) with
\(|\lambda|=m>0\), and let
\(a=(a_1,\ldots,a_m)\in\Red(w_\lambda^{-1})\). For every
\(1\leq r\leq m\), \(\mathcal I_r\) consists of exactly one
component, and
\( e_r(a)=n-\des(a_1\cdots a_r) \).
Consequently, every reduced word of \( w_\lambda^{-1} \) is endpoint-decreasing, so
\( \ERed(w_\lambda^{-1})=\Red(w_\lambda^{-1}) \).
\end{prop}

\begin{proof}
For each \(r\), the first \(r\) boxes of the corresponding linear
extension form a Young diagram. Hence \(a_1\cdots a_r\) is a reduced
word for the inverse of a Grassmannian permutation with descent at
\(n\). The argument in \cite[Section~6.2(a)]{NT2024} shows that its
\(P\)-symbol consists of a single tree whose least canonical label
is \(n-\des(a_1\cdots a_r)\).
By \eqref{eq:omega-insertion-data}, the occupied blocks therefore
form a single component with this left endpoint.
The endpoints are weakly decreasing, which proves the final assertion.
\end{proof}

For \( a\in\Red(w_\lambda^{-1}) \), by \Cref{lem:block-path-successors}, the successor forest of \(P_a\)
is therefore the chain
\(S_1\rightarrow S_2\rightarrow\cdots\rightarrow S_m\),
with \(S_{r-1}\rav S_r\) if \(a_{r-1}\leq a_r\) and
\(S_{r-1}\rad S_r\) if \(a_{r-1}>a_r\).
For any nonzero cell selection \(Q\),
\eqref{eq:word-path-nonzero-selections} and \eqref{arr:csel_cond} give
\(\col_r(Q)\leq\col_1(Q)<2a_1=2n\) in every row.
Thus \(\FL_a\in\mathbb Z[\vec t][x_1,\ldots,x_n]\), independently
of the horizontal shift.

\begin{prop}\label{prop:grassmannian-padded-word-path}
Let \(\lambda\) be a partition of length at most \( n \) with
\(|\lambda|=m>0\), and let \(a\in\Red(w_\lambda^{-1})\).  The code of
\(P_a\) is a padded composition of length \(n+m-1\).  Moreover, if the
\(r\)-th insertion creates the block \(I_i\), then
\( \kappacan_{P_a}(S_r)=i+m-1 \).
\end{prop}
\begin{proof}
By \Cref{prop:grassmannian-block-construction},
\(|J_r(a)|=r\) for every \(r\in[m]\), so \(s(a)=m-1\).
Thus \eqref{eq:word-horizontal-position} gives
\[
  h_r(a)=n+m-1-\des(a_1\cdots a_r).
\]
In particular, \(h_1(a)=n+m-1\), \(h_m(a)\geq n\), and successive
values differ by \(0\) or \(-1\). Hence the support of \(c(P_a)\)
is the integer interval \([h_m(a),n+m-1]\), proving that \(c(P_a)\)
is padded of length \(n+m-1\).

If the \(r\)-th insertion creates \(I_i\), then \(v_r\) has canonical
label \(i\) by \Cref{tab:omega-dictionary}.
Translating the leaf indices by \(m-1\) and applying
\Cref{lem:forest-path-incidence} gives
\(\kappacan_{P_a}(S_r)=i+m-1\).
\end{proof}

\subsection{A lattice path model for Grothendieck polynomials}

We now convert the Hecke subword formula
\eqref{eq:grothendieck-hecke-subword} into a lattice path formula.
For \(w\in S_{n+1}\), let \(\EH_n(w)\) be the set of
endpoint-decreasing words \(a=(a_1,\ldots,a_m)\), where
\(0\leq m\leq \binom{n+1}{2}\) and \(a_r\in[n]\), such that
\(\operatorname{Dem}(a)=w\).

For a nonempty \(a\in\EH_n(w)\), let
\((P_a,\kappa_a)\) be its word path and form \(P_a^\circ\) by replacing each west step by \( (-1,0) \). Define its injective cell selections by
\begin{equation}\label{eq:injective-hecke-cell-selection}
  \ICSel_a
  =\left\{Q\in\CSel(P_a^\circ):
      (a_r,\col_r(Q))\neq(a_s,\col_s(Q))\quad(r\neq s)
  \right\}.
\end{equation}
We define
\[
  \FL_a^K(\vec x)
  :=\sum_{Q\in\ICSel_a}\wt_{\kappa_a}(Q).
\]
For the empty word, from
\( \ICSel_\varnothing=\{\varnothing\} \), we have
\( \FL_\varnothing^K(\vec x)=1 \).
\begin{thm}\label{thm:grothendieck-lattice-path}
For every \(w\in S_{n+1}\),
\begin{equation}\label{eq:grothendieck-lattice-path}
  G_w^{(\beta)}(\vec x)
  =\sum_{a\in\EH_n(w^{-1})}
     \beta^{|a|-\ell(w)}\FL_a^K(\vec x).
\end{equation}
\end{thm}

We prove \Cref{thm:grothendieck-lattice-path} by constructing a
weight-preserving bijection.

Let \( \ICSel_a^\times\) be the set of cell selections with nonzero weight in \( \ICSel_a \).
By \eqref{eq:compressed-cell-weight}, we have
\begin{equation}\label{eq:nonzero-injective-cell-selection}
  \ICSel_a^\times =\{Q\in\ICSel_a:\col_r(Q)\leq a_r\text{ for every }r\in[m]\}.
\end{equation}
Recall that the subword formula \eqref{eq:grothendieck-hecke-subword}
sums over \(\mathcal R^K(\omega_n^+;w)\), the subwords whose value
words have Demazure product \(w\).
For \(a\in\EH_n(w^{-1})\), let
\[
  \mathcal R_a^K(\omega_n^+;w)
  =\left\{\pi\in\mathcal R^K(\omega_n^+;w):
      \ESort\bigl(\rev(\val(\pi))\bigr)=a
    \right\}.
\]
When \(n\) and \(w\) are clear, we simply write \(\mathcal R_a^K\).
Reversal satisfies
\( \operatorname{Dem}(\rev(b))=\operatorname{Dem}(b)^{-1} \).
Together with \Cref{lem:endpoint-sort-well-defined}, this gives the disjoint union
\begin{equation}\label{eq:hecke-subword-decomposition}
  \mathcal R^K(\omega_n^+;w)
  =\bigsqcup_{a\in\EH_n(w^{-1})}\mathcal R_a^K.
\end{equation}

Given \(\pi=(\pi_1,\ldots,\pi_m)\in\mathcal R_a^K\), put
\(u=\rev(\val(\pi))\), and let \(\sigma_u\) be the permutation
defining \(\ESort(u)\). For each \(r\in[m]\), write
\( \pi_{m-\sigma_u(r)+1}=a_r^{(q_r)} \),
and define \( \Phi_a^K:\mathcal R_a^K\rightarrow\ICSel_a^\times \) by
\[
  \Phi_a^K(\pi)=\{C_{q_r,r}:1\leq r\leq m\}.
\]

\begin{lem}\label{lem:hecke-path-bijection}
For every \(a\in\EH_n(w^{-1})\), the map
\( \Phi_a^K\) is a weight-preserving bijection.\end{lem}
\begin{proof}
The assertion is immediate for \(a=\varnothing\), so assume
\(a=(a_1,\ldots,a_m)\) is nonempty.
For \(r\trl_a s\) and distinct letters
\(B_r=a_r^{(q_r)}\), \(B_s=a_s^{(q_s)}\) of \(\omega_n^+\),
\Cref{lem:block-path-successors} and the vine order \(\prec_+\) give
\begin{equation}\label{eq:staircase-letter-order-successor}
  B_s\prec_+B_r
  \quad\Longleftrightarrow\quad
  \begin{cases}
    q_r\geq q_s,&S_r\rav S_s,\\
    q_r>q_s,&S_r\rad S_s.
  \end{cases}
\end{equation}
We use the letter reconstruction from the proof of
\Cref{lem:subword-path-bijection}. Here equal labels may occur along
vertical arrows, so the injectivity condition is needed.
For \(\pi\in\mathcal R_a^K\), the letters
\(B_r=\pi_{m-\sigma_u(r)+1}=a_r^{(q_r)}\), where
\(u=\rev(\val(\pi))\), are distinct and satisfy \(q_r\leq a_r\).
Their reversed order is a linear extension of \(\mathcal H(a)\)
by \Cref{prop:endpoint-sort-linear-extensions}. Thus
\eqref{eq:staircase-letter-order-successor} gives the compressed
cell-selection inequalities, and the pairs \((a_r,q_r)\) are distinct.
Hence \(\Phi_a^K(\pi)\in\ICSel_a^\times\).

Conversely, for \(Q\in\ICSel_a^\times\), set
\(q_r=\col_r(Q)\) and \(B_r=a_r^{(q_r)}\).
Injectivity and \eqref{eq:nonzero-injective-cell-selection} make the
\(B_r\) distinct letters of \(\omega_n^+\). Read them in vine order to obtain \(\pi\). By
\eqref{eq:staircase-letter-order-successor}, their reversed order is
a linear extension of \(\mathcal H(a)\), so
\(u=\rev(\val(\pi))\) satisfies \(\ESort(u)=a\).
Commutation preserves the Demazure product, giving
\(\operatorname{Dem}(u)=\operatorname{Dem}(a)=w^{-1}\);
the reversal identity then gives
\(\operatorname{Dem}(\val(\pi))=w\). Thus \(\pi\in\mathcal R_a^K\).

If \(a_r=a_s\) with \(r<s\), then \(r\prec_a s\) by
\eqref{eq:incomparable-supports-commute}, so \(B_s\prec_+B_r\).
The reversed reading therefore preserves the order of equal-valued
indices, and the inverse verification in \Cref{lem:subword-path-bijection}
applies. Finally,
\(\wt(\pi)=\prod_{r=1}^m x_{q_r}
=\wt_{\kappa_a}(\Phi_a^K(\pi))\), proving weight preservation.
\end{proof}

\begin{proof}[Proof of \Cref{thm:grothendieck-lattice-path}]
By \eqref{eq:grothendieck-hecke-subword},
\eqref{eq:hecke-subword-decomposition}, and
\Cref{lem:hecke-path-bijection},
\[
  \begin{aligned}
  G_w^{(\beta)}(\vec x)
  &=\sum_{\pi\in\mathcal R^K(\omega_n^+;w)}
      \beta^{|\pi|-\ell(w)}\wt(\pi) \\
  &=\sum_{a\in\EH_n(w^{-1})}
      \beta^{|a|-\ell(w)}
      \sum_{\pi\in\mathcal R_a^K}\wt(\pi)
  =\sum_{a\in\EH_n(w^{-1})}
    \beta^{|a|-\ell(w)}\FL_a^K(\vec x).
  \end{aligned}
\]
\end{proof}
\flushbottom
\section{Combinatorial interpretations of expansion coefficients}
\label{sec:change-of-basis}

In this section, we derive combinatorial formulas for expansion
coefficients from the lattice path models. We first express the
coefficients in the double fundamental expansion of double Schur
polynomials as weighted sums over admissible deletion sets of reduced
words, and discuss the additional cancellations required for general
double Schubert polynomials. We then express Grothendieck polynomials
in the grove basis using endpoint-decreasing Hecke words.
Setting \(\beta=0\) yields a counting formula for the forest expansion
of Schubert polynomials. We write \([B]f\) for the coefficient of a
basis element \(B\) in the corresponding basis expansion of \(f\).

\subsection{From double Schur to double fundamental polynomials}
\label{sec:double-schur-successor-replay}

Throughout this subsection, a \emph{positive root} means \(t_q-t_p\) with
\(p<q\) and we fix a partition \(\lambda=(\lambda_1\geq\cdots\geq\lambda_n\geq0)\).
By \Cref{prop:grassmannian-block-construction}, all reduced words of
\(w_\lambda^{-1}\) are endpoint-decreasing. Thus the sums below
range over all of \(\Red(w_\lambda^{-1})\).

Let \(b=(b_1,\ldots,b_\ell)\) be a finite word. A position
\(2\leq k\leq\ell\) is a \emph{double ascent} in \( b \) if
\(b_{k-1}<b_k<b_{k+1}\), and a \emph{double descent} in \( b \) if
\(b_{k-1}>b_k>b_{k+1}\). Here we set \(b_{\ell+1}=+\infty\)
only for these comparisons; this boundary value is not included when
counting ascents or descents. For \(\ell\geq2\), the last position is
a double ascent exactly when \(b_{\ell-1}<b_\ell\), and is never a
double descent.
An \emph{increasing run} of \(b\) is a nonempty consecutive subword
that is strictly increasing and maximal with respect to inclusion.
The unique decomposition of \(b\) into its increasing runs is called
its \emph{run decomposition}; see \cite[Definition~2.2]{Assaf2021}
for this terminology for reduced words. A nonempty word \(b\) with
distinct adjacent entries has \(\des(b)+1\) increasing runs.
A position \(2\leq k\leq\ell\) is called \emph{deletable} in \(b\)
if either
\begin{equation}\label{eq:successor-replay-delete}
\begin{aligned}
 &k\text{ is a double ascent}
 &&\text{and}\quad b_k<b_1+\asc(b_1\cdots b_k),\ \text{or}\\
 &k\text{ is a double descent}
 &&\text{and}\quad 1\le b_1-\des(b_1\cdots b_k)<b_k.
\end{aligned}
\end{equation}
For a deletable position \(k\), define its deletion weight by
\begin{equation}\label{eq:successor-replay-root}
  \Delta_b(k)=
  \begin{cases}
    t_{b_1+\asc(b_1\cdots b_k)}-t_{b_k},
      &k\text{ is a double ascent},\\
    t_{b_k}-t_{b_1-\des(b_1\cdots b_k)},
      &k\text{ is a double descent}.
  \end{cases}
\end{equation}
Thus \(\Delta_b(k)\) is a positive root.

Let \(a=(a_1,\ldots,a_m)\in\Red(w_\lambda^{-1})\), where
\(m=|\lambda|>0\). For \(U=\{r_1<\cdots<r_\ell\}\subseteq[m]\),
write \(a_U=(a_{r_1},\ldots,a_{r_\ell})\), with
\(a_\varnothing=\varnothing\).
We construct deletion sets by considering each letter of \(a\)
once, in increasing order of its original index. Whenever its
position is deletable in the remaining word, we choose either to
keep the letter or to delete it; all other letters are kept.
The set \(D\) of original indices of the deleted letters is called
a \emph{deletion set} of \(a\).
We call \(D\) an \emph{admissible deletion set} if the final word
\(a_{[m]\setminus D}\) has at most \(n\) increasing runs.
Let \(\mathcal A_n(a)\) denote the set of admissible deletion sets.
Since \(a\) has distinct adjacent letters and permitted deletions
preserve this property, the final run condition is equivalent to
\(\des(a_{[m]\setminus D})<n\).

Each keep choice has weight \(1\), and each delete choice has the
weight in \eqref{eq:successor-replay-root}, evaluated in the word
before deletion. For \(D\in\mathcal A_n(a)\), let \(\wt_a(D)\)
be the product of the weights of the corresponding choices.
If the increasing runs of \(a_{[m]\setminus D}\), read from left to right,
have lengths \((\ell_1,\ldots,\ell_q)\), set
\begin{equation}\label{eq:replay-code}
  c_D=(0^{\,n-q},\ell_q,\ldots,\ell_1)\in\operatorname{Pad}_n.
\end{equation}
For the empty word, set \(\mathcal A_n(\varnothing)=\{\varnothing\}\),
\(c_\varnothing=(0^n)\), and \(\wt_\varnothing(\varnothing)=1\).

\begin{thm}\label{thm:double-schur-successor-replay}
For every \(a\in\Red(w_\lambda^{-1})\),
\begin{equation}\label{eq:word-successor-replay-expansion}
  \FL_a(\vec x;\vec t)
  =\sum_{D\in\mathcal A_n(a)}
      \wt_a(D)\,\mathfrak F_{c_D}(x_1,\ldots,x_n;\vec t).
\end{equation}
\end{thm}
Summing over \(a\) and applying \eqref{eq:double-schur-lattice-path}
gives the coefficient formula.
\begin{cor}\label{cor:graham-pos}
For every \(c\in\operatorname{Pad}_n\),
\begin{equation}\label{eq:double-schur-successor-replay-coefficient}
  [\mathfrak F_c]s_\lambda
  =\sum_{a\in\Red(w_\lambda^{-1})}
     \ \sum_{\substack{D\in\mathcal A_n(a)\\c_D=c}}
       \wt_a(D).
\end{equation}
In particular, double Schur polynomials expand Graham-positively in
the double fundamental basis.
\end{cor}

The following example illustrates both types of deletion, their dependence
on earlier choices, and the final admissibility condition.

\begin{exam}\label{exa:successor-replay-boundary}\label{exa:dynamic-deletion}
Take \(n=3\), \(\lambda=(3,2,2)\), and
\[
  a=(3,2,4,3,1,2,5)\in\Red(w_\lambda^{-1}).
\]
Its run decomposition is \(3\mid24\mid3\mid125\), with four runs.
Initially, the deletable positions are \(4\), \(6\), and \(7\).
Position \(4\) is a double descent of weight \(t_3-t_1\), while
positions \(6\) and \(7\) are double ascents of weights
\(t_5-t_2\) and \(t_6-t_5\), respectively.
If position \(6\) is deleted, the current word is
\((3,2,4,3,1,5)\). Its last letter is still a double ascent, but
its deletion inequality becomes
\[
  5<3+\asc(3,2,4,3,1,5)=5,
\]
which fails. The same failure occurs if position \(4\) was also
deleted. Thus positions \(6\) and \(7\) cannot both be deleted.

The deletion sets \(\varnothing\), \(\{6\}\), and \(\{7\}\)
all leave four runs and are not admissible. The three admissible
sets and their contributions are
\[
\begin{array}{c|c|c|c}
 D&\text{remaining word}&c_D&\wt_a(D)\\ \hline
 \{4\}&(3,2,4,1,2,5)&(3,2,1)&t_3-t_1\\
 \{4,6\}&(3,2,4,1,5)&(2,2,1)&(t_3-t_1)(t_5-t_2)\\
 \{4,7\}&(3,2,4,1,2)&(2,2,1)&(t_3-t_1)(t_6-t_5).
\end{array}
\]
Consequently,
\[
  \FL_a
  =(t_3-t_1)\mathfrak F_{(3,2,1)}
   +(t_3-t_1)(t_6-t_2)\mathfrak F_{(2,2,1)}.
\]
To see how different reduced words contribute to one coefficient,
keep the same \(\lambda\) and consider \(c=(1,2,2)\).
The pairs \((a',D)\) with \(c_D=c\) are exactly
\[
\begin{array}{c|c|c}
 a'&D&\wt_{a'}(D)\\ \hline
 (3,4,2,5,3,1,2)&\{5,7\}&(t_3-t_1)(t_6-t_2)\\
 (3,4,2,3,1,2,5)&\{6,7\}&(t_6-t_2)(t_6-t_5)\\
 (3,4,2,3,5,1,2)&\{4,7\}&(t_5-t_3)(t_6-t_2).
\end{array}
\]
Each remaining word has increasing run lengths \((2,2,1)\).
Summing these weights in \Cref{cor:graham-pos} gives
\[
  [\mathfrak F_{(1,2,2)}]
  s_{(3,2,2)}(x_1,x_2,x_3;\vec t)
  =(t_6-t_2)(t_6-t_1).
\]
\end{exam}

The coefficient formula also computes equivariant degrees. Let
\(N>n\), let \(T=(\mathbb C^\times)^N\) be the diagonal torus, and
write \(\vec t_N=(t_1,\ldots,t_N)\). Evaluation at \(\vec t_N\)
means setting \(t_i=0\) for \(i>N\).

A \emph{ribbon} is a nonempty skew Ferrers diagram
\(D(\eta)\setminus D(\nu)\), where \(\eta,\nu\) are partitions
with \(D(\nu)\subseteq D(\eta)\), that is connected under edge
adjacency and contains no \(2\times2\) block of boxes.
For a nonzero padded composition
\(c=(0^{n-q},\gamma_1,\ldots,\gamma_q)\), its positive parts
\(\gamma=(\gamma_1,\ldots,\gamma_q)\) determine, up to translation,
the ribbon with \(\gamma_i\) boxes in row \(i\) for \(1\leq i\leq q\).
Consecutive rows overlap in exactly one column, so this ribbon has
\(q\) rows and \(|c|-q+1\) columns. When this ribbon fits in the
\(n\times(N-n)\) rectangle, let \(A_c\in\binom{[N]}{n}\)
be the subset corresponding to \(\gamma\) under the bijections in
\cite[Lemmas~3.4 and~3.7]{BGST2026QG}, and define
\[
  X_c=X_{\mathrm{NC}}^{A_c}\subseteq\operatorname{QGr}(n;N).
\]
Here \(X_{\mathrm{NC}}^{A_c}\) is the closure of the quasisymmetric
Schubert cell indexed by \(A_c\). For \(c=(0^n)\), use the empty
composition, the empty diagram, and \(A_c=[n]\).
Bergeron, Gagnon, Spink, and Tewari prove the following equivariant
Kronecker duality.

\begin{thm}[{\cite[Theorem~8.8]{BGST2026QG}}]\label{thm:qgr-kronecker-duality}
For padded compositions \(c,d\in\operatorname{Pad}_n\) whose ribbons
fit in the \(n\times(N-n)\) rectangle,
\[
  \int_{X_c}^T\mathfrak F_d(x_1,\ldots,x_n;\vec t_N)
  =\delta_{c,d}.
\]
\end{thm}

Applying this duality to \Cref{cor:graham-pos} gives the following
formula.
\begin{cor}\label{cor:qgr-deletion-degree}
Put \(m=|\lambda|\), choose \(N>n\) with \(N\geq n+m\), and
let \(c\in\operatorname{Pad}_n\) satisfy \(|c|\leq m\). Then
\[
  \int_{X_c}^T s_\lambda(x_1,\ldots,x_n;\vec t_N)
  =\sum_{a\in\Red(w_\lambda^{-1})}
     \ \sum_{\substack{D\in\mathcal A_n(a)\\c_D=c}}\wt_a(D).
\]
Each summand is a product of \(m-|c|\) positive roots
\(t_j-t_i\) with \(1\leq i<j\leq N\).
\end{cor}
\begin{proof}
The bound \(N\geq n+m\) places \(\lambda\) and all output ribbons
in the rectangle, and every root index is at most \(n+m-1<N\),
so truncation leaves the weights unchanged. The formula follows from
\Cref{thm:qgr-kronecker-duality} and \Cref{cor:graham-pos}, with
\(|D|=m-|c_D|\) positive roots contributed by each deletion set.
\end{proof}

The next lemma expresses the block intervals in terms of prefix
ascent and descent counts. For \(E=\{1=i_1<\cdots<i_\ell\}\),
we index the block intervals, endpoints, and rows of \(a_E\) by \(E\):
\(J_{i_k}(a_E)\), \(e_{i_k}(a_E)\), and \(S_{i_k}\) correspond
to the \(k\)-th retained letter.

\begin{lem}\label{lem:ferrers-interval-replay}
Let \(a=(a_1,\ldots,a_m)\in\Red(w_\lambda^{-1})\) be nonempty,
and let \(D\) be a deletion set of \(a\).
Set \(E=[m]\setminus D=\{1=i_1<\cdots<i_\ell\}\).
For every \(1\leq k\leq\ell\), the first \(k\) block insertions
for \(a_E\) form a single component with interval
\begin{equation}\label{eq:replay-interval}
  J_{i_k}(a_E)=
  [a_{1}-\des(a_{i_1}\cdots a_{i_k}),
     a_{1}+\asc(a_{i_1}\cdots a_{i_k})].
\end{equation}
\end{lem}

\begin{proof}
We prove \eqref{eq:replay-interval} by induction on \(k\), together
with the assertion that the component contains \(I_{a_s}\) for every
\(s\leq i_k\). The case \(k=1\) follows from \(i_1=1\).

Suppose \eqref{eq:replay-interval} holds for \(k<\ell\) with the above assertion, and write
\(J_{i_k}(a_E)=[L,R]\). Each \(a_r\) with \(i_k<r<i_{k+1}\) is deleted from a word
beginning with \(a_{i_1}\cdots a_{i_k}a_r\).
Thus \eqref{eq:successor-replay-delete} gives
\[
  \begin{aligned}
    a_{i_k}<a_r<R+1
      &\qquad\text{at a double ascent},\\
    L-1<a_r<a_{i_k}
      &\qquad\text{at a double descent}.
  \end{aligned}
\]
Since \(a_{i_k}\in[L,R]\), we have
\(a_r\in[L,R]\) in both cases. The component therefore contains
\(I_{a_s}\) for every \(s<i_{k+1}\).

The box in the Ferrers diagram \(D(\lambda)\) of \( \lambda \) corresponding to \(a_{i_{k+1}}\) has a box immediately
above or to its left in \(D(\lambda)\). That box appears earlier
in the linear extension, and its label differs from
\(a_{i_{k+1}}\) by one. The preceding containment therefore gives
\(a_{i_{k+1}}\in[L-1,R+1]\), so its insertion preserves a single
component. The parking rule adds \(I_{R+1}\) when
\(a_{i_k}<a_{i_{k+1}}\) and \(I_{L-1}\) when
\(a_{i_k}>a_{i_{k+1}}\).
These extensions give \eqref{eq:replay-interval} for \(k+1\), and
the enlarged component contains \(I_{a_s}\) for every
\(s\leq i_{k+1}\), completing the induction.
\end{proof}

To prove \Cref{thm:double-schur-successor-replay}, we replace the word
labels by the positions of the corresponding inserted blocks, one row
at a time, to obtain the canonical labeling.
Fix a cell selection \(Q\) of a path \(P\).
Let \(\kappa\) and \(\kappa'\) be labelings of \(P\) that agree outside
row \(r\), with \(\kappa(S_r)=i\) and \(\kappa'(S_r)=j\).
If \(\col_r(Q)\leq2\min(i,j)\), then
only the factor from \(C_{\col_r(Q),r}\) changes, as illustrated in
\Cref{fig:fixed-selection-relabeling}.

\begin{figure}
  \centering
  \begin{tikzpicture}[x=0.4cm,y=0.4cm]
    \begin{scope}
      \node[font=\small] at (7.5,5.2) {Odd column, \(d=3\)};
      \begin{scope}[shift={(0,0)}]
        \fill[black!12] (4,2) rectangle (5,3);
        \fill[black!12] (0,0) rectangle (1,1);
        \fill[purple!22] (2,1) rectangle (3,2);
        \draw[gray!45,line width=.3pt,step=1] (0,0) grid (6,3);
        \draw[blue!65!black,line width=.6pt] (0,1) rectangle (6,2);
        \draw[black!75,line width=.8pt] (6,3)--(6,0);
        \node[font=\scriptsize,blue!70!black,anchor=east,inner sep=1pt] at (-.35,1.5) {\(r\)};
        \node[font=\small] at (3,3.9) {\(\kappa(S_r)=i\)};
        \node[font=\small] at (3,-.75) {\(x_2-t_i\)};
      \end{scope}
      \begin{scope}[shift={(9,0)}]
        \fill[black!12] (4,2) rectangle (5,3);
        \fill[black!12] (0,0) rectangle (1,1);
        \fill[purple!22] (2,1) rectangle (3,2);
        \draw[gray!45,line width=.3pt,step=1] (0,0) grid (6,3);
        \draw[blue!65!black,line width=.6pt] (0,1) rectangle (6,2);
        \draw[black!75,line width=.8pt] (6,3)--(6,0);
        \node[font=\small] at (3,3.9) {\(\kappa'(S_r)=j\)};
        \node[font=\small] at (3,-.75) {\(x_2-t_j\)};
      \end{scope}
      \draw[->,black!70,line width=.5pt] (6.7,1.5)--(8.0,1.5);
      \node[font=\small] at (7.5,-2.05) {\((x_2-t_i)-(x_2-t_j)=t_j-t_i\)};
    \end{scope}
    \begin{scope}[xshift=7.3cm]
      \node[font=\small] at (7.5,5.2) {Even column, \(d=4\)};
      \begin{scope}[shift={(0,0)}]
        \fill[black!12] (4,2) rectangle (5,3);
        \fill[black!12] (0,0) rectangle (1,1);
        \fill[orange!25] (3,1) rectangle (4,2);
        \draw[gray!45,line width=.3pt,step=1] (0,0) grid (6,3);
        \draw[blue!65!black,line width=.6pt] (0,1) rectangle (6,2);
        \draw[black!75,line width=.8pt] (6,3)--(6,0);
        \node[font=\scriptsize,blue!70!black,anchor=east,inner sep=1pt] at (-.35,1.5) {\(r\)};
        \node[font=\small] at (3,3.9) {\(\kappa(S_r)=i\)};
        \node[font=\small] at (3,-.75) {\(t_i-t_2\)};
      \end{scope}
      \begin{scope}[shift={(9,0)}]
        \fill[black!12] (4,2) rectangle (5,3);
        \fill[black!12] (0,0) rectangle (1,1);
        \fill[orange!25] (3,1) rectangle (4,2);
        \draw[gray!45,line width=.3pt,step=1] (0,0) grid (6,3);
        \draw[blue!65!black,line width=.6pt] (0,1) rectangle (6,2);
        \draw[black!75,line width=.8pt] (6,3)--(6,0);
        \node[font=\small] at (3,3.9) {\(\kappa'(S_r)=j\)};
        \node[font=\small] at (3,-.75) {\(t_j-t_2\)};
      \end{scope}
      \draw[->,black!70,line width=.5pt] (6.7,1.5)--(8.0,1.5);
      \node[font=\small] at (7.5,-2.05) {\((t_i-t_2)-(t_j-t_2)=t_i-t_j\)};
    \end{scope}
  \end{tikzpicture}
  \[
    \wt_\kappa(Q)-\wt_{\kappa'}(Q)
    =\left(\prod_{C\in Q\setminus\{C_{d,r}\}}\wt_\kappa(C)\right)f(\vec t), \qquad f(\vec t)=
    \begin{cases}
      t_j-t_i,&d\text{ odd},\\
      t_i-t_j,&d\text{ even}.
    \end{cases}
  \]
  \caption{Relabeling row \(r\) with the selected cells fixed.
    Odd and even columns give opposite weight differences.}
  \label{fig:fixed-selection-relabeling}
\end{figure}
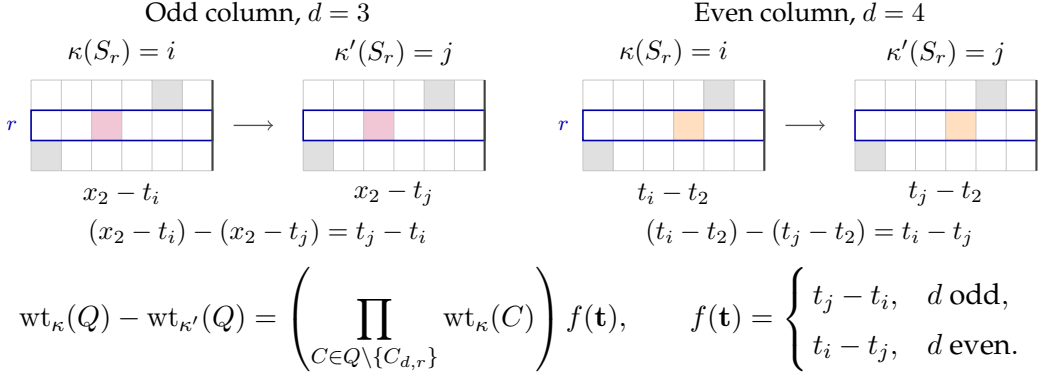

Fix \(a=(a_1,\ldots,a_m)\in\Red(w_\lambda^{-1})\) with \(m>0\).
For \(r\in[m+1]\), let \(D\subseteq[r-1]\) be a deletion set and put
\(E=[m]\setminus D=\{1=i_1<\cdots<i_\ell\}\).
Index the rows and south steps of \(P_{a_E}\) by \(E\).
By \Cref{lem:ferrers-interval-replay}, the word \(a_E\) is
endpoint-decreasing, and each insertion removes the preceding mark.
Thus \(i_1\trl_{a_E}\cdots\trl_{a_E}i_\ell\).
By \Cref{lem:block-path-successors}, consecutive retained rows have
vertical successors at ascents and diagonal successors at descents.

If \(e_s(a_E)\leq0\) for some \(s\in E\) with \(s<r\), set
\(\FL_{E,r}=0\). Otherwise, define \(\kappa_{E,r}\) by labeling each
row \(s<r\) with the position of the block inserted at \(s\) in
\(\Omega\)-parking for \(a_E\), and each row \(s\geq r\) with
\(a_s\). Write
\[
  \FL_{E,r}:=\FL_{P_{a_E},\kappa_{E,r}}.
\]
Thus the labels in rows before \(r\) have been replaced, while the
other rows still have their word labels.

Fix a cell selection \(Q\in\CSel(P_{a_E})\) with nonzero weight under
\(\kappa_{E,r}\). For \(s\geq r\), the row label is \(a_s\), so
\(\col_s(Q)<2a_s\). Consider a row \(i_k<r\).
If \(k=1\), its label is \(e_{i_1}(a_E)\), so
\(\col_{i_1}(Q)<2e_{i_1}(a_E)\).
Suppose \(k>1\).
At a descent, the insertion adds the leftmost block, so the row
label is \(e_{i_k}(a_E)\) and
\(\col_{i_k}(Q)<2e_{i_k}(a_E)\).
At an ascent, the insertion adds the rightmost block, so
\(S_{i_{k-1}}\rav S_{i_k}\) and
\(e_{i_{k-1}}(a_E)=e_{i_k}(a_E)\).
The successor inequality gives
\(\col_{i_k}(Q)\leq\col_{i_{k-1}}(Q)\).
Let \(j\in[k]\) be the position in \(a_E\) at which the increasing
run containing \(a_{i_k}\) begins. The first row of this run is
therefore indexed by \(i_j\). This row is either \(i_1\) or follows a descent, so its
label is \(e_{i_j}(a_E)\). Along the run, all endpoints are equal
and the selected columns weakly decrease. Since \(Q\) has nonzero
weight, we obtain
\[
  \col_{i_k}(Q)\leq\col_{i_j}(Q)
  <2e_{i_j}(a_E)=2e_{i_k}(a_E).
\]

Conversely, choose one cell in each row satisfying
\eqref{arr:csel_cond} and these column bounds. Since
\(e_s(a_E)\leq a_s\leq h_s(a_E)\), every selected cell lies above
\(P_{a_E}\). For \(s<r\), the label is an insertion position in
\(J_s(a_E)\), hence is at least \(e_s(a_E)\); for \(s\geq r\),
it is \(a_s\). The column bounds therefore make each cell
weight nonzero. Consequently, the nonzero terms of \(\FL_{E,r}\)
are indexed by \(Q\in\CSel(P_{a_E})\) satisfying
\begin{equation}\label{eq:mixed-label-columns}
  \col_s(Q)<
  \begin{cases}
    2e_s(a_E),&s<r,\\
    2a_s,&s\geq r.
  \end{cases}
\end{equation}
When the zero convention applies, these inequalities admit no
selection.

For \(r=1\), no letters have been deleted and no labels have been
replaced, so
\begin{equation}\label{eq:mixed-initial-state}
  \FL_{[m],1}=\FL_a(\vec x;\vec t).
\end{equation}
For \(r=m+1\), all retained rows have their block insertion labels.
If \(\des(a_E)<n\), the endpoints are positive and \(D\) is
admissible. Their multiplicities give
\(c(F_\Omega(a_E))=c_D\) by \eqref{eq:replay-code}.
By \Cref{tab:omega-dictionary,lem:forest-path-incidence}, these
labels agree with the canonical labels of \(P(c_D)\).
The two paths have the same successor relations, and
\eqref{eq:mixed-label-columns} places every nonzero selection above
\(P(c_D)\), whose row \(s\) ends at column \(2e_s(a_E)\).
Replacing \(P_{a_E}\) by \(P(c_D)\) therefore preserves the nonzero
selections and their weights.
If \(\des(a_E)\geq n\), the last endpoint is nonpositive, so the
zero convention applies. Together with
\Cref{cor:double-fundamental-path}, this proves
\begin{equation}\label{eq:mixed-terminal-state}
  \FL_{E,m+1}=
  \begin{cases}
    \mathfrak F_{c_D}(x_1,\ldots,x_n;\vec t),
      &\des(a_E)<n,\\
    0,&\text{otherwise}.
  \end{cases}
\end{equation}

Before comparing consecutive labelings, we handle the first
nonpositive endpoint.

\begin{lem}\label{lem:nonpositive-D-boundary}
Let \(D\) be a deletion set of \(a\), and put
\(E=[m]\setminus D=\{1=i_1<\cdots<i_\ell\}\).
Suppose \(2\leq k\leq\ell\) and \(D\subseteq[i_k-1]\).
If \(S_{i_{k-1}}\rad S_{i_k}\) and \(e_{i_k}(a_E)\leq0\), then
\[
  \FL_{E,i_k}=\FL_{E,i_k+1}=0.
\]
\end{lem}
\begin{proof}
If an earlier endpoint is nonpositive, both polynomials vanish by
convention. Otherwise,
\(e_{i_{k-1}}(a_E)=e_{i_k}(a_E)+1>0\), so
\(e_{i_{k-1}}(a_E)=1\) and \(e_{i_k}(a_E)=0\).
By \eqref{eq:mixed-label-columns}, a nonzero selection \( Q \) for
\(\FL_{E,i_k}\) would have \(\col_{i_{k-1}}(Q)=1\).
Neither an odd nor an even column in row \(i_k\) can then satisfy
the diagonal successor condition in \eqref{arr:csel_cond}.
Thus \(\FL_{E,i_k}=0\), and \(\FL_{E,i_k+1}=0\) by the zero convention.
\end{proof}

For the next two lemmas, fix a deletion set \(D\) of \(a\), and put
\(E=[m]\setminus D=\{1=i_1<\cdots<i_\ell\}\).
For \(2\leq k\leq\ell\), by \Cref{lem:ferrers-interval-replay}, we write 
\[
  J_{i_k}(a_E)=[a_{1}-\des(a_{i_1}\cdots a_{i_k}),
  a_{1}+\asc(a_{i_1}\cdots a_{i_k})]=:[L,R],
\]
 and let
\(p\) be the position of the block inserted at \(i_k\) in
\(\Omega\)-parking for \(a_E\).

\begin{lem}\label{lem:deletable-double-pattern}
If \(L>0\) and \(p\ne a_{i_k}\), then \(k\) is deletable in
\(a_E\) if and only if \(k\) is a double ascent or a double descent
in \(a_E\). For a deletable position \(k\),
\[
  \Delta_{a_E}(k)=
  \begin{cases}
    t_R-t_{a_{i_k}},&\text{if }k\text{ is a double ascent in }a_E,\\
    t_{a_{i_k}}-t_L,&\text{if }k\text{ is a double descent in }a_E.
  \end{cases}
\]
\end{lem}
\begin{proof}
The insertion rule and \Cref{lem:ferrers-interval-replay} give
\begin{equation}\label{eq:deletable-canonical}
  p=
  \begin{cases}
    R,&a_{i_{k-1}}<a_{i_k},\\
    L,&a_{i_{k-1}}>a_{i_k}.
  \end{cases}
\end{equation}
Since \(a_{i_k}\in[L,R]\), the assumptions \(L>0\) and
\(p\ne a_{i_k}\) imply \(a_{i_k}<R\) when \(p=R\), and
\(1\leq L<a_{i_k}\) when \(p=L\).
By \eqref{eq:replay-interval}, these are the inequalities in
\eqref{eq:successor-replay-delete}, giving the equivalence.
Substitution into \eqref{eq:successor-replay-root} gives the weights.
\end{proof}

For fixed \( 2\le k \le \ell \), assume further that \(D\subseteq[i_k-1]\), that is \( \{i_k<i_{k+1}< \cdots <i_\ell\}=\{i_k<i_{k}+1< \cdots <i_k+\ell-k\} \), and \(L>0\).
Let \(\mathcal U_{E,k}\) consist of the cell selections
\(Q'\in\CSel(P_{a_E})\) for which
\(\wt_{\kappa_{E,i_k}}(Q')\ne0\) and the following column condition
holds. If \(k=\ell\), then we take \(\col_{i_k}(Q')=1\).
If \(k<\ell\), then we take
\begin{equation}\label{eq:deletion-column}
  \col_{i_k}(Q')=
    2\bigl\lfloor\tfrac{1}{2}\col_{i_{k+1}}(Q')\bigr\rfloor+1
    \;(a_{i_k}<a_{i_{k+1}}),
  \quad
  \col_{i_k}(Q')=
    2\bigl\lceil\tfrac{1}{2}\col_{i_{k+1}}(Q')\bigr\rceil
    \;(a_{i_k}>a_{i_{k+1}}).
\end{equation}

\begin{lem}\label{lem:successor-replay-deletion-bijection}
Suppose \(k\) is deletable in \(a_E\). Deleting the cell in row
\(i_k\) defines a bijection
\[
  \mathcal U_{E,k}\longrightarrow
  \bigl\{\widetilde{Q}\in\CSel(P_{a_{E\setminus\{i_k\}}}) :
    \wt_{\kappa_{E\setminus\{i_k\},i_k+1}}(\widetilde{Q})\ne0\bigr\}.
\]
Under this map, the weight of every remaining cell is unchanged.
\end{lem}
\begin{proof}
The rows of \(P_{a_E}\) are indexed by \(E\), and those of
\(P_{a_{E\setminus\{i_k\}}}\) by \(E\setminus\{i_k\}\).
To construct the inverse, choose
\(\widetilde{Q}\in\CSel(P_{a_{E\setminus\{i_k\}}})\) such that
\(\wt_{\kappa_{E\setminus\{i_k\},i_k+1}}(\widetilde{Q})\ne0\).
Form a set of cells \(Q'\) by keeping every cell of \(\widetilde{Q}\)
in its original row and column and adding one cell in row \(i_k\).
For \(k<\ell\), choose the added column by \eqref{eq:deletion-column},
using \(\col_{i_{k+1}}(\widetilde{Q})\) on the right-hand side.
For \(k=\ell\), set \(Q'=\widetilde{Q}\cup\{C_{1,i_k}\}\).

Since \(k\) is deletable, it is a double ascent or a double descent.
Thus, for \(k<\ell\), \Cref{lem:block-path-successors} and
\eqref{arr:csel_cond} make the successor inequality between rows
\(i_{k-1}\) and \(i_{k+1}\) of \(\widetilde{Q}\) equivalent to
\(\col_{i_{k-1}}(\widetilde{Q})\geq\col_{i_k}(Q')\).
Together with \eqref{eq:deletion-column}, this is also equivalent
to the two successor inequalities involving row \(i_k\) in \(Q'\).
For \(k=\ell\), we have \(a_{i_{k-1}}<a_{i_k}\), so column \(1\)
satisfies \eqref{arr:csel_cond}.

If \(k<\ell\) and \(k\) is a double ascent, \eqref{eq:mixed-label-columns}
gives
\(\col_{i_k}(Q')\leq\col_{i_{k-1}}(\widetilde{Q})<2L\leq2a_{i_k}\).
If \(k\) is a double descent, then
\(\col_{i_{k-1}}(\widetilde{Q})<2L+2\); since the inserted column
is even and \(L<a_{i_k}\), we have
\(\col_{i_k}(Q')\leq2L<2a_{i_k}\).
For \(k=\ell\), \(\col_{i_k}(Q')=1<2a_{i_k}\).
Thus the inserted cell has nonzero weight.
All other columns and row labels are unchanged, so
\(Q'\in\mathcal U_{E,k}\), and the same successor equivalences and
\eqref{eq:mixed-label-columns} verify deletion in the reverse direction.
The weights of the retained cells are unchanged.
By construction, insertion and deletion are inverse.
\end{proof}

We now compare the polynomials for two consecutive labelings.

\begin{lem}\label{lem:successor-replay-straightening}
Let \(D\) be a deletion set of \(a\), and put
\(E=[m]\setminus D=\{1=i_1<\cdots<i_\ell\}\).
Suppose \(k\in[\ell]\) and \(D\subseteq[i_k-1]\). Then
\begin{equation}\label{eq:successor-replay-recurrence}
  \FL_{E,i_k}-\FL_{E,i_k+1}
  =\begin{cases}
    \Delta_{a_E}(k)\FL_{E\setminus\{i_k\},i_k+1},
      &k\text{ is deletable in }a_E,\\
    0,&\text{otherwise.}
  \end{cases}
\end{equation}
\end{lem}

\begin{proof}
Since \(i_1=1\) and \(a_1=n\), the case \(k=1\) is immediate.

Assume \(k>1\). If an earlier endpoint is nonpositive, all terms
vanish because deletion at \(i_k\) leaves earlier insertions unchanged.
If the current endpoint is the first nonpositive one, it is zero
and follows a descent. Then \Cref{lem:nonpositive-D-boundary}
gives zero on the left, and the deletion condition fails.

Assume \(J_{i_k}(a_E)=[L,R]\) with \(L>0\), and let
\(p=\kappa_{E,i_k+1}(S_{i_k})\) be the position of the block inserted
at \(i_k\).
If \(p=a_{i_k}\), the labelings coincide and deletion is not allowed,
so assume \(p\ne a_{i_k}\).

Let \(Q\in\CSel(P_{a_E})\), and let
\(\mathcal C\subseteq\CSel(P_{a_E})\) be the collection of all cell
selections that select the same cells as \(Q\) outside row \(i_k\).
The weights outside row \(i_k\) are common to both labelings;
if their product is zero, \(\mathcal C\) contributes zero.
Otherwise, for \(Q'\in\mathcal C\), \eqref{eq:mixed-label-columns} gives
\(\col_{i_{k-1}}(Q')<2e_{i_{k-1}}(a_E)\).
Using \eqref{arr:csel_cond} and \eqref{eq:replay-interval}, we obtain
\[
  \begin{aligned}
    \col_{i_k}(Q')&<2L&&\text{if }a_{i_{k-1}}<a_{i_k},\\
    \col_{i_k}(Q')&\leq2L&&\text{if }a_{i_{k-1}}>a_{i_k}.
  \end{aligned}
\]
Together with \eqref{eq:deletable-canonical}, these imply
\(\col_{i_k}(Q')<2a_{i_k}\) and
\(\col_{i_k}(Q')\leq2p\).
Thus \(Q'\) has nonzero weight in \(\FL_{E,i_k}\), and
\Cref{fig:fixed-selection-relabeling} gives
\[
   \wt_{\kappa_{E,i_k}}(Q')-\wt_{\kappa_{E,i_k+1}}(Q')
   =(-1)^{\col_{i_k}(Q')}(t_{a_{i_k}}-t_p)
      \prod_{\substack{C_{j,s}\in Q\\s\ne i_k}}
        \wt_{\kappa_{E,i_k}}(C_{j,s}).
\]

For \(q\geq1\), pair columns \((2q,2q+1)\) if
\(a_{i_{k-1}}<a_{i_k}\), and \((2q-1,2q)\) if
\(a_{i_{k-1}}>a_{i_k}\). The weight differences in each pair cancel.
For \(k<\ell\), \eqref{arr:csel_cond} shows that the admissible column
interval starts at
the right-hand side of \eqref{eq:deletion-column} and ends at an
odd column if \(a_{i_{k-1}}<a_{i_k}\) and an even column if
\(a_{i_{k-1}}>a_{i_k}\).
Hence only the lower endpoint remains unpaired at a double ascent
or double descent; all columns cancel at the other two patterns.
For \(k=\ell\), column \(1\) remains only when
\(a_{i_{k-1}}<a_{i_k}\).
By \eqref{eq:deletable-canonical} and
\Cref{lem:deletable-double-pattern}, an unpaired choice occurs
exactly when \(k\) is deletable, and its row weight difference is
\(\Delta_{a_E}(k)\); see \Cref{fig:local-deletion-cancellation}.

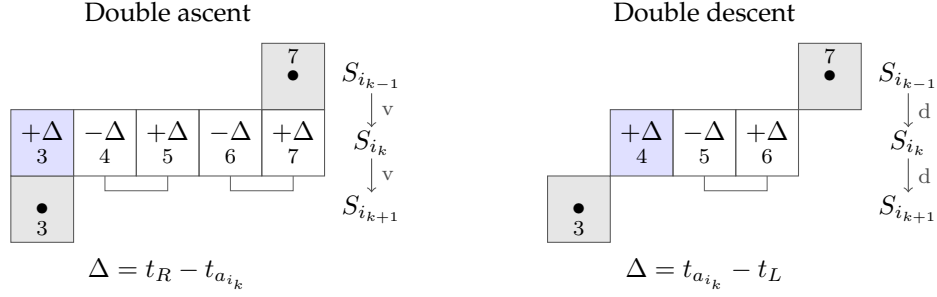
\begin{figure}
  \centering
  \begin{tikzpicture}[x=0.83cm,y=0.78cm]
    \begin{scope}
      \node[font=\small] at (2.5,2.5) {Double ascent};
      \draw[black!65,fill=black!10] (4,0.8375) rectangle (5,1.9625);
      \node[font=\small] at (4.5,1.4) {$\bullet$};
      \node[font=\scriptsize] at (4.5,1.75) {$7$};
      \draw[black!65,fill=black!10] (0,-1.4125) rectangle (1,-0.2875);
      \node[font=\small] at (0.5,-0.85) {$\bullet$};
      \node[font=\scriptsize] at (0.5,-1.18) {$3$};
      \fill[blue!12] (0,-0.2875) rectangle (1,0.8375);
      \foreach \j/\d/\sgn in {0/3/+,1/4/-,2/5/+,3/6/-,4/7/+} {
        \draw[black!65] (\j,-0.2875) rectangle (\j+1,0.8375);
        \node[font=\small] at (\j+0.5,0.45) {$\sgn\Delta$};
        \node[font=\scriptsize] at (\j+0.5,0.05) {$\d$};
      }
      \draw[black!60] (1.5,-0.2875)--(1.5,-0.5375)--(2.5,-0.5375)--(2.5,-0.2875);
      \draw[black!60] (3.5,-0.2875)--(3.5,-0.5375)--(4.5,-0.5375)--(4.5,-0.2875);
      \node[font=\small] at (2.5,-1.95) {$\Delta=t_R-t_{a_{i_k}}$};
      \node[font=\small,inner sep=1pt] (ascent-r) at (5.75,0.275) {$S_{i_k}$};
      \node[font=\small,inner sep=1pt] (ascent-prev) at (5.75,1.4)
        {$S_{i_{k-1}}$};
      \node[font=\small,inner sep=1pt] (ascent-next) at (5.75,-0.85)
        {$S_{i_{k+1}}$};
      \draw[->,black!70] (ascent-prev.south)--(ascent-r.north)
        node[midway,right,font=\scriptsize] {$\mathrm{v}$};
      \draw[->,black!70] (ascent-r.south)--(ascent-next.north)
        node[midway,right,font=\scriptsize] {$\mathrm{v}$};
    \end{scope}
    \begin{scope}[xshift=7.1cm]
      \node[font=\small] at (2.5,2.5) {Double descent};
      \draw[black!65,fill=black!10] (4,0.8375) rectangle (5,1.9625);
      \node[font=\small] at (4.5,1.4) {$\bullet$};
      \node[font=\scriptsize] at (4.5,1.75) {$7$};
      \draw[black!65,fill=black!10] (0,-1.4125) rectangle (1,-0.2875);
      \node[font=\small] at (0.5,-0.85) {$\bullet$};
      \node[font=\scriptsize] at (0.5,-1.18) {$3$};
      \fill[blue!12] (1,-0.2875) rectangle (2,0.8375);
      \foreach \j/\d/\sgn in {1/4/+,2/5/-,3/6/+} {
        \draw[black!65] (\j,-0.2875) rectangle (\j+1,0.8375);
        \node[font=\small] at (\j+0.5,0.45) {$\sgn\Delta$};
        \node[font=\scriptsize] at (\j+0.5,0.05) {$\d$};
      }
      \draw[black!60] (2.5,-0.2875)--(2.5,-0.5375)--(3.5,-0.5375)--(3.5,-0.2875);
      \node[font=\small] at (2.5,-1.95) {$\Delta=t_{a_{i_k}}-t_L$};
      \node[font=\small,inner sep=1pt] (descent-r) at (5.75,0.275) {$S_{i_k}$};
      \node[font=\small,inner sep=1pt] (descent-prev) at (5.75,1.4)
        {$S_{i_{k-1}}$};
      \node[font=\small,inner sep=1pt] (descent-next) at (5.75,-0.85)
        {$S_{i_{k+1}}$};
      \draw[->,black!70] (descent-prev.south)--(descent-r.north)
        node[midway,right,font=\scriptsize] {$\mathrm{d}$};
      \draw[->,black!70] (descent-r.south)--(descent-next.north)
        node[midway,right,font=\scriptsize] {$\mathrm{d}$};
    \end{scope}
  \end{tikzpicture}
  \caption{Black dots mark the fixed selected cells. Contributions from paired choices
  cancel; the blue cell gives the deletion weight.
  }
  \label{fig:local-deletion-cancellation}
\end{figure}

If \(k\) is not deletable, all contributions cancel.
Otherwise, the unpaired selections over all \(\mathcal C\) form
\(\mathcal U_{E,k}\). Applying
\Cref{lem:successor-replay-deletion-bijection} to their weights
outside row \(i_k\), and multiplying by \(\Delta_{a_E}(k)\), gives
the first case of \eqref{eq:successor-replay-recurrence}.
\end{proof}

\begin{proof}[Proof of \Cref{thm:double-schur-successor-replay}]
The empty word satisfies the formula by convention. For \(m>0\),
start with \eqref{eq:mixed-initial-state} and apply
\eqref{eq:successor-replay-recurrence} in the order \(r=1,\ldots,m\),
with \(i_k=r\) at each step. The term \(\FL_{E,r+1}\) keeps the letter and replaces
its row label. When the current position is deletable, the second
term removes that letter and multiplies the polynomial for the
shorter word by its deletion weight.
Thus the choices are precisely those defining a deletion set \(D\).

At \(r=m+1\), \eqref{eq:mixed-terminal-state} gives
\(\mathfrak F_{c_D}(x_1,\ldots,x_n;\vec t)\) for an admissible
\(D\), and zero otherwise. The product of the deletion weights is
\(\wt_a(D)\), proving \eqref{eq:word-successor-replay-expansion}.
\end{proof}

Specializing \Cref{thm:double-schur-successor-replay} at
\(\vec t=\vec0\) eliminates every term with a nonempty deletion set.
For a nonempty \(a\in\Red(w_\lambda^{-1})\), let \(q(a)\) be its
number of increasing runs, with lengths
\((\ell_1,\ldots,\ell_{q(a)})\) from left to right.
The empty deletion set is admissible exactly when \(q(a)\leq n\),
in which case its code is
\[
  c_\varnothing(a)=(0^{n-q(a)},\ell_{q(a)},\ldots,\ell_1).
\]
For the empty word, put \(q(\varnothing)=0\) and
\(c_\varnothing(\varnothing)=(0^n)\).
Summing the surviving terms in \eqref{eq:double-schur-lattice-path}
gives the following expansion.

\begin{cor}\label{rem:Gessel-recover}
For every partition \(\lambda\) with at most \(n\) parts,
\[
  s_\lambda(\vec x_n)
  =\sum_{\substack{a\in\Red(w_\lambda^{-1})\\q(a)\leq n}}
     \mathfrak F_{c_\varnothing(a)}(\vec x_n;\vec0).
\]
\end{cor}

We compare this expansion with Gessel's tableau formula
\cite{Gessel1984}. Write \(\mathfrak F_\gamma(\vec x_n)\) for the
fundamental quasisymmetric polynomial indexed by a composition
\(\gamma\), with \(\mathfrak F_\varnothing=1\) and
\(\mathfrak F_\gamma=0\) when \(\gamma\) has more than \(n\) parts.
A \emph{standard Young tableau} of shape \(\lambda\) fills
\(D(\lambda)\) with \(1,\ldots,m\), each used once, so that entries
increase along rows and columns. Let
\(\operatorname{SYT}(\lambda)\) denote the set of these tableaux.
Under the heap correspondence in \Cref{subsec:double-schubert-path},
place \(k\) in the box selected at step \(k\) of the linear
extension. This gives a bijection between reduced words
\(a\in\Red(w_\lambda^{-1})\) and tableaux
\(T\in\operatorname{SYT}(\lambda)\).

The box labels \(n-i+j\) give
\[
  a_k>a_{k+1}
  \quad\Longleftrightarrow\quad
  k+1\text{ lies in a lower row than }k\text{ in }T.
\]
Define \(\operatorname{DesComp}(T)\) as the composition of \(m\) whose partial
sums, excluding the total \(m\), are exactly these tableau descent
positions; for the empty tableau, use the empty composition.
The increasing run lengths of \(a\) are therefore the parts of
\(\operatorname{DesComp}(T)\). When \(q(a)\leq n\), its summand
in \Cref{rem:Gessel-recover} is
\(\mathfrak F_{\rev(\operatorname{DesComp}(T))}(\vec x_n)\).

Reversing the order of the variables reverses the composition index:
\[
  \mathfrak F_{\rev(\gamma)}(x_n,\ldots,x_1)
  =\mathfrak F_\gamma(x_1,\ldots,x_n).
\]
Applying \Cref{rem:Gessel-recover} at \((x_n,\ldots,x_1)\) and
using the symmetry of \(s_\lambda\) therefore gives
\[
  s_\lambda(x_1,\ldots,x_n)
  =\sum_{T\in\operatorname{SYT}(\lambda)}
     \mathfrak F_{\operatorname{DesComp}(T)}(x_1,\ldots,x_n).
\]
Tableaux with more than \(n\) parts in their descent composition
contribute zero, so the sum may range over all of
\(\operatorname{SYT}(\lambda)\).

\begin{remark}\label{rem:double-schubert-cancellation}
For a general permutation, a word polynomial need not have
Graham-positive coefficients in the double forest basis.
The words \((3,4,3,1)\) and \((3,4,1,3)\) differ by a commutation
move. With \(\alpha_i=t_{i+1}-t_i\), their expansions are
\begin{align*}
  \FL_{(3,4,3,1)}
  &=\PP_{(1,1,2)}
    +\alpha_2\PP_{(1,0,2)}
    -\alpha_2\PP_{(2,0,1)}
   -\alpha_2\PP_{(3)}
    +\alpha_2(\alpha_2+\alpha_3)\PP_{(1,0,1)},\\
  \FL_{(3,4,1,3)}
  &=\PP_{(2,0,2)}
    -\alpha_2\PP_{(1,0,2)}
    +\alpha_2\PP_{(2,0,1)}
   +\alpha_2\PP_{(3)}
    -\alpha_2(\alpha_2+\alpha_3)\PP_{(1,0,1)}.
\end{align*}
Adding the two expansions cancels every term with a positive-root
factor:
\[
  \FL_{(3,4,3,1)}+\FL_{(3,4,1,3)}
  =\PP_{(1,1,2)}+\PP_{(2,0,2)}.
\]
The negative terms cancel between distinct word paths for the same
permutation. A weight-preserving sign-reversing involution on the
contributing cell selections would give a combinatorial explanation
of such cancellations.
\end{remark}

\subsection{From Grothendieck to grove polynomials}
\label{subsec:grothendieck-grove-expansion}

In the grove basis, Grothendieck polynomials have nonnegative
coefficients by \cite[Corollary~4.8]{NST2026Grove}.
We interpret these coefficients by a bijection between the injective
cell selections of \Cref{thm:grothendieck-lattice-path} and the
multi-cell selections of \Cref{thm:grove-multicell-model}.

Fix \(n\geq1\). For an endpoint-decreasing word
\(b=(b_1,\ldots,b_\ell)\) with positive endpoints, define its
\emph{endpoint code} by
\[
  c_i^{\mathrm{end}}(b)
  =\bigl|\{r\in[\ell]:e_r(b)=i\}\bigr|,
  \qquad
  c^{\mathrm{end}}(b)=(c_i^{\mathrm{end}}(b))_{i\geq1},
\]
and set \(P_b^{\mathrm{end}}=P(c^{\mathrm{end}}(b))\).
For the empty word, set \(c^{\mathrm{end}}(\varnothing)=0\) and
\(P_\varnothing^{\mathrm{end}}=P(0)\).
Since the endpoints are weakly decreasing, the \(r\)-th south step
of \(P_b^{\mathrm{end}}\) lies on \(x=2e_r(b)\).
The \(r\)-th south step of the word path \(P_b\) lies on
\(x=2(e_r(b)+s(b))\), so the two paths have the same successor relations.
Thus \(c^{\mathrm{end}}(b)\) uses the endpoints before translation,
whereas \(c(b)\) is the code of the translated word path; see
\Cref{tab:word-endpoint-codes}.
The grove polynomial in the expansion will be indexed by
\(c^{\mathrm{end}}(b)\).

\begin{table}[htbp]
  \centering
  \small
  \renewcommand{\arraystretch}{1.2}
  \begin{tabular}{@{}lll@{}}
    \toprule
    Quantity&Definition&Role\\
    \midrule
    \(e_r(b)\)&\(\min J_r(b)\)&Endpoint before translation\\
    \(h_r(b)\)&\(e_r(b)+s(b)\)&Position \(x=2h_r(b)\) of \(S_r\) in \(P_b\)\\
    \(c_i(b)\)&\(\#\{r:h_r(b)=i\}\)&Code of \(P_b\)\\
    \(c_i^{\mathrm{end}}(b)\)&\(\#\{r:e_r(b)=i\}\)&Index of the grove polynomial\\
    \bottomrule
  \end{tabular}
  \caption{Endpoints and codes for an endpoint-decreasing word
  with positive endpoints.}
  \label{tab:word-endpoint-codes}
\end{table}

For \(w\in S_{n+1}\), recall that \(\EH_n(w)\) consists of the
endpoint-decreasing Hecke words in \Cref{thm:grothendieck-lattice-path}.
For the grove expansion, we use the subset
\begin{equation}\label{eq:stable-hecke-cores}
  \EH_n^+(w)
  =\{b\in\EH_n(w):
      e_r(b)\geq1 \ (1\leq r\leq|b|),\quad
      b_r\neq b_{r+1} \ (1\leq r<|b|)
  \}.
\end{equation}
Positive endpoints ensure that \(P_b^{\mathrm{end}}\in\Path\).
Since a vertical successor joins consecutive rows, the
condition \( b_r\neq b_{r+1} \) makes the letter inequality in
\Cref{lem:block-path-successors} strict along vertical arrows.
We will use this strict inequality to construct the inverse.
Both conditions in \eqref{eq:stable-hecke-cores} are vacuous for the empty word, so
\(\EH_n^+(\mathrm{id})=\{\varnothing\}\).

\begin{thm}\label{thm:grothendieck-grove-labeling-positivity}
For every \(w\in S_{n+1}\),
\begin{equation}\label{eq:grothendieck-grove-labeling-expansion}
  G_w^{(\beta)}(\vec x)
  =\sum_{b\in\EH_n^+(w^{-1})}
     \beta^{|b|-\ell(w)}
     \widetilde\PP_{c^{\mathrm{end}}(b)}^{(\beta)}(\vec x).
\end{equation}
Consequently, for every \(c\in\Codes_n\) with \(|c|\geq\ell(w)\),
\begin{equation}\label{eq:grothendieck-grove-core-count}
  [\widetilde\PP_c^{(\beta)}]G_w^{(\beta)}
  =\beta^{|c|-\ell(w)}
    \bigl|\{b\in\EH_n^+(w^{-1}):
              c^{\mathrm{end}}(b)=c\}\bigr|.
\end{equation}
For \(|c|<\ell(w)\), we have
\([\widetilde\PP_c^{(\beta)}]G_w^{(\beta)}=0\).
\end{thm}

\begin{exam}\label{ex:grothendieck-grove-expansion}
Let \(w=24513\in S_5\), so \(w^{-1}=41523\) and \(\ell(w)=5\).
For \(b=(3,2,4,1,3,1)\), we have
\( \operatorname{Dem}(b)=41523\) and \((e_1(b),\ldots,e_6(b))=(3,2,2,1,1,0) \).
Thus \(b\in\EH_4(w^{-1})\), but
\(b\not\in\EH_4^+(w^{-1})\).
The words that do belong to \(\EH_4^+(w^{-1})\),
together with their endpoints and endpoint codes, are
\[
\begin{array}{c|c|c}
  b & (e_1(b),\ldots,e_{|b|}(b)) & c^{\mathrm{end}}(b)\\ \hline
  (3,2,4,1,3)   & (3,2,2,1,1)   & (2,2,1)\\
  (3,4,2,1,3)   & (3,3,2,1,1)   & (2,1,2)\\
  (3,4,2,3,1)   & (3,3,2,2,1)   & (1,2,2)\\
  (3,4,2,3,1,3) & (3,3,2,2,1,1) & (2,2,2)\\
  (3,4,2,4,1,3) & (3,3,2,2,1,1) & (2,2,2).
\end{array}
\]
Thus \Cref{thm:grothendieck-grove-labeling-positivity} gives
\[
  G_{24513}^{(\beta)}
  =\widetilde\PP_{(2,2,1)}^{(\beta)}
     +\widetilde\PP_{(2,1,2)}^{(\beta)}
     +\widetilde\PP_{(1,2,2)}^{(\beta)}
     +2\beta\,\widetilde\PP_{(2,2,2)}^{(\beta)}.
\]
The coefficient \(2\beta\) comes from the two words of length \(6\)
with endpoint code \((2,2,2)\).
\end{exam}

We construct the bijection from \( \bigsqcup_{a\in\EH_n(w)}\ICSel_a^\times \) onto \( \bigsqcup_{b\in\EH_n^+(w)}\MCSel((P_b^{\mathrm{end}})^\circ) \).
Fix \(w\in S_{n+1}\), \(a\in\EH_n(w)\), and
\(Q\in\ICSel_a^\times\), and put \(m=|a|\).
The cell selected in row \(t\) of the path \(P_a^\circ\)
gives the pair \((a_t,\col_t(Q))\). Set
\[
  X(a,Q)=\{(a_t,\col_t(Q)):1\leq t\leq m\}.
\]
By \eqref{eq:injective-hecke-cell-selection}, the pairs in \(X(a,Q)\)
are distinct. Equation \eqref{eq:nonzero-injective-cell-selection}
gives \(1\leq q\leq j\leq n\) for every \((j,q)\in X(a,Q)\).
Read them in decreasing order of \(q\), resolving ties by increasing
\(j\), and write this ordered list as
\[
  (j_1,q_1),\ldots,(j_m,q_m).
\]
The index \(k\) in \((j_k,q_k)\) refers to this reading order,
which can differ from the order of the rows of \(a\).

For \(i\in\ZZ\) and \(h\in\mathbb Z_{\geq0}\), let
\[
  I_{i,h}=[i,i+1]\times[h,h+1],
  \qquad I_{i,0}=I_i.
\]
A \emph{stack} at position \(i\) consists of the blocks
\(I_{i,0},\ldots,I_{i,h-1}\) for \(h\geq1\); its height is \(h\), and its bottom
block is \(I_{i,0}\).
Components correspond to maximal intervals of occupied bottom positions,
as in \Cref{subsec:endpoint-decreasing}.
Each component has one marked bottom block.
Every inserted block will be labeled by its pair \((j_k,q_k)\).

We define \emph{\(\Omega\)-multiparking} of the ordered pairs as follows.
Let \(\mathcal I_k\) be the set of blocks after the first \(k\)
pairs have been read, with \(\mathcal I_0=\varnothing\).
If \(I_{j_k,0}\notin\mathcal I_{k-1}\), insert \(I_{j_k,0}\).
Otherwise, let \(\{I_{L,0},\ldots,I_{R,0}\}\) be the component
containing \(I_{j_k,0}\).
Suppose its marked block is at position \(p\), has label
\((j_s,q_s)\) with \(s<k\), and belongs to a stack of height \(h\).
Define
\begin{equation}\label{eq:stacked-block-rule}
  \mathcal I_k=
  \begin{cases}
    \mathcal I_{k-1}\cup\{I_{L-1,0}\},&j_k<j_s,\\
    \mathcal I_{k-1}\cup\{I_{p,h}\},&j_k=j_s,\\
    \mathcal I_{k-1}\cup\{I_{R+1,0}\},&j_k>j_s.
  \end{cases}
\end{equation}
Label the new block by \((j_k,q_k)\).
Whenever a bottom block is inserted, remove the existing marks
from its component and mark that new block.
The marks in other components are unchanged.
When a block is added above a marked block, all marks stay in place.

When a bottom block is inserted, assign to it the leftmost occupied
position in its component immediately after insertion.
We define the \emph{endpoint sequence} of the \(\Omega\)-multiparking
to be these values in bottom-block creation order.

\begin{exam}\label{ex:bundle-labeling-change}
Take \(n=5\), \(w=124653\), and
\(a=(5,3,3,5,3,4,5)\in\EH_5(w)\).
Let \(Q\in\ICSel_a^\times\) have selected columns
\((3,3,2,2,1,1,1)\). The pairs of \(X(a,Q)\) are read in the order
\[
  (3,3),\ (5,3),\ (3,2),\ (5,2),\ (3,1),\ (4,1),\ (5,1).
\]
\Cref{fig:stacked-block-insertion} shows the seven insertions.
The first two pairs create stacks at positions \(3\) and \(5\).
The next three pairs add two blocks to the stack at \(3\) and one
to the stack at \(5\).
Inserting \((4,1)\) joins the two components and marks the new
bottom block at \(4\).
For the last pair \((5,1)\), the marked block has letter \(4\).
Since \(5>4\), the rule adds a bottom block at position \(6\);
it does not add a block to the existing stack at \(5\).
\end{exam}

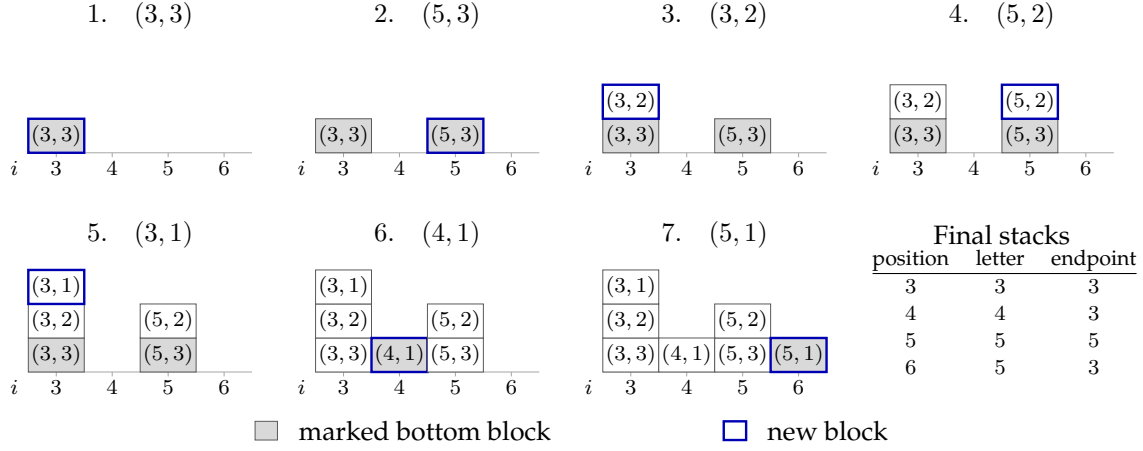
\begin{figure}
  \centering
  \begin{tikzpicture}[x=.74cm,y=.45cm]
    \tikzset{
      stackblock/.style={draw=black!65,line width=.35pt,fill=white,
        minimum width=.74cm,minimum height=.45cm,inner sep=0pt,font=\scriptsize},
      marked/.style={fill=black!15},
      newblock/.style={draw=blue!70!black,line width=.9pt}
    }
    \begin{scope}[xshift=0.0cm,yshift=2.9cm]
      \node[font=\small] at (2,4.05) {\(1.\quad(3,3)\)};
      \draw[black!35,line width=.3pt] (0,0)--(4,0);
      \node[font=\scriptsize] at (-.25,-.45) {\(i\)};
      \draw[black!35,line width=.3pt] (0.5,0)--(0.5,-.13);
      \node[font=\scriptsize] at (0.5,-.45) {$ 3 $};
      \draw[black!35,line width=.3pt] (1.5,0)--(1.5,-.13);
      \node[font=\scriptsize] at (1.5,-.45) {$ 4 $};
      \draw[black!35,line width=.3pt] (2.5,0)--(2.5,-.13);
      \node[font=\scriptsize] at (2.5,-.45) {$ 5 $};
      \draw[black!35,line width=.3pt] (3.5,0)--(3.5,-.13);
      \node[font=\scriptsize] at (3.5,-.45) {$ 6 $};
      \node[stackblock,marked,newblock] at (0.5,0.5) {$(3,3)$};
    \end{scope}
    \begin{scope}[xshift=3.8cm,yshift=2.9cm]
      \node[font=\small] at (2,4.05) {\(2.\quad(5,3)\)};
      \draw[black!35,line width=.3pt] (0,0)--(4,0);
      \node[font=\scriptsize] at (-.25,-.45) {\(i\)};
      \draw[black!35,line width=.3pt] (0.5,0)--(0.5,-.13);
      \node[font=\scriptsize] at (0.5,-.45) {$ 3 $};
      \draw[black!35,line width=.3pt] (1.5,0)--(1.5,-.13);
      \node[font=\scriptsize] at (1.5,-.45) {$ 4 $};
      \draw[black!35,line width=.3pt] (2.5,0)--(2.5,-.13);
      \node[font=\scriptsize] at (2.5,-.45) {$ 5 $};
      \draw[black!35,line width=.3pt] (3.5,0)--(3.5,-.13);
      \node[font=\scriptsize] at (3.5,-.45) {$ 6 $};
      \node[stackblock,marked] at (0.5,0.5) {$(3,3)$};
      \node[stackblock,marked,newblock] at (2.5,0.5) {$(5,3)$};
    \end{scope}
    \begin{scope}[xshift=7.6cm,yshift=2.9cm]
      \node[font=\small] at (2,4.05) {\(3.\quad(3,2)\)};
      \draw[black!35,line width=.3pt] (0,0)--(4,0);
      \node[font=\scriptsize] at (-.25,-.45) {\(i\)};
      \draw[black!35,line width=.3pt] (0.5,0)--(0.5,-.13);
      \node[font=\scriptsize] at (0.5,-.45) {$ 3 $};
      \draw[black!35,line width=.3pt] (1.5,0)--(1.5,-.13);
      \node[font=\scriptsize] at (1.5,-.45) {$ 4 $};
      \draw[black!35,line width=.3pt] (2.5,0)--(2.5,-.13);
      \node[font=\scriptsize] at (2.5,-.45) {$ 5 $};
      \draw[black!35,line width=.3pt] (3.5,0)--(3.5,-.13);
      \node[font=\scriptsize] at (3.5,-.45) {$ 6 $};
      \node[stackblock,marked] at (0.5,0.5) {$(3,3)$};
      \node[stackblock,newblock] at (0.5,1.5) {$(3,2)$};
      \node[stackblock,marked] at (2.5,0.5) {$(5,3)$};
    \end{scope}
    \begin{scope}[xshift=11.4cm,yshift=2.9cm]
      \node[font=\small] at (2,4.05) {\(4.\quad(5,2)\)};
      \draw[black!35,line width=.3pt] (0,0)--(4,0);
      \node[font=\scriptsize] at (-.25,-.45) {\(i\)};
      \draw[black!35,line width=.3pt] (0.5,0)--(0.5,-.13);
      \node[font=\scriptsize] at (0.5,-.45) {$ 3 $};
      \draw[black!35,line width=.3pt] (1.5,0)--(1.5,-.13);
      \node[font=\scriptsize] at (1.5,-.45) {$ 4 $};
      \draw[black!35,line width=.3pt] (2.5,0)--(2.5,-.13);
      \node[font=\scriptsize] at (2.5,-.45) {$ 5 $};
      \draw[black!35,line width=.3pt] (3.5,0)--(3.5,-.13);
      \node[font=\scriptsize] at (3.5,-.45) {$ 6 $};
      \node[stackblock,marked] at (0.5,0.5) {$(3,3)$};
      \node[stackblock] at (0.5,1.5) {$(3,2)$};
      \node[stackblock,marked] at (2.5,0.5) {$(5,3)$};
      \node[stackblock,newblock] at (2.5,1.5) {$(5,2)$};
    \end{scope}
    \begin{scope}[xshift=0.0cm,yshift=0.0cm]
      \node[font=\small] at (2,4.05) {\(5.\quad(3,1)\)};
      \draw[black!35,line width=.3pt] (0,0)--(4,0);
      \node[font=\scriptsize] at (-.25,-.45) {\(i\)};
      \draw[black!35,line width=.3pt] (0.5,0)--(0.5,-.13);
      \node[font=\scriptsize] at (0.5,-.45) {$ 3 $};
      \draw[black!35,line width=.3pt] (1.5,0)--(1.5,-.13);
      \node[font=\scriptsize] at (1.5,-.45) {$ 4 $};
      \draw[black!35,line width=.3pt] (2.5,0)--(2.5,-.13);
      \node[font=\scriptsize] at (2.5,-.45) {$ 5 $};
      \draw[black!35,line width=.3pt] (3.5,0)--(3.5,-.13);
      \node[font=\scriptsize] at (3.5,-.45) {$ 6 $};
      \node[stackblock,marked] at (0.5,0.5) {$(3,3)$};
      \node[stackblock] at (0.5,1.5) {$(3,2)$};
      \node[stackblock,newblock] at (0.5,2.5) {$(3,1)$};
      \node[stackblock,marked] at (2.5,0.5) {$(5,3)$};
      \node[stackblock] at (2.5,1.5) {$(5,2)$};
    \end{scope}
    \begin{scope}[xshift=3.8cm,yshift=0.0cm]
      \node[font=\small] at (2,4.05) {\(6.\quad(4,1)\)};
      \draw[black!35,line width=.3pt] (0,0)--(4,0);
      \node[font=\scriptsize] at (-.25,-.45) {\(i\)};
      \draw[black!35,line width=.3pt] (0.5,0)--(0.5,-.13);
      \node[font=\scriptsize] at (0.5,-.45) {$ 3 $};
      \draw[black!35,line width=.3pt] (1.5,0)--(1.5,-.13);
      \node[font=\scriptsize] at (1.5,-.45) {$ 4 $};
      \draw[black!35,line width=.3pt] (2.5,0)--(2.5,-.13);
      \node[font=\scriptsize] at (2.5,-.45) {$ 5 $};
      \draw[black!35,line width=.3pt] (3.5,0)--(3.5,-.13);
      \node[font=\scriptsize] at (3.5,-.45) {$ 6 $};
      \node[stackblock] at (0.5,0.5) {$(3,3)$};
      \node[stackblock] at (0.5,1.5) {$(3,2)$};
      \node[stackblock] at (0.5,2.5) {$(3,1)$};
      \node[stackblock,marked,newblock] at (1.5,0.5) {$(4,1)$};
      \node[stackblock] at (2.5,0.5) {$(5,3)$};
      \node[stackblock] at (2.5,1.5) {$(5,2)$};
    \end{scope}
    \begin{scope}[xshift=7.6cm,yshift=0.0cm]
      \node[font=\small] at (2,4.05) {\(7.\quad(5,1)\)};
      \draw[black!35,line width=.3pt] (0,0)--(4,0);
      \node[font=\scriptsize] at (-.25,-.45) {\(i\)};
      \draw[black!35,line width=.3pt] (0.5,0)--(0.5,-.13);
      \node[font=\scriptsize] at (0.5,-.45) {$ 3 $};
      \draw[black!35,line width=.3pt] (1.5,0)--(1.5,-.13);
      \node[font=\scriptsize] at (1.5,-.45) {$ 4 $};
      \draw[black!35,line width=.3pt] (2.5,0)--(2.5,-.13);
      \node[font=\scriptsize] at (2.5,-.45) {$ 5 $};
      \draw[black!35,line width=.3pt] (3.5,0)--(3.5,-.13);
      \node[font=\scriptsize] at (3.5,-.45) {$ 6 $};
      \node[stackblock] at (0.5,0.5) {$(3,3)$};
      \node[stackblock] at (0.5,1.5) {$(3,2)$};
      \node[stackblock] at (0.5,2.5) {$(3,1)$};
      \node[stackblock] at (1.5,0.5) {$(4,1)$};
      \node[stackblock] at (2.5,0.5) {$(5,3)$};
      \node[stackblock] at (2.5,1.5) {$(5,2)$};
      \node[stackblock,marked,newblock] at (3.5,0.5) {$(5,1)$};
    \end{scope}
    \begin{scope}[xshift=11.4cm]
      \node[font=\small] at (2,4.05) {Final stacks};
      \node[inner sep=0pt,font=\scriptsize] at (2,1.7) {
        \renewcommand{\arraystretch}{.85}
        \begin{tabular}{@{}ccc@{}}
          position&letter&endpoint\\ \hline
          \(3\)&\(3\)&\(3\)\\
          \(4\)&\(4\)&\(3\)\\
          \(5\)&\(5\)&\(5\)\\
          \(6\)&\(5\)&\(3\)
        \end{tabular}
      };
    \end{scope}
    \begin{scope}[shift={(3.0cm,-.9cm)}]
      \draw[black!65,fill=black!15,line width=.35pt] (0,0) rectangle (.4,.55);
      \node[anchor=west,font=\small] at (.6,.275) {marked bottom block};
    \end{scope}
    \begin{scope}[shift={(9.2cm,-.9cm)}]
      \draw[blue!70!black,line width=.9pt] (0,0) rectangle (.4,.55);
      \node[anchor=west,font=\small] at (.6,.275) {new block};
    \end{scope}
  \end{tikzpicture}
  \caption{\(\Omega\)-multiparking for \Cref{ex:bundle-labeling-change}.}
  \label{fig:stacked-block-insertion}
\end{figure}

After all pairs have been inserted, let \(\ell\) be the number of
stacks. Since only the equality case adds a block above a bottom
block, all blocks in a stack have the same letter in their labels.
Let \(k_1<\cdots<k_\ell\) be the insertion steps that create bottom
blocks. Define
\[
  b(a,Q):=\ESort(j_{k_1},\ldots,j_{k_\ell}),
\]
and reorder the stacks by the permutation used in this endpoint sorting.
Write \(b=b(a,Q)=(b_1,\ldots,b_\ell)\).
Set
\[
  A_r=\{q:\text{a block in the \(r\)-th stack has label }(b_r,q)\},
\]
and form the set of cells
\begin{equation}\label{eq:rectification-output-selection}
  R=\{C_{q,r}:q\in A_r,\ 1\leq r\leq \ell\}.
\end{equation}
Each pair gives one cell of \(R\), so
\(|R|=|a|\). Each stack gives one row, so there are \(|b|\) rows.
We write \(\Xi_w(a,Q)=(b(a,Q),R)\).
In \Cref{prop:hecke-multicell-bijection}, we will show that \(b\)
has positive endpoints and that \(R\) is a multi-cell selection of
\((P_b^{\mathrm{end}})^\circ\), and prove that this assignment is a
weight-preserving bijection
\begin{equation}\label{eq:path-rectification-bijection-goal}
  \Xi_w:
  \bigsqcup_{a\in\EH_n(w)}\ICSel_a^\times
  \longrightarrow
  \bigsqcup_{b\in\EH_n^+(w)}
    \MCSel((P_b^{\mathrm{end}})^\circ).
\end{equation}

Continuing \Cref{ex:bundle-labeling-change}, the letters on the bottom blocks in creation order are
\((3,5,4,5)\), and the endpoint sequence is \((3,5,3,3)\).
Endpoint sorting orders the stacks as \(2,1,3,4\), giving
\[
  b(a,Q)=(5,3,4,5),\qquad
  (A_1,A_2,A_3,A_4)
    =(\{2,3\},\{1,2,3\},\{1\},\{1\}).
\]
Hence \(\Xi_{124653}(a,Q)=(b,R)\).
The endpoints of \(b\) are \((5,3,3,3)\), so
\(b\in\EH_5^+(w)\) and
\(c^{\mathrm{end}}(b)=(0,0,3,0,1)\).
Both \(Q\) and \(R\) select seven cells and have weight
\(x_1^3x_2^2x_3^2\).
\Cref{fig:multiparking-cell-selection} shows how each reordered stack gives the selected cells in one row.

\begin{figure}
  \centering
  \begin{tikzpicture}[x=1cm,y=1cm,font=\scriptsize]
    \tikzset{pairblock/.style={draw=black!65,line width=.35pt,fill=white,
      minimum width=.72cm,minimum height=.43cm,inner sep=0pt}}
    \node[font=\small] at (1.70,2.55) {Creation order};
    \node at (1.70,2.13) {endpoints \((3,5,3,3)\)};
    \node[font=\small] at (6.19,2.55) {Endpoint order};
    \node at (6.19,2.13) {endpoints \((5,3,3,3)\)};
    \node[font=\small] at (11.05,2.55) {Multi-cell selection};
    \node[pairblock] at (0.360,0.665) {\((3,3)\)};
    \node[pairblock] at (0.360,1.095) {\((3,2)\)};
    \node[pairblock] at (0.360,1.525) {\((3,1)\)};
    \node at (0.360,.15) {\(1\)};
    \node[pairblock] at (1.250,0.665) {\((5,3)\)};
    \node[pairblock] at (1.250,1.095) {\((5,2)\)};
    \node at (1.250,.15) {\(2\)};
    \node[pairblock] at (2.140,0.665) {\((4,1)\)};
    \node at (2.140,.15) {\(3\)};
    \node[pairblock] at (3.030,0.665) {\((5,1)\)};
    \node at (3.030,.15) {\(4\)};
    \node[pairblock] at (4.850,0.665) {\((5,3)\)};
    \node[pairblock] at (4.850,1.095) {\((5,2)\)};
    \node at (4.850,.15) {\(1\)};
    \node[pairblock] at (5.740,0.665) {\((3,3)\)};
    \node[pairblock] at (5.740,1.095) {\((3,2)\)};
    \node[pairblock] at (5.740,1.525) {\((3,1)\)};
    \node at (5.740,.15) {\(2\)};
    \node[pairblock] at (6.630,0.665) {\((4,1)\)};
    \node at (6.630,.15) {\(3\)};
    \node[pairblock] at (7.520,0.665) {\((5,1)\)};
    \node at (7.520,.15) {\(4\)};
    \draw[->,line width=.55pt] (3.65,1.05)--(4.25,1.05);
    \node at (3.95,1.34) {sort};
    \draw[->,line width=.55pt] (8.20,1.05)--(9.10,1.05);
    \node at (8.65,1.34) {form \(R\)};
    \draw[draw=black!35,line width=.3pt,fill=white] (9.80,1.50) rectangle ++(.5,.5);
    \draw[draw=black!35,line width=.3pt,fill=blue!14] (10.30,1.50) rectangle ++(.5,.5);
    \fill[blue!65!black] (10.55,1.75) circle (1.2pt);
    \draw[draw=black!35,line width=.3pt,fill=blue!14] (10.80,1.50) rectangle ++(.5,.5);
    \fill[blue!65!black] (11.05,1.75) circle (1.2pt);
    \draw[draw=black!35,line width=.3pt,fill=white] (11.30,1.50) rectangle ++(.5,.5);
    \draw[draw=black!35,line width=.3pt,fill=white] (11.80,1.50) rectangle ++(.5,.5);
    \node at (9.55,1.75) {\(1\)};
    \draw[draw=black!35,line width=.3pt,fill=blue!14] (9.80,1.00) rectangle ++(.5,.5);
    \fill[blue!65!black] (10.05,1.25) circle (1.2pt);
    \draw[draw=black!35,line width=.3pt,fill=blue!14] (10.30,1.00) rectangle ++(.5,.5);
    \fill[blue!65!black] (10.55,1.25) circle (1.2pt);
    \draw[draw=black!35,line width=.3pt,fill=blue!14] (10.80,1.00) rectangle ++(.5,.5);
    \fill[blue!65!black] (11.05,1.25) circle (1.2pt);
    \node at (9.55,1.25) {\(2\)};
    \draw[draw=black!35,line width=.3pt,fill=blue!14] (9.80,0.50) rectangle ++(.5,.5);
    \fill[blue!65!black] (10.05,0.75) circle (1.2pt);
    \draw[draw=black!35,line width=.3pt,fill=white] (10.30,0.50) rectangle ++(.5,.5);
    \draw[draw=black!35,line width=.3pt,fill=white] (10.80,0.50) rectangle ++(.5,.5);
    \node at (9.55,0.75) {\(3\)};
    \draw[draw=black!35,line width=.3pt,fill=blue!14] (9.80,0.00) rectangle ++(.5,.5);
    \fill[blue!65!black] (10.05,0.25) circle (1.2pt);
    \draw[draw=black!35,line width=.3pt,fill=white] (10.30,0.00) rectangle ++(.5,.5);
    \draw[draw=black!35,line width=.3pt,fill=white] (10.80,0.00) rectangle ++(.5,.5);
    \node at (9.55,0.25) {\(4\)};
    \node at (10.05,2.13) {\(1\)};
    \node at (10.55,2.13) {\(2\)};
    \node at (11.05,2.13) {\(3\)};
    \node at (11.55,2.13) {\(4\)};
    \node at (12.05,2.13) {\(5\)};
    \node at (9.55,2.13) {\(r\backslash q\)};
    \draw[line width=.8pt] (12.60,2.00)--(12.30,2.00)--(12.30,1.50)--(11.30,1.50)--(11.30,0)--(9.80,0)--(9.80,-.30);
    \node at (11.05,-.38) {\((P_b^{\mathrm{end}})^\circ\)};
  \end{tikzpicture}
  \caption{Endpoint sorting of the stacks and the resulting multi-cell selection.}
  \label{fig:multiparking-cell-selection}
\end{figure}
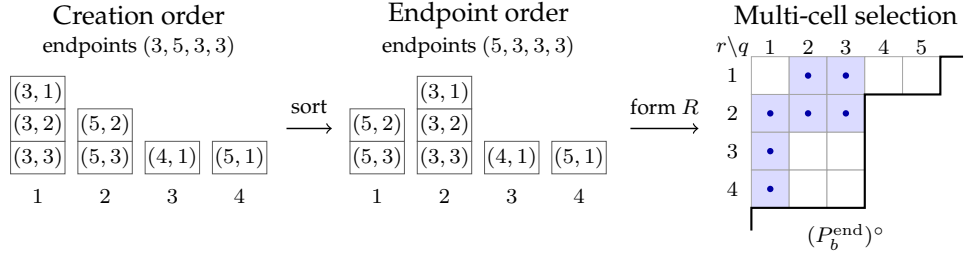

\begin{prop}\label{lem:stacking-bottom-blocks}
For \(w\in S_{n+1}\), let \(a\in\EH_n(w)\) and
\(Q\in\ICSel_a^\times\).
Let \((j_{k_1},q_{k_1}),\ldots,(j_{k_\ell},q_{k_\ell})\) be the
subsequence of \(X(a,Q)=\{(j_1,q_1),\ldots,(j_m,q_m)\}\) that creates bottom blocks in its \(\Omega\)-multiparking. Then
\[
  \operatorname{Dem}(j_1,\ldots,j_m)
  =\operatorname{Dem}(j_{k_1},\ldots,j_{k_\ell})
  =\operatorname{Dem}(b(a,Q)).
\]
\end{prop}
\begin{proof}
Write \(*\) for the \(0\)-Hecke product, and set
\(u_k=\operatorname{Dem}(j_1,\ldots,j_k)\), with
\(u_0=\mathrm{id}\).
Suppose the \(k\)-th insertion does not create a bottom block.
The new block is then placed above a marked bottom block created
at a step \(h<k\), and the insertion rule gives \(j_h=j_k\).
This bottom block remains marked from step \(h\) through step \(k\).
Since \(I_{j_k}\) is occupied after step \(h\), an intervening
insertion with letter \(j_k-1\) or \(j_k+1\) would create a new
bottom block in the component containing \(I_{j_k}\) and remove
that mark. Hence \(|j_t-j_k|\ne1\) for \(h<t<k\).
Thus \(s_{j_k}\) commutes with every intervening factor, and the
\(0\)-Hecke relations give
\[
  u_k
  =u_{h-1}*s_{j_h}*s_{j_k}*
    s_{j_{h+1}}*\cdots*s_{j_{k-1}}
  =u_{k-1}.
\]
At each step \(k\) that creates a bottom block, we have
\(u_k=u_{k-1}*s_{j_k}\).
Starting from \(u_0=\mathrm{id}\), these relations give
\(u_m=\operatorname{Dem}(j_{k_1},\ldots,j_{k_\ell})\),
which proves the first equality.
The second follows from the definition of \(b(a,Q)\) and
\Cref{lem:endpoint-sort-well-defined}.
\end{proof}

\begin{lem}\label{prop:hecke-multicell-bijection}
For every \(w\in S_{n+1}\), the map \(\Xi_w\) is a
weight-preserving bijection.
\end{lem}

\begin{proof}
Fix a nonempty \(a\in\EH_n(w)\) and \(Q\in\ICSel_a^\times\).
Put \(m=|a|\), and write
\(X(a,Q)=\{(j_1,q_1),\ldots,(j_m,q_m)\}\),
where the pairs are listed in decreasing order of \(q\), resolving
ties by increasing \(j\).
Let \(k_1<\cdots<k_\ell\) be the steps that create bottom blocks in
the \(\Omega\)-multiparking of this list.
The letters on the bottom blocks, in creation order, are
\(j_{k_1},\ldots,j_{k_\ell}\).
Set \(b=b(a,Q)\), and let \(R\) be the set of cells defined in
\eqref{eq:rectification-output-selection}.
We first show that \(b\in\EH_n^+(w)\) and that \(R\) is a multi-cell
selection of \((P_b^{\mathrm{end}})^\circ\).

Since the word read from \(X(a,Q)\) has endpoint sort \(a\),
\Cref{lem:stacking-bottom-blocks} gives
\(\operatorname{Dem}(b)=w\).
Since \(b\) is endpoint-decreasing, its letters lie in \([n]\), and
\(|b|\leq |a|\leq\binom{n+1}{2}\), we have \(b\in\EH_n(w)\).

Let \(B_r\) denote the bottom block of the \(r\)-th stack after
endpoint sorting, so its letter is \(b_r\).
Whenever the insertion of \(B_s\) removes the mark on \(B_r\),
the \(\Omega\)-multiparking rule gives \(b_r\ne b_s\).
By construction, ordinary \(\Omega\)-parking of
\((j_{k_1},\ldots,j_{k_\ell})\) has the same occupied and marked
bottom positions as \(\Omega\)-multiparking of \(X(a,Q)\) at each
bottom-block insertion.
Together with \eqref{eq:linear-extension-induction-data}, this gives
\(b_r\ne b_s\) whenever \(r\trl_b s\).
Equal adjacent letters in ordinary \(\Omega\)-parking would give
\(r\trl_b r+1\), so the adjacent letters of \(b\) are distinct.

We next show that \(R\) satisfies the inequalities in
\eqref{eq:grove-multicell-compatibility}.
Let \(S_r\) denote the \(r\)-th south step of the word path \(P_b\).
Suppose the insertion of \(B_s\) removes the mark on \(B_r\).
No further blocks can then be added above \(B_r\).
Thus every block in the stack of \(B_r\) is inserted before \(B_s\). Thus for all \(p\in A_r\) and \(q\in A_s\), we have \(q\leq p\) with strict inequality if \(b_r>b_s\).
Together with \eqref{eq:linear-extension-induction-data} and
\Cref{lem:block-path-successors}, these inequalities give
\begin{equation}\label{eq:insertion-multicell-inequalities}
\begin{aligned}
  \max A_s&\leq\min A_r&&\text{if }S_r\rav S_s,\\
  \max A_s&<\min A_r&&\text{if }S_r\rad S_s.
\end{aligned}
\end{equation}

By \eqref{eq:block-rule} and \Cref{lem:block-path-successors},
for every \(s\in[|b|]\),
\begin{equation}\label{eq:effective-flag-recursion}
  e_s(b)=\min\bigl(
    \{b_s\}\cup\{e_r(b):S_r\rav S_s\}
    \cup\{e_r(b)-1:S_r\rad S_s\}\bigr).
\end{equation}
Every \(q\in A_s\) satisfies \(1\leq q\leq b_s\).
For a row with no incoming successor arrows, this gives
\(e_s(b)=b_s\).
Inductively, \eqref{eq:insertion-multicell-inequalities} bounds
\(\max A_s\) by \(e_r(b)\) for a vertical successor and by
\(e_r(b)-1\) for a diagonal successor.
Hence \(\max A_s\) is bounded by every entry in the minimum in
\eqref{eq:effective-flag-recursion},
and \(1\leq\max A_s\leq e_s(b)\).
This proves \(b\in\EH_n^+(w)\) and
\(R\in\MCSel((P_b^{\mathrm{end}})^\circ)\).

Each input pair \((j,q)\) gives one output cell in column \(q\),
with the same weight \(x_q\), so \(\Xi_w\) preserves weights.
Moreover, \((b,R)\) determines
\[
  X(a,Q)=\{(b_r,q):q\in A_r,\ 1\leq r\leq |b|\}.
\]
By \Cref{lem:hecke-path-bijection}, \(X(a,Q)\) determines the input
\((a,Q)\) uniquely. Thus \(\Xi_w\) is injective.

For surjectivity, fix \(b\in\EH_n^+(w)\) and
\(R\in\MCSel((P_b^{\mathrm{end}})^\circ)\).
Let \(S_r\) denote the \(r\)-th south step of
\((P_b^{\mathrm{end}})^\circ\).
Set \(A_r=\{q:C_{q,r}\in R\}\).
For each row \(r\), form a stack of blocks labeled
\((b_r,q)\), \(q\in A_r\), with larger \(q\) below smaller \(q\).
Let \(B_r\) denote its bottom block, labeled \((b_r,\max A_r)\).
List these stacks in row order, as required after endpoint
sorting. The set of all block labels is
\[
  X(b,R)=\{(b_r,q):q\in A_r,\ 1\leq r\leq |b|\}.
\]
We first show that different rows contribute distinct pairs to
\(X(b,R)\).
For every \(q\in A_r\), we have
\(1\leq q\leq e_r(b)\leq b_r\leq n\).
\Cref{lem:block-path-successors} and \(b\in\EH_n^+(w)\) give
\(b_r<b_s\) whenever \(S_r\rav S_s\).
If \(r<s\) and \(b_r=b_s\), then
\eqref{eq:incomparable-supports-commute} gives \(r\prec_b s\).
A successor chain from \(r\) to \(s\) must contain a diagonal
arrow, since the letters strictly increase along vertical arrows.
Applying \eqref{eq:grove-multicell-compatibility} along this chain
gives \(\max A_s<\min A_r\).
Thus different rows contribute distinct pairs, as claimed.

Read \(X(b,R)\) in decreasing order of \(q\), resolving ties by
increasing \(j\), and write the resulting list as
\((j_1,q_1),\ldots,(j_m,q_m)\), where \(m=|R|\).
For \(r\trl_b s\), all pairs from row \(r\) precede all pairs from
row \(s\): the column inequality is strict for a diagonal successor,
and an equal-column tie for a vertical successor is resolved by
\(b_r<b_s\).
Thus the bottom labels \((b_r,\max A_r)\) appear in an order
of rows that is a linear extension of \(\mathcal H(b)\).
Place the bottom blocks \(B_r\) by ordinary \(\Omega\)-parking
of the letters \(b_r\) in this order.
By \eqref{eq:linear-extension-induction-data}, insertion of \(B_s\)
removes the mark on \(B_r\) exactly when \(r\trl_b s\).
\Cref{lem:block-path-successors}, together with \(b_r<b_s\) for
\(S_r\rav S_s\), gives \(b_r\ne b_s\) on every such relation.
Apply \(\Omega\)-multiparking to the pairs \((b_r,\max A_r)\)
for \(r\in[|b|]\), in the same order of \( X(b,R) \).
The equality case of \eqref{eq:stacked-block-rule} therefore never
occurs, so the bottom blocks and marked positions are the same
as in ordinary parking.

All pairs from row \(r\) precede \((b_s,\max A_s)\) whenever
\(r\trl_b s\).
Each pair \((b_r,q)\) with \(q<\max A_r\) can therefore be
inserted at its prescribed place in the reading, above the
still marked \(B_r\).
These additions do not change the bottom positions or marks.
By \Cref{prop:endpoint-sort-linear-extensions}, endpoint sorting
returns these stacks to their original row order.

Finally, apply ordinary \(\Omega\)-parking and endpoint sorting
to the full word \((j_1,\ldots,j_m)\), carrying each \(q_i\)
with its letter \(j_i\).
Let \(\sigma\) be the sorting permutation, and set
\[
  a=(j_{\sigma(1)},\ldots,j_{\sigma(m)}),\qquad
  Q=\{C_{q_{\sigma(t)},t}:t\in[m]\}.
\]
The distinct pairs in \(X(b,R)\) determine a subword of
\(\omega_n^+\), so the construction in
\Cref{lem:hecke-path-bijection} gives \(Q\in\ICSel_a^\times\)
and \(X(a,Q)=X(b,R)\).
The stack construction above therefore gives \(b(a,Q)=b\) and
the output cell selection \(R\).
By \Cref{lem:stacking-bottom-blocks},
\(\operatorname{Dem}(a)=\operatorname{Dem}(b)=w\).
Thus \(a\in\EH_n(w)\) and \(\Xi_w(a,Q)=(b,R)\), proving surjectivity.
\end{proof}

\begin{proof}[Proof of \Cref{thm:grothendieck-grove-labeling-positivity}]
The lattice path formula in \Cref{thm:grothendieck-lattice-path}
uses Hecke words of \(w^{-1}\), so we apply \(\Xi_{w^{-1}}\).
Each input cell gives one output cell, hence
\(|R|=|a|\) whenever \(\Xi_{w^{-1}}(a,Q)=(b,R)\).
Reindexing the sum by this bijection and using
\Cref{thm:grove-multicell-model}, we obtain
\[
\begin{aligned}
  G_w^{(\beta)}(\vec x)
  &=\sum_{a\in\EH_n(w^{-1})}
      \ \sum_{Q\in\ICSel_a^\times}
        \beta^{|a|-\ell(w)}
        \prod_{r=1}^{|a|}x_{\col_r(Q)}\\
  &=\sum_{b\in\EH_n^+(w^{-1})}
      \beta^{|b|-\ell(w)}
      \sum_{R\in\MCSel((P_b^{\mathrm{end}})^\circ)}
        \beta^{|R|-|b|}
        \prod_{C_{q,r}\in R}x_q\\
  &=\sum_{b\in\EH_n^+(w^{-1})}
      \beta^{|b|-\ell(w)}
      \widetilde\PP_{c^{\mathrm{end}}(b)}^{(\beta)}(\vec x).
\end{aligned}
\]
Since \(|c^{\mathrm{end}}(b)|=|b|\geq\ell(w)\), collecting the terms
with \(c^{\mathrm{end}}(b)=c\) gives
\eqref{eq:grothendieck-grove-core-count} and the vanishing for
\(|c|<\ell(w)\).
\end{proof}

\subsection{From Schubert to forest polynomials}
\label{subsec:ordinary-schubert-forest-expansion}

For \(w\in S_\infty\setminus\{\mathrm{id}\}\), set
\[
  \ERed^+(w)
  =\{a\in\ERed(w):e_{\ell(w)}(a)\geq1\},
\]
and set \(\ERed^+(\mathrm{id})=\{\varnothing\}\).
The words in \(\ERed^+(w)\) have positive endpoints, so their endpoint
codes \(c^{\mathrm{end}}(a)\) are defined as in
\Cref{subsec:grothendieck-grove-expansion}.

The following restates \cite[Theorem~1.4]{NT2024}, indexing each
\(\Omega\)-class by its unique representative in \(\ERed(w^{-1})\),
as in \eqref{eq:endpoint-sort-omega-classes}. It is also a
specialization of the grove expansion.
Choose \(n\geq1\) such that \(w\in S_{n+1}\). A Hecke word of
\(w^{-1}\) has length at least \(\ell(w)\), with equality if and
only if it is reduced. Since reduced words have distinct adjacent
letters, setting \(\beta=0\) in
\eqref{eq:grothendieck-grove-labeling-expansion} retains precisely
the words in \(\ERed^+(w^{-1})\). Using
\(G_w^{(0)}=\mathfrak S_w\) and
\(\widetilde\PP_c^{(0)}=\PP_c\), together with the linear
independence of forest polynomials, we obtain the following theorem.

\begin{thm}\label{thm:ordinary-forest-expansion}
For every \(w\in S_\infty\),
\[
  \mathfrak S_w(\vec{x})
  =\sum_{a\in\ERed^+(w^{-1})}
    \PP_{c^{\mathrm{end}}(a)}(\vec{x}).
\]
Consequently, for every \(c\in\Codes\),
\[
  [\PP_c]\mathfrak S_w
  =\bigl|\{a\in\ERed^+(w^{-1}):c^{\mathrm{end}}(a)=c\}\bigr|.
\]
In particular, Schubert polynomials expand positively in forest
polynomials.
\end{thm}

By \eqref{eq:omega-insertion-data}, the forest of the class
represented by \(a\) has code \(c^{\mathrm{end}}(a)\).
The omitted \(\Omega\)-classes have a nonpositive final endpoint.
By \eqref{eq:omega-positive-support}, their forests lie outside
\(\mathsf{IndFor}_+\), so their forest polynomials vanish
\cite[Definition~3.4]{NT2024}.

\begin{exam}\label{ex:ordinary-forest-expansion}
Let \(w=246135\in S_6\), so \(w^{-1}=415263\) and \(\ell(w)=6\).
The reduced word \(a=(3,4,2,5,3,1)\) has endpoint sequence
\((3,3,2,2,1,0)\), so
\(a\in\ERed(w^{-1})\setminus\ERed^+(w^{-1})\).
The words in \(\ERed^+(w^{-1})\), with their endpoints and codes, are
\[
\begin{array}{c|c|c}
  a & (e_1(a),\ldots,e_6(a)) & c^{\mathrm{end}}(a)\\ \hline
  (3,2,4,1,3,5)&(3,2,2,1,1,1)&(3,2,1)\\
  (3,2,4,5,1,3)&(3,2,2,2,1,1)&(2,3,1)\\
  (3,4,2,1,3,5)&(3,3,2,1,1,1)&(3,1,2)\\
  (3,4,2,3,1,5)&(3,3,2,2,1,1)&(2,2,2)\\
  (3,4,2,3,5,1)&(3,3,2,2,2,1)&(1,3,2)\\
  (3,4,2,5,1,3)&(3,3,2,2,1,1)&(2,2,2)\\
  (3,4,5,2,1,3)&(3,3,3,2,1,1)&(2,1,3)\\
  (3,4,5,2,3,1)&(3,3,3,2,2,1)&(1,2,3).
\end{array}
\]
Thus \Cref{thm:ordinary-forest-expansion} gives
\[
  \mathfrak S_{246135}
  =\PP_{(3,2,1)}+\PP_{(2,3,1)}+\PP_{(3,1,2)}
      +2\PP_{(2,2,2)}
     +\PP_{(1,3,2)}+\PP_{(2,1,3)}+\PP_{(1,2,3)}.
\]
The coefficient \(2\) counts the two words with endpoint code
\((2,2,2)\).
\end{exam}

\end{document}